\documentclass[12pt]{article}

\usepackage{amsmath,amsfonts}
\usepackage{amssymb}
\usepackage{fullpage}
\usepackage{amsthm} 
\usepackage{graphicx}
\usepackage{placeins}
\usepackage{latexsym}
\usepackage{color}

\usepackage{hyperref}

\newtheorem{thm}{Theorem}[section]
\newtheorem{prop}{Proposition}[section]

\newtheorem{rmk}{Remark}[section]
\newtheorem{lem}{Lemma}[section]

\newtheorem{cor}{Corollary}[section]\label{key}

\makeatletter
\newcommand{\rom}[1]{\expandafter\@slowromancap\romannumeral #1@}
\makeatother
\newcommand{\ur}{u_{r}}
\newcommand{\uthe}{u_{\theta}}

\newcommand{\uphi}{u_{\phi}}

\newcommand{\dive}{\mbox{div}}

\newcommand{\ctthe}{\cot\theta}
\newcommand{\sthe}{\sin\theta}
\newcommand{\cthe}{\cos\theta}

\newcommand{\er}{e_r}
\newcommand{\ethe}{e_{\theta}}
\newcommand{\ephi}{e_{\phi}}

\begin{document}

\title{Homogeneous solutions of stationary Navier-Stokes equations with isolated singularities on the unit sphere. I. One singularity}
\author{Li Li\footnote{Department of Mathematics, Harbin Institute of Technology, Harbin 150080, China. Email: lilihit@126.com}, 
YanYan Li\footnote{Department of Mathematics, Rutgers University, 110 Frelinghuysen Road, Piscataway, NJ 08854, USA. Email: yyli@math.rutgers.edu}, 
Xukai Yan\footnote{Department of Mathematics, Rutgers University, 110 Frelinghuysen Road, Piscataway, NJ 08854, USA. Email: xkyan@math.rutgers.edu}}
\date{}
\maketitle

\abstract{We classify all $(-1)-$homogeneous axisymmetric no swirl solutions of incompressible stationary Navier-Stokes equations in three dimension which are smooth  on the unit sphere minus the south pole, parameterize them as a two dimensional surface with boundary, and analyze their pressure profiles near  the north pole. Then we prove that there is a curve of $(-1)-$homogeneous axisymmetric solutions with nonzero swirl, having the same smoothness property, emanating from every point of the interior and one part of the boundary of the solution surface. Moreover we prove that  there is no such curve of solutions for any point on the other part of the boundary.   We also  establish asymptotic expansions for every  (-1)-homogeneous axisymmetric solutions  in a neighborhood of the singular  point on the unit sphere. }

\tableofcontents

\setcounter{section}{0}

\section{Introduction} 
Consider $(-1)$-homogeneous solutions of incompressible stationary Navier-Stokes Equations (NSE) in $\mathbb{R}^3$:   
\begin{equation}\label{NS}
\left\{
\begin{split}
	& -\triangle u + u\cdot \nabla u +\nabla p = 0, \\
	& \dive\textrm{ } u=0.
\end{split}
\right. 
\end{equation}

The NSE is invariant under the scaling $u(x)\rightarrow \lambda u(\lambda x)$. It is natural to study $(-1)$-homogeneous solutions, namely, solutions which are invariant under this scaling. 

In 1944, L.D. Landau discovered a 3-parameter family of explicit $(-1)$-homogeneous
solutions of stationary NSE in $C^\infty(\mathbb{R}^3\setminus\{0\})$. They are axisymmetric with no-swirl.
He arrived at these solutions, now called  {\it Landau solutions},
using the following ansatz: looking for solutions which are 
axisymmetric, no-swirl, and with two vanishing diagonal components of the tensor of momentum
flow density. Tian and Xin proved in \cite{TianXin} that all $(-1)$-homogeneous, axisymmetric nonzero solutions 
of the stationary NSE (\ref{NS}) in $C^2(\mathbb{R}^3\setminus\{0\})$
are Landau solutions.
\bigskip
\v{S}ver\'{a}k established the following result in 2006:\\
\textbf{Theorem A (\cite{Sverak})} \emph{All (-1)-homogeneous nonzero solutions of (\ref{NS}) in $C^2(\mathbb{R}^3\setminus\{0\})$ are Landau solutions.}\\

He also proved in the same paper that there is no nonzero (-1)-homogeneous solution of the stationary NSE in $C^{2}(\mathbb{R}^n\setminus\{0\})$ for $n\ge 4$. In dimension $n=2$, he characterized all such solutions satisfying a zero flux condition.\\
 
Starting from this paper, we analyze (-1)-homogeneous solutions in $\mathbb{R}^n$ with finite singularities on $\mathbb{S}^{n-1}$, as well as (-1)-homogeneous solutions in half space $\mathbb{R}^n_+$ with finite singularities on $\mathbb{S}^n_+$ and zero velocity on $\partial \mathbb{R}^n_+$. In this paper we focus on axisymmetric solutions of the problem in $\mathbb{R}^3$ which have exactly one singularity on the unit sphere $\mathbb{S}^2$.

In polar coordinates $(r,\theta,\phi)$, where $r$ is the radial distance from the origin, $\theta$ is the angle between the radius vector and the positive $x_3$-axis, and $\phi$ is the meridian angle about the $x_3$-axis. A vector field $u$ can be written as
\begin{equation}\label{u_polar}
	u = u_r \er + u_\theta \ethe + u_\phi \ephi, 
\end{equation}
where 
\begin{equation*}
	\er = \left(
	\begin{matrix}
		\sthe\cos\phi \\
		\sthe\sin\phi \\
		\cthe
	\end{matrix} \right),  \hspace{0.5cm}
	\ethe = \left(
	\begin{matrix}
		\cthe\cos\phi  \\
		\cthe\sin\phi   \\
		-\sthe	
	\end{matrix} \right), \hspace{0.5cm}
	\ephi = \left(
	\begin{matrix}
		-\sin\phi \\  \cos\phi \\ 0
	\end{matrix} \right).  
\end{equation*}

A vector field $u$ is called axisymmetric if $u_r$, $u_{\theta}$ and $u_{\phi}$ depend only on $r$ and $\theta$, and is called {\it no-swirl} if $u_{\phi}=0$.

 If $u$ is (-1)-homogeneous and $p$ is (-2)-homogeneous, system (1) is a system of partial differential equations of $u|_{\mathbb{S}^2}$ and $p|_{\mathbb{S}^2}$ on $\mathbb{S}^2$. For a (-1)-homogeneous and axisymmetric solution, $(u,p)$ depends only on $\theta$ in polar coordinates, and the system on $\mathbb{S}^2$ takes the form
\begin{equation}\label{eq_homo}
	\left\{
	\begin{split}
		& \frac{d^2 \ur}{d\theta^2} + (\ctthe - \uthe) \frac{d \ur}{d\theta} + \ur^2 +\uthe^2 +\uphi^2 +2p = 0; \\
		& \frac{d}{d\theta} (\frac{1}{2}\uthe^2 - \ur + p) =  \cot\theta \uphi^2; \\
		& \frac{d}{d \theta}(\frac{d \uphi}{d \theta} + \cot\theta \uphi) - \uthe (\frac{d \uphi}{d \theta} + \cot\theta \uphi) =0; \\
		& \ur + \frac{d \uthe}{d \theta} + \ctthe \uthe = 0. \quad \quad \mbox{(divergence free condition)}\\    
	\end{split}
	\right.
\end{equation}
Since $p$ is determined by $u$ and its derivatives up to second order, in view of the first line of (\ref{eq_homo}), we often say that $u$ is a solution of (\ref{NS}) without mentioning $p$.

 By the divergence free condition in (\ref{eq_homo}), the radial component $u_r$ of the velocity $u$ is determined by $u_{\theta}$ and its first derivative.

Our first result classifies all (-1)-homogeneous axisymmetric no-swirl solutions $u$ of (\ref{NS}) in $C^2(\mathbb{S}^2\setminus\{S\})$, where $S$ denotes the south pole of $\mathbb{S}^2$. In this case $u_{\phi}=0$, and $u_{r}$, $p$ can be determined by $u_{\theta}$ and its derivatives. So we only need to solve $u_{\theta}$. 

We introduce the following subsets of $\mathbb{R}^2$
\begin{equation}\label{eqJ}
  \begin{split}
  &  J_1:= \{(\tau,\sigma)\mid \tau<2, \sigma <\frac{1}{4}(4-\tau)\},\\
   & J_2:=\{(\tau,\sigma)\mid \tau=2, \sigma< \frac{1}{2}\},\\
   & J_3:=\{(\tau,\sigma)\mid \tau\ge  2, \sigma =\frac{\tau}{4}\},
  \end{split}
\end{equation}
and $J:=J_1\cup J_2\cup J_3$.

\begin{thm}\label{thm1_1}
    For every $ (\tau, \sigma)\in J$, there exists a unique $u_{\theta}:=(u_{\theta})_{\tau,\sigma}\in C^{\infty}\left(\mathbb{S}^2\setminus\{S\}\right )$ such that the corresponding $(u,p)$ satisfies (\ref{NS}) on $\mathbb{S}^2\setminus\{S\}$, and
\begin{equation}\label{eq_thm1_1}
   \lim_{\theta\to \pi^-} u_{\theta}\sin\theta=\tau,\quad  \lim_{\theta\to 0^+}\frac{u_{\theta}}{\sin\theta}=\sigma.
\end{equation}
Moreover, these are all the axisymmetric no-swirl solutions in $C^2(\mathbb{S}^2\setminus\{S\})$.

The solutions $u_{\theta}$ are explicitly given by, with $b := |1-\frac{\tau}{2}|$,
\begin{equation}\label{eq_temp11}
	u_\theta = 
	\left\{
	\begin{array}{ll}
		 \frac{\displaystyle 1-\cos\theta}{\displaystyle \sin\theta}\left(1- b - \frac{\displaystyle 2 b(1-2\sigma-b)}{\displaystyle(1-2\sigma+b) (\frac{1+\cos\theta}{2})^{-b} + 2\sigma-1 + b} \right), & (\tau, \sigma)\in J_1; \\
		\frac{\displaystyle 1-\cos\theta}{\displaystyle \sin\theta} \left( 1+ \frac{\displaystyle 2 (1-2\sigma)  }{ \displaystyle (1-2\sigma) \ln \frac{1+\cos\theta}{2}  - 2 } \right), & (\tau, \sigma)\in J_2; \\
            \frac{\displaystyle (1+b)(1-\cos\theta)}{\displaystyle \sin\theta}, & (\tau, \sigma)\in J_3. 
	\end{array}
	\right.  
	\end{equation}

\end{thm}

By saying $(u,p)$ corresponding to $u_{\theta}$ we mean that on $\mathbb{S}^2$, $u_r$ is determined by $u_{\theta}$ through the divergence free condition in (\ref{eq_homo}), and $p$ is determined by the first line in (\ref{eq_homo}) using $u_{\phi}=0$, and then $u$ and $p$ are extended respectively to (-1)-homogeneous and (-2)-homogeneous functions. Also, we often simply say that $u$ satisfies (\ref{NS}) instead of saying that $(u,p)$ satisfies (\ref{NS}).\\

We use $(u_{\tau, \sigma}, p_{\tau, \sigma})$ to denote the vector-valued function corresponding to $(u_{\theta})_{\tau, \sigma}$.\\

In (\ref{eq_temp11}), $\{(u_{\theta})_{\tau,\sigma}| \tau=0, \sigma\in (-\infty,0)\cup (0,1)\}$ are Landau solutions. They can also be rewritten as 
\[
    u_{\theta}=\frac{2\sin\theta}{\lambda+\cos\theta}, \quad |\lambda|>1.
\]

$\{(u_{\theta})_{\tau,\sigma}|(\tau,\sigma)\in J\}$ is a 2-parameter family of axisymmetric no swirl solutions of (\ref{NS}) in $C^{2}(\mathbb{S}^2\setminus\{S\})$. In the following theorem, we prove the existence of a curve of axisymmetric solutions with nonzero swirl in $C^{2}(\mathbb{S}^2\setminus\{S\})$ emanating from $(u_{\theta})_{\tau,\sigma}$ for each $(\tau,\sigma)\in J_1\cup J_2\cup \{J_3\cap\{2\le \tau<3\}\}$. We also prove the nonexistence of such solutions for $(\tau,\sigma)\in J_3\cap\{\tau>3\}$.

Define
\begin{equation*}
   a_{\tau,\sigma}(\theta)=-\int_{\frac{\pi}{2}}^{\theta}(2\cot t+(u_{\theta})_{\tau,\sigma} )dt, \quad     b_{\tau,\sigma}(\theta)=-\int_{\frac{\pi}{2}}^{\theta}(u_{\theta})_{\tau,\sigma}dt,
\end{equation*}
and
\begin{equation*}
	v_{\tau,\sigma}^{1} = \left(
	\begin{matrix}
		\frac{1}{\sin\theta}e^{-a_{\tau,\sigma}(\theta)}\frac{d a_{\tau,\sigma}(\theta)}{d\theta} \\ \frac{1}{\sin\theta}e^{-a_{\tau,\sigma}(\theta)}   \\  0
	\end{matrix}\right),
	\quad 
	v_{\tau,\sigma}^{2} = \left(
	\begin{matrix}
		0\\0  \\ \frac{1}{\sin\theta}\int_{0}^{\theta}e^{-b_{\tau,\sigma}(t)}\sin t dt 
	\end{matrix}\right), 
      \quad
       v_{\tau,\sigma}^3=\left(
	\begin{matrix}
		0\\0  \\ \frac{1}{\sin\theta}
	\end{matrix}\right). 
\end{equation*}

\begin{thm}\label{thm1_2}
Let $K$ be a compact subset of one of the four sets $J_1$, $J_2$, $J_3\cap\{2< \tau<3\}$ and $J_3\cap \{\tau=2\}$, then there exist  $\delta=\delta(K)>0$, and $(u,p)\in C^{\infty}(K\times (-\delta,\delta)\times (\mathbb{S}^2\setminus\{S\}))$ such that for every $(\tau,\sigma, \beta)\in K\times(-\delta,\delta)$, $(u,p)(\tau,\sigma,\beta; \cdot)\in C^{\infty}(\mathbb{S}^2\setminus\{S\})$ satisfies (\ref{NS}) in $\mathbb{R}^3\setminus\{(0,0,x_3)|x_3\le 0\}$, with nonzero swirl if $\beta\ne 0$, and $||\left(\sin\frac{\theta+\pi}{2}\right)(u(\tau, \sigma,\beta)-u_{\tau, \sigma})||_{L^{\infty}(\mathbb{S}^2\setminus\{S\})}\to 0$ as $\beta\to 0$. Moreover, $ \frac{\partial }{\partial \beta}u(\tau, \sigma,\beta)|_{\beta=0}=v_{\tau,\sigma}^2 $.

 On the other hand, for $(\tau,\sigma)\in J_3\cap\{\tau>3\}$, there does not exist any sequence of solutions $\{u^i\}$ of (\ref{NS}) in $C^{\infty}(\mathbb{S}^2\setminus\{S\})$, with nonzero swirl, such that $||\left(\sin\frac{\theta+\pi}{2}\right)(u^i-u_{\tau, \sigma})||_{L^{\infty}(\mathbb{S}^2\setminus\{S\})}\to 0$ as $i\to \infty$. 
\end{thm}

In the above theorem, $(u,p)\in C^{\infty}(\mathbb{S}^2\setminus\{S\})$ is understood to have been extended to $\mathbb{R}^3\setminus\{(0,0,x_3|x_3 \leq 0)\}$ so that $u$ is (-1)-homogeneous and $p$ is (-2)-homogeneous. We use this convention throughout the paper unless otherwise stated.\\

\begin{rmk}
  As far as we know, all previously known (-1)-homogeneous solutions $u\in C^{\infty}(\mathbb{S}^2\setminus\{S\})\setminus C^{\infty}(\mathbb{S}^2)$ 
  of (\ref{NS}) satisfying $\limsup_{y\to S} \emph{dist} (y,S)^Nu(y)<\infty$ for some $N>0$ are axisymmetric with no swirl. The existence of such solutions with nonzero swirl are given by Theorem \ref{thm1_2}. A more detailed and stronger version of Theorem \ref{thm1_2}, including a uniqueness result, is given by Theorem \ref{thm4_1}, Theorem \ref{thm4_2} and Theorem \ref{thm4_3} in Section 4.  
\end{rmk}

In this paper we work with new functions and a different variable:
\begin{equation}\label{Va}
	x:=\cthe, \quad U_r := u_r \sthe, \quad U_\theta := u_\theta \sthe, \quad U_\phi:= u_\phi \sthe. 
\end{equation}

In the variable $x$, $x=1$ and $-1$ correspond to the north and south pole $N$ and $S$ of $\mathbb{S}^2$ respectively, while $-1<x<1$ corresponds to $\mathbb{S}^2\setminus\{S,N\}$. We will use " $'$ " to denote differentiation in $x$.\\

Our next two theorems are on the asymptotic behavior of a solution $u$ in a punctured ball $B_{\delta}(S)\setminus\{S\}$ of $\mathbb{S}^2$,  $\delta>0$. 

In the next two theorems, we will state that $U=(U_{\theta}, U_{\phi})$ is a solution of (\ref{NS}),  meaning  that the $u$ determined by $U$ through (\ref{Va}) and (\ref{u_polar}), extended as a (-1)-homogeneous function, satisfies (\ref{NS}).

\begin{thm}\label{thm1_5}
   For $\delta>0$, let $U_{\theta}\in C^1(-1,-1+\delta]$, $U_{\phi}\in C^2(-1,-1+\delta]$, and $U=(U_{\theta}, U_{\phi})$ be an axisymmetric solution of  (\ref{NS}). Then\\
(i)  $U_{\theta}(-1):=\lim_{x\to -1^+}U_{\theta}(x)$ exists and is finite.\\
(ii) $\lim_{x\to -1^+}(1+x)U'_{\theta}(x)=0$.\\
(iii) If $U_{\theta}(-1)<2$ and $U_{\theta}(-1)\ne 0$, denote $\alpha_0=1-\frac{U_{\theta}(-1)}{2}$, then there exist some constants $a_1,a_2$ such that for every $\epsilon>0$, 
\[
   U_{\theta}(x)=U_{\theta}(-1)+a_1(1+x)^{\alpha_0}+a_2(1+x)+O((1+x)^{2\alpha_0-\epsilon})+O((1+x)^{2-\epsilon}).
\] 
If $U_{\theta}(-1)=0$, then there exist some constants $a_1,a_2$ such that  for every $\epsilon>0$, 
\begin{equation*}
  U_{\theta}(x)=a_1(1+x)\ln (1+x)+a_2(1+x)+O((1+x)^{2-\epsilon}).
\end{equation*}
If $U_{\theta}(-1)=2$, then, for every $\epsilon>0$,  either
\begin{equation*}
   U_{\theta}(x)=2+\frac{4}{\ln(1+x)}+O((\ln (1+x))^{-2+\epsilon}),
\end{equation*}
or 
\begin{equation*}
   U_{\theta}(x)=2+O((1+x)^{1-\epsilon}).
\end{equation*}
If $2<U_{\theta}(-1)<3$,  then there exist constants $a_1, a_2$ such that for every $\epsilon>0$, 
\begin{equation*}
   U_{\theta}(x)=U_{\theta}(-1)+a_1(1+x)^{3-U_{\theta}(-1)}+a_2(1+x)+O((1+x)^{2(3-U_{\theta}(-1))-\epsilon}). 
\end{equation*}
\end{thm}

Recall that we denote $\alpha_0=1-\frac{U_{\theta}(-1)}{2}$.
\begin{thm}\label{thm1_6}
 For $\delta>0$, let $U_{\theta}\in C^1(-1,-1+\delta)$,  $U_{\phi}\in C^2(-1,-1+\delta)$, and $U=(U_{\theta}, U_{\phi})$ be an axisymmetric solution of  (\ref{NS}). Then \\
(i) If  $U_{\theta}(-1)<2$, then $U_{\phi}(-1)$ exists and is finite, and there exist some constants $b_1,b_2,b_3$ such that 
\[
   U_{\phi}(x)=\left\{
     \begin{split}
           & U_{\phi}(-1)+b_1(1+x)^{\alpha_0}+b_2(1+x)^{2\alpha_0}+b_3(1+x)^{1+\alpha_0}\\
             &  \quad +O((1+x)^{\alpha_0+2-\epsilon}) +O((1+x)^{3\alpha_0-\epsilon}), \hspace{1cm} \textrm{ if } U_{\theta}(-1)\ne 0; \\
           & U_{\phi}(-1)+b_1(1+x)+b_2(1+x)^{2}\ln(1+x)+b_3(1+x)^{2} \\
             & \quad + O((1+x)^{3-\epsilon}), \hspace{4.65cm} \textrm{ if } U_{\theta}(-1)=0. 
        \end{split}
        \right.
\]
(ii)  If  $2<U_{\theta}(-1)<3$, then there exist some constants $b_1,b_2,b_3,b_4$ such that
\[
\begin{split}
   U_{\phi}(x) = & b_1(1+x)^{1-\frac{U_{\theta}(-1)}{2}}+b_2+b_1b_3(1+x)^{4-\frac{3U_{\theta}(-1)}{2}}+b_1b_4(1+x)^{2-\frac{U_{\theta}(-1)}{2}} \\
	& + b_1O((1+x)^{7-\frac{5U_{\theta}(-1)}{2}-\epsilon}) .
\end{split}
\]
In particular, $U_{\phi}$ is either a constant or an unbounded function in $(-1,-1+\delta)$.\\
(iii) If $U_{\theta}(-1)\ge 3$, then $U_{\phi}$ must be a constant in $(-1,-1+\delta)$.\\
(iv) If $U_{\theta}(-1)=2$, then $\eta:=\lim_{x\to -1^+}(U_{\theta}-2)\ln(1+x)$ exists and is $0$ or $4$. If $\eta=0$, then $U_{\phi}$ is either constant or unbounded, and there exist some constants $b_1,b_2$ such that
\[
  U_{\phi}= b_1\ln(1+x)+b_2+b_1O((1+x)^{1-\epsilon}). 
\]
If $\eta=4$, then $U_{\phi}$ is in $L^{\infty}(-1,-1+\delta)$, and there exists some constant $b$ such that
\[
   U_{\phi}=U_{\phi}(-1)+\frac{b}{\ln(1+x)}+O((\ln (1+x))^{-2+\epsilon}).
\]
\end{thm}

\begin{rmk}
   We have obtained near $x=-1$ much more detailed expansions of $U_{\theta}$ and $U_{\phi}$  than those in Theorem \ref{thm1_5} and Theorem \ref{thm1_6}. In particular,  if $U_{\theta}(-1)\ge 3$, $U_{\theta}$ can be expressed as a power series of $(1+x)$ near $x=-1$. These will be presented  in a later paper.
\end{rmk}

A consequence of Theorem \ref{thm1_2} and Theorem \ref{thm1_6} is 
\begin{cor}
  For every $\tau<3$, there exists an axisymmetric solution   $(U_{\theta}, U_{\phi})$ with nonzero swirl of (\ref{NS})   in $C^{\infty}(\mathbb{S}^2\setminus\{S\})$ such that $U_{\theta}(-1)=\tau$. On the other hand, every axisymmetric solution $(U_{\theta}, U_{\phi})$ of (\ref{NS}) in $C^{\infty}(\mathbb{S}^2\setminus\{S\})$ with $U_{\theta}(-1)\ge 3$ necessarily has zero swirl, i.e. $U_{\phi}\equiv 0$.
\end{cor}

Landau interpreted the solutions he found (Landau solutions)  as a jet discharged from a point. Experimentally, a pingpong ball can float and be stable in a jet of air (such as when we blow into a straw upwards). However, as pointed out by \v{S}ver\'{a}k, the pressure in the center of the Landau jet is higher than the pressure nearby, and therefore the exact  Landau jets solutions are unlikely to support a pingpong ball in a stable way. The real-life jets are turbulent and this plays an important role. The Landau solutions could still be relevant when one thinks in terms of averaging, turbulent viscosity, Reynolds stress, etc.  Still, the pressure profiles are of interest and in Section 6, we identify all axisymmetric no-swirl solutions in a neighborhood of the north pole of $\mathbb{S}^2$, which describe fluid jets with lower pressure in the center. It would be interesting to compare some of these solutions to real-life jets.

There have been some other papers on (-1)-homogeneous axisymmetric solutions of the stationary NSE (\ref{NS}), see  \cite{Goldshtik}, \cite{PAULL1}, \cite{PAULL2}, \cite{PAULL3}, \cite{Serrin}, \cite{Slezkin}, \cite{Squire}, \cite{WangSteady} and \cite{Y1950}. In the no-swirl case, the equations were converted to an equation of Riccati type in \cite{Slezkin}, see also \cite{Y1950} where various exact solutions on annulus regions of $\mathbb{S}^2$ were given.

The organization of the paper is as follows. 
In Section 2, we reduce the NSE in the framework of spherical coordinates.  We also give an alternative proof of the above mentioned result in \cite{TianXin} in the framework. 
In Section 3, we classify all (-1)-homogeneous axisymmetric no-swirl solutions of the stationary NSE (\ref{NS}) on $\mathbb{S}^2\setminus\{S\}$.  
The existence part of
Theorem \ref{thm1_2} is established in Section 4. 
It is proved by using implicit function theorems in suitably chosen
weighted norm Banach spaces.    Three different sets of spaces are used
according to which of the three
parts of $J$, $J_1$, $J_2$ or $J_3\cap 
\{2 \le \tau<3\}$,  $(\tau, \sigma)$ belongs to.
They are as Theorem \ref{thm4_1}, \ref{thm4_2} and \ref{thm4_3}, proved respectively
in Section 4.2, 4.3 and 4.4. 
Asymptotic behavior of solutions in a punctured ball $B_{\delta}(S)\setminus\{S\}$ of $\mathbb{S}^2$ is studied in Section 5. Theorem \ref{thm1_5}, \ref{thm1_6} and the nonexistence part, therefore the completion of Theorem \ref{thm1_2} are established in this section.
Several results on  first order ordinary differential equations used in Section 5 are given in Section 7.

\noindent
{\bf Acknowledgment}. The authors thank L. Nirenberg and V. \v{S}ver\'{a}k for stimulating and encouraging conversations. The work of the first named author is carried out during visits to Rutgers University. The hospitality of the department is warmly acknowledged. Her work is partially supported by NSFC (grants No.11001066 and No.11371113). The work of the second named author is partially supported by NSF grants DMS-1065971 and DMS-1501004.

\section{Reduction of equations}

Our first attempt in proving Theorem \ref{thm1_2} is to work with $(u_{\theta}, u_{\phi})$ and to find some spaces with appropriate weights on $u_{\theta}$ and $u_{\phi}$ together with their derivatives near the south pole $S$. However, we encounter difficulties of loss of derivatives when trying to apply implicit function theorems.  As mentioned earlier, we work with new functions $U_r$, $U_{\theta}$ and $U_{\phi}$, and a new variable $x$ as defined in (\ref{Va}). Both formulations, with $u$ and $\theta$ or with $U$ and $x$ are widely used in literature.

For any $-1\leq \delta_1<\delta_2\leq 1$, system (\ref{eq_homo}) in the range $\delta_1<x<\delta_2$ can be reformulated into the following third order ODE system of $U_{\theta}, U_{\phi}$ and $p$: 
\begin{equation}\label{eq_NSwithpresure}
\left\{
\begin{split}
	& -(1-x^2) U_\theta''' +2x U_\theta''   - {U_\theta'}^2 - U_\theta U_\theta''  - \frac{U_\theta^2}{1-x^2}  - \frac{U_\phi^2}{1-x^2} - 2p = 0, \\
	& (1-x^2) U_\theta''  -  U_\theta U_\theta' - \frac{x}{1-x^2} U_\theta^2 - \frac{x}{1-x^2} U_\phi^2 - (1-x^2) p' =0,  \\
	& -(1-x^2) U_\phi'' - U_\theta U_\phi' = 0. 
\end{split}
\right. 
\end{equation}
with the divergence free condition 
\begin{equation}\label{div}
  U_r = U_\theta'\sthe. 
\end{equation} 
Differentiating the first line of (\ref{eq_NSwithpresure}) in $x$, then subtracting $\frac{2}{1-x^2}$ times the second line, we have the following fourth order ODE system of $U_{\theta}$ and $U_{\phi}$
\begin{equation}\label{eq_NS_x1}
\left\{
\begin{split}
	& -(1-x^2) U_\theta'''' +4x U_\theta'''  - 3U_\theta' U_\theta'' - U_\theta U_\theta''' - \frac{2 U_\phi U_\phi'}{1-x^2} = 0, \\
	& -(1-x^2) U_\phi'' - U_\theta U_\phi' = 0. 
\end{split}
\right. 
\end{equation} 
Since   
\begin{equation*}
	-(1-x^2) U_\theta'''' +4x U_\theta'''  - 3U_\theta' U_\theta'' - U_\theta U_\theta''' = -\left( (1-x^2) U_\theta' + 2x U_\theta +\frac{1}{2} U_\theta^2  \right)''',
\end{equation*}
system (\ref{eq_NS_x1}) can be converted into
\begin{equation}\label{eq_NSx2}
\left\{
\begin{split}
	& (1-x^2) U_\theta' + 2x U_\theta +\frac{1}{2} U_\theta^2 + \int \int \int \frac{2 U_\phi(s) U_\phi'(s)}{1-s^2} ds dt dl = c_1 x^2 + c_2 x + c_3, \\
	& (1-x^2) U_\phi'' + U_\theta U_\phi' = 0,
\end{split}
\right. 
\end{equation}
for some constants $c_1, c_2, c_3$. By (\ref{div}),  $U_r\in C((\delta_1, \delta_2), \mathbb{R})$ is well-defined if $U_\theta\in C^1((\delta_1, \delta_2), \mathbb{R})$, and $U_r = O(1)\sthe$ if $U_\theta'$ is bounded. The original Navier-Stokes system (\ref{NS}) is equivalent to (\ref{eq_NS_x1}) and (\ref{div}).

If there exist some constants $c_1,c_2,c_3$ and $U_{\theta}\in C^1(\delta_1,\delta_2)$, $U_{\phi}\in C^2(\delta_1,\delta_2)$ such that $(U_\theta, U_\phi)$ is a solution of (\ref{eq_NSx2}) in $(\delta_1,\delta_2)$, then the (-1)-homogeneous $u = (u_r, u_\theta, u_\phi)$ given in the corresponding domain on $\mathbb{S}^2$ by
$$
	u_r = U_\theta', \quad u_\theta = \frac{U_\theta}{\sin\theta}, \quad \quad u_\phi = \frac{U_\phi}{\sin\theta} ,
$$
satisfies the stationary NSE (\ref{NS}). We will use $U=(U_{\theta}, U_{\phi})$ to denote solutions of the stationary Navier-Stokes equations (\ref{NS}), with the meaning that $u$ determined by $U$ as above is a solution to (\ref{NS}).

With the above set up, we give an alternative proof of the following theorem:\\
\textbf{Theorem B (\cite{TianXin})}\label{thm_tianxin}
	\emph{All (-1)-homogeneous nonzero axisymmetric solutions of (\ref{NS}) in $C^2(\mathbb{R}^3\setminus\{0\})$ are Landau solutions. }
\begin{proof}
	Since the solution $u$ is smooth in $\mathbb{R}^3\setminus\{0\}$, the components $U_r$, $U_\theta$, $U_\phi$ and their derivatives are well-defined on $\mathbb{S}^2$. $U_\theta$ and $U_\phi$ vanish at $x=\pm 1$, $U_\theta = O(1)(1-x^2)$, $U_\theta'$, $U_\theta''$ are bounded in $[-1,1]$.  

	From the second line of (\ref{eq_NSx2}), we have
	\begin{equation*}
		U'_\phi = c e^{-\int \frac{U_\theta}{1-s^2}ds}, 
	\end{equation*}
	for some constant $c$, so $U_\phi$ is monotone for $x\in [-1,1]$. Since $U_\phi(1)=U_\phi(-1) = 0$, we must have $U_\phi \equiv 0$, i.e. the solution does not have a swirling components.
	
	Let $x$ go to 1 in the first line of (\ref{eq_NSx2}). Notice that $U_\theta = O(1)(1-x^2)$, and $U_\theta'$ is bounded,  we obtain
	\begin{equation*}
		c_1+c_2+c_3=\lim_{x\rightarrow 1} \left( (1-x^2) U_\theta' + 2x U_\theta +\frac{1}{2} U_\theta^2 \right) = 0, 
	\end{equation*}
	Differentiate the first line of (\ref{eq_NSx2}) with respect to $x$, then send $x\rightarrow 1$, we have
	\begin{equation*}
		2c_1+c_2=\lim_{x\rightarrow 1} \left( (1-x^2) U_\theta'' + 2 U_\theta + U_\theta U_\theta' \right) = 0. 
	\end{equation*}
	It follows that
	$$
		c_1 x^2 + c_2 x + c_3 = c_1(1-x)^2. 
	$$
	Repeat the above analysis similarly as $x$ goes to $-1$, we have
	$$
		c_1 x^2 + c_2 x + c_3 = c_1(1+x)^2. 
	$$
	Therefore, we must have $c_1=c_2=c_3=0$, $U_\phi = 0$. It is now easy to see that $u$ is a Landau solution, $u=\frac{2\sin\theta}{\lambda+\cos\theta}$ with $|\lambda|>1$.
\end{proof}

\section{Classification of axisymmetric no-swirl solutions on $\mathbb{S}^2\setminus\{S\}$}

In this section, we will prove Theorem \ref{thm1_1}, which classifies all (-1)-homogeneous axisymmetric no-swirl $C^{\infty}(\mathbb{S}^2\setminus\{S\})$ solutions of (\ref{NS}). More generally,  we study axisymmetric no-swirl solutions of (\ref{NS}) which are smooth in a neighborhood of the north pole.

By arguments used in Section 2, $u$ is a solution of (\ref{NS}) in $\mathbb{S}^2\setminus\{N,S\}$  if and only if $U$ defined by (\ref{Va}) satisfies  (\ref{eq_NSx2}) in $(-1,1)$ for some constants $c_1,c_2$ and $c_3$. When the solution has no swirling component, (\ref{eq_NSx2}) becomes 
\begin{equation}\label{eq_NStemp}
	(1-x^2) U_\theta' + 2x U_\theta +\frac{1}{2} U_\theta^2 = c_1 x^2 + c_2 x + c_3. 
\end{equation}

Let $u$ be a solution which is smooth in a neighborhood of the north pole, the proof of Theorem B in Section 2 actually shows that the polynomial on the right hand side of (\ref{eq_NStemp}) must be $\mu(1-x)^2$ for some constant $\mu$. Therefore, the NSE is
\begin{equation}\label{eq_smthN}
	(1-x^2) U_\theta' + 2x U_\theta +\frac{1}{2} U_\theta^2 = \mu(1-x)^2. 
\end{equation}

\begin{lem}\label{lem3_1}
	Let $\mu, \gamma\in \mathbb{R}$ and $\delta\in[-1,1)$, equation (\ref{eq_smthN}) has at most one solution $U_\theta\in C^1(\delta,1)$ satisfying 
	\begin{equation}\label{eq3_1a}
		\lim_{x\to 1^-}U_{\theta}(x)=0, \textrm{ and }\lim_{x\rightarrow 1^-} U_\theta'(x) = \gamma. 
	\end{equation}
\end{lem}
\begin{proof}
	Let $U_\theta^{(i)}$ ($i=$1, 2) be two such solutions. Then $g_i(x):= (1-x^2)^{-1}U_\theta^{(i)}$ satisfies
	$$
		g'_i(x) + \frac{1}{2} g_i^2(x) = \frac{\mu}{(1+x)^2}, \quad \delta<x<1, \quad i=1,2. 
	$$
	Using (\ref{eq3_1a}) and the L'Hospital's rule, 
	$$
		\lim_{x\rightarrow 1^-}g_i(x)  = -\frac{\gamma}{2}, \quad i=1,2. 
	$$
	So $g_i(x)$ can be extended as functions in $C^0(\delta,1]$, $g_1(1) = g_2(1)$, and $g_1-g_2$ satisfies $(g_1-g_2)'+\frac{1}{2}(g_1+g_2)(g_1-g_2)=0$ in $(\delta,1)$ with $(g_1-g_2)(1)=0$. It follows that $g_1\equiv g_2$ in $(\delta,1)$, so $U^{(1)}_{\theta}\equiv U^{(2)}_{\theta}$ in $(\delta,1)$. 
\end{proof}

Let $b := \sqrt{|1+2\mu|}$,  $\delta^*\in C(\mathbb{R}^2, [-1,1))$ be given by
\begin{equation}\label{eq_deltastar}
	\delta^* := \delta^*(\mu, \gamma) := 
	\left\{ 
	\begin{matrix}
		& -1, & \mu\geq -\frac{1}{2}, \gamma \geq - (1+\sqrt{1+2\mu}); \\
		& -1+ 2\left( \frac{\gamma+1 - b}{\gamma+1+ b}\right)^{-1/b}, & \mu> -\frac{1}{2}, \gamma < - (1+\sqrt{1+2\mu}); \\
		& -1+ 2 e^{\frac{2}{\gamma+1}}, & \mu = -\frac{1}{2}, \gamma<-1; \\
		& -1+ 2 \exp\left( \frac{2}{b}(\arctan \frac{b}{\gamma+1} - \pi)\right), & \mu < -\frac{1}{2}, \gamma > -1; \\	
		& -1+ 2 \exp\left( -\frac{\pi}{b} \right), & \mu < -\frac{1}{2}, \gamma = -1; \\	
		& -1+ 2 \exp\left( \frac{2}{b}\arctan \frac{b}{\gamma+1} \right), & \mu < -\frac{1}{2}, \gamma < -1. \\
	\end{matrix}
	\right. 
\end{equation}

\begin{thm}[Exact form of axisymmetric no-swirl solutions]\label{thm_smooth_north}
	For every $(\mu,\gamma)\in\mathbb{R}^2$, there exists a unique $U_\theta := U_\theta(\mu,\gamma;\cdot)\in C^\infty(\delta^*,1)$ satisfying  (\ref{eq_smthN}) in $(\delta^*,1)$ and 
	\begin{equation}\label{eq_Uthelim}
		\lim_{x\rightarrow 1^-} U_\theta'(x) = \gamma. 
	\end{equation}
	The interval $(\delta^*,1)$ is the maximal interval of existence for $U_\theta$, and in particular, 
	\begin{equation}\label{eq_Utheblowup}
		\lim_{x\rightarrow \delta^{*+}} |U_\theta(x)| = \infty, \textrm{ if $\delta^*>-1$. }
	\end{equation}
	Moreover, $U_\theta$ is explicitly given by 
	\begin{equation}\label{eq_temp110}
	U_\theta(x) = 
	\left\{
	\begin{matrix}
		& (1-x)\left(1- b - \frac{2 b(\gamma+1-b)}{(\gamma+1+b) (\frac{1+x}{2})^{-b} - \gamma-1 + b} \right), & \mu > - \frac{1}{2}, \\
		& (1-x)  \left( 1+ \frac{2 (\gamma+1)  }{  (\gamma+1) \ln \frac{1+x}{2}  - 2 } \right), & \mu = - \frac{1}{2}, \\
		& (1-x) \left( 1  + \frac{ b ( b \tan\frac{\beta(x)}{2} + \gamma+1) }{(\gamma+1) \tan\frac{\beta(x)}{2} - b } \right), & \mu < - \frac{1}{2}, \\
	\end{matrix}
	\right.  
	\end{equation}
	where $b := \sqrt{|1+2\mu|}$,  and $\beta(x):= b \ln \frac{1+x}{2}$. 
\end{thm}

We will also use $U^{\mu,\gamma}$ to denote the axisymmetric no-swirl solution $(U_{\theta}(\mu,\gamma; \cdot),0)$ in the above theorem.

Let $u=u(\mu,\gamma)$ be the solution generated by $(U_\theta(\mu,\gamma),0)$, then $\{u(0,\gamma)\mid \gamma > - 2, \gamma\ne 0\}$ are Landau solutions. In particular, $U_{\theta}(x)=\frac{2(1-x^2)}{x+\lambda}$ with $|\lambda|>1$, and $\delta^*(0,\gamma) = -1$ for any $\gamma >-2$, $\gamma\not=0$. 

It is easy to see that $U_\theta(\mu,\gamma) \neq U_\theta(\mu',\gamma')$ if $(\mu,\gamma) \neq(\mu',\gamma')$. Let $I$ be defined by
\begin{equation*}
	I := \{(\mu,\gamma)\mid \mu\ge -\frac{1}{2}, \gamma \ge -1-\sqrt{1+2\mu}\}, \quad 
	I^c = \mathbb{R}^2\setminus I.  
\end{equation*}
Then $\delta^*(\mu,\gamma) = -1$ if and only if $(\mu,\gamma)\in I$. 
Consequently, $u(\mu,\gamma)\in C^\infty(\mathbb{S}^2\setminus\{S\}) \setminus C^1(\mathbb{S}^2)$ if and only if for all $(\mu,\gamma)\in I \setminus \{(0,\gamma)\mid \gamma>-2\}$. Also, it is not hard to see
\begin{equation*}
	\lim_{\gamma\rightarrow -\infty} \delta^* (\mu,\gamma) = 1, \quad \quad \lim_{\mu\rightarrow -\infty} \delta^* (\mu,\gamma) = 1. 
\end{equation*}
\begin{equation*}
	\frac{\partial \delta^*(\mu,\gamma)}{\partial \mu} < 0, \quad \frac{\partial \delta^*(\mu,\gamma)}{\partial \gamma} < 0, \quad \mbox{for } (\mu,\gamma)\in I^c.  
\end{equation*} \\

\noindent \emph{Proof of Theorem \ref{thm_smooth_north}:}
	For every $(\mu,\gamma)\in\mathbb{R}^2$,  let $a$ be a root of $\frac{1}{2}a^2-a=\mu$ (real or complex) then $h(x) := a(1-x)$ is a solution of (\ref{eq_smthN}). If $U_{\theta}$ is a solution of (\ref{eq_smthN}), denote $g := U_\theta - h$, then $g$ satisfies
	\begin{equation*}
		(1-x^2) g' + 2x g + hg+\frac{1}{2} g^2 = 0.  
	\end{equation*}
	
	Multiplying both sides by the integrating factor $(1+x)^{a-1}(1-x)^{-1}$, we have
	\begin{equation*}
		\tilde{g}' +\frac{1}{2} (1+x)^{-a}\tilde{g}^2 = 0, 
	\end{equation*}
	where $\tilde{g}(x):=(1+x)^{a-1}(1-x)^{-1} g$. Solving this equation directly we have 
	\begin{equation*}
		\tilde{g} = \frac{2(1-a)}{(1+x)^{1-a}+c'}, \quad \mbox{if } a\not=1. 
	\end{equation*}
	Then
	\begin{equation}\label{eq3_1U}
		U_\theta = a(1-x) + \frac{2(1-a)(1-x)(\frac{1+x}{2})^{1-a}}{(\frac{1+x}{2})^{1-a}+c} 
	\end{equation}
	where $c$ is a (real or complex) constant.
	
	Let $b := \sqrt{|1+2\mu|}$. When $\mu>-\frac{1}{2}$, we can take $a = 1+b$, $c  = \frac{-\gamma-2+a}{\gamma+a}$. Then $U_\theta$ is the function in the first line of (\ref{eq_temp110}) which satisfies (\ref{eq_Uthelim}) and (\ref{eq_smthN}) in $(\delta^*,1)$ where $\delta^*$ is given in (\ref{eq_deltastar}), and by Lemma \ref{lem3_1} it is the only solution satisfying (\ref{eq_smthN}) and (\ref{eq_Uthelim}). Property (\ref{eq_Utheblowup}) follows from standard ODE theory. $U_\theta$ are Landau solutions when $\mu=0$ and $\gamma > -2$, $\gamma\ne 0$,
	\begin{equation*}
		U_\theta(x) = \frac{2(1-x^2)}{ x +\lambda} 
	\end{equation*} 
	where $\lambda =- \frac{\gamma+4}{\gamma}$. It can be seen that when $\gamma > -2$, $\gamma\ne 0$, there is $|\lambda|>1$.

	When $\mu< - \frac{1}{2}$, we can take $a = 1+i b$, $c = \frac{-\gamma-2+a}{\gamma+a}$. Then the real part of (\ref{eq3_1U}) can be rewritten as the function in the third line of (\ref{eq_temp110}), which satisfies the required properties. 
	
	When $\mu=-1/2$, we have $a=1$ and
	\begin{equation*}
		\tilde{g}' +\frac{1}{2} (1+x)^{-1}\tilde{g}^2 = 0 
	\end{equation*}
	where $\tilde{g}(x):=(1-x)^{-1} g$. Thus, 
	\begin{equation*}
		U_\theta = (1-x) + \frac{1-x}{\frac{1}{2}\ln(1+x)+c}. 
	\end{equation*}
	
	Choosing  $c = \frac{-1}{\gamma+1} - \frac{1}{2}\ln 2$,   then $U_\theta$ is the function in the second line of (\ref{eq_temp110}) which satisfies all the required properties. 	
\qed

\begin{figure}[h]
	\label{fig_plane}
	\centering
	\includegraphics[height=6cm]{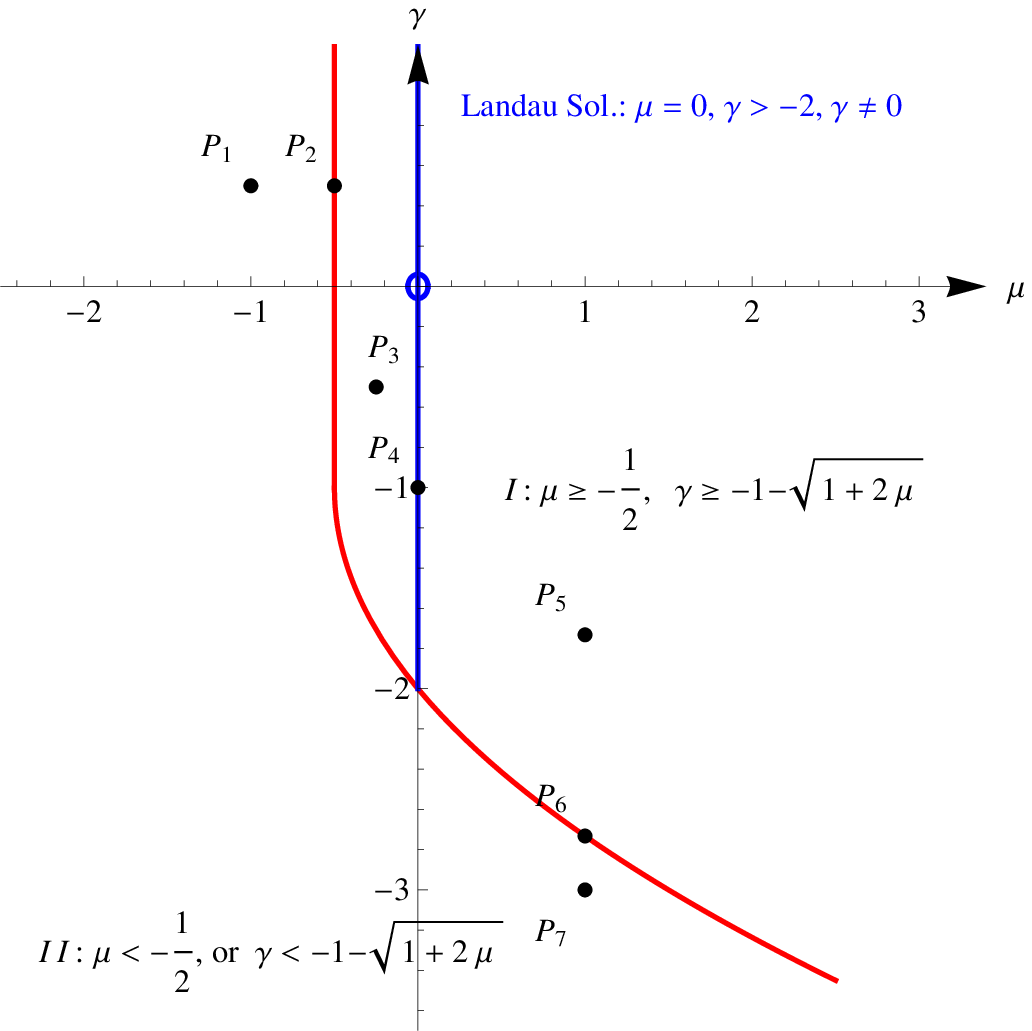}
	\includegraphics[height=5cm]{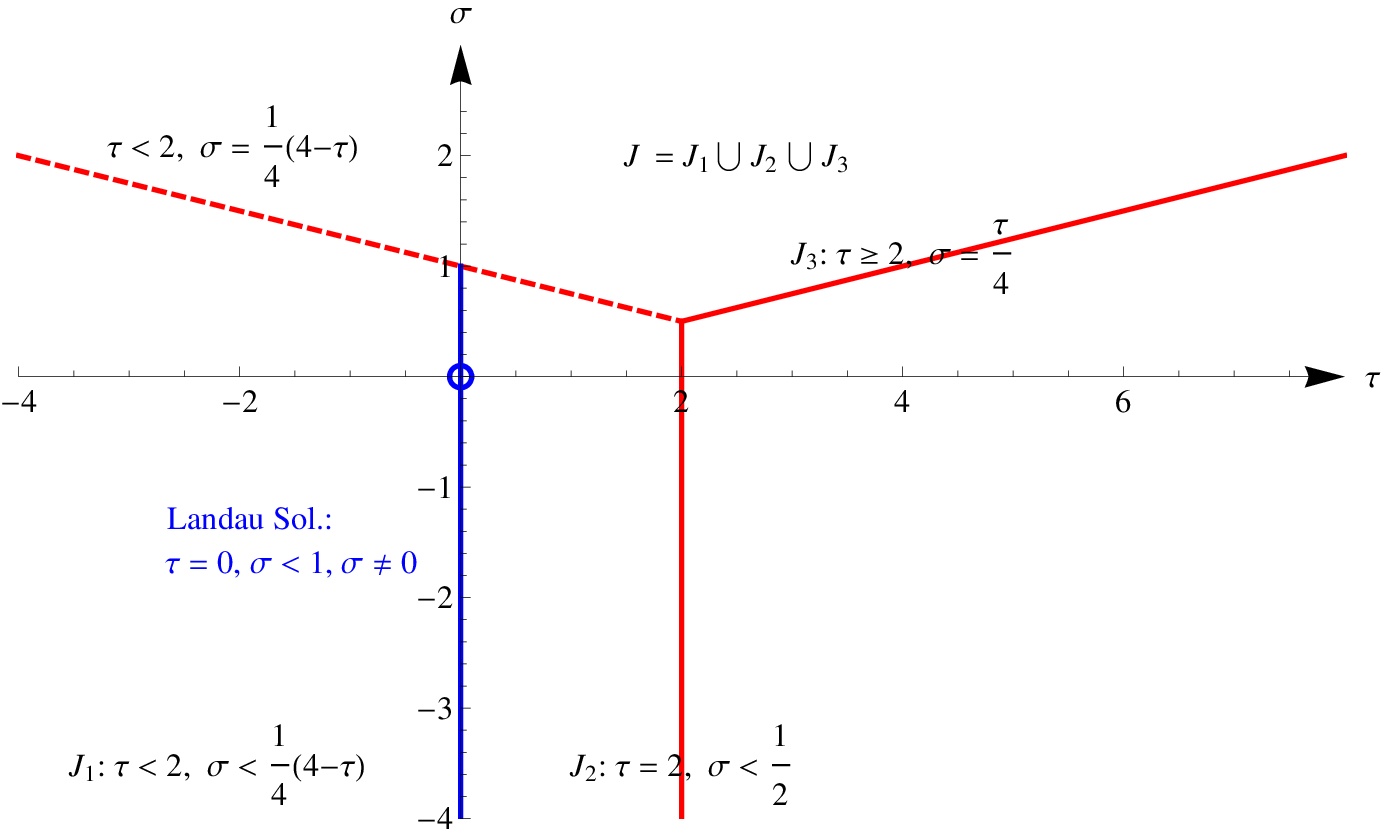}
	\caption{Dependence on parameters ($\mu, \gamma$) or ($\tau, \sigma$) of the maximal existence domains of the solutions to NSE}\label{figure3_1}
\end{figure}

	Figure 1 shows the dependence of the maximal existence domains on parameters ($\mu, \gamma$) or ($\tau, \sigma$). When the parameters $(\mu, \gamma) \in I$, or equivalently $(\tau, \sigma) \in J$, the solution is smooth on $\mathbb{S}^2\setminus\{S\}$; When the parameters $(\mu, \gamma) \not\in I$, or equivalently $(\tau, \sigma) \not\in J$, the solution exists and is smooth in a neighborhood of the north pole $\{N\}$ but not on the entire $\mathbb{S}^2\setminus\{S\}$. Some typical points are chosen in the $(\mu,\gamma)$ plane, (i.e. left part of Figure 1). The graph and stream lines at these points are presented in Section \ref{sec_fig}.

Here is an immediate consequence of Theorem \ref{thm_smooth_north}: 
\begin{cor}\label{cor3_1}
	Suppose $U$ is an axisymmetric, no-swirl solution of Navier-Stokes equation and is smooth on $\mathbb{S}^2\setminus \{S\}$, then $U_\theta(x)$ is given by a two-parameter-family $(\mu, \gamma)$ with $\mu\geq -\frac{1}{2}$, $\gamma \geq -1-\sqrt{1+2\mu}$: 
	\begin{equation}\label{eqcor3_1_1}
		U_\theta(x) =\left\{
		\begin{array}{ll}
		   (1-x)\left(1- b - \frac{2 b(\gamma+1-b)}{(\gamma+1+b) (\frac{1+x}{2})^{-b} - \gamma-1 + b} \right), & \mu > -\frac{1}{2}, \gamma>-1-\sqrt{1+2\mu},\\
		   (1-x)  \left( 1+ \frac{2(\gamma+1)  }{  (\gamma+1) \ln \frac{1+x}{2}  - 2 } \right) & \mu =  -\frac{1}{2}, \gamma>-1,\\
		    (1+b)(1-x), & \mu \ge -\frac{1}{2}, \gamma=-1-\sqrt{1+2\mu}. 
		   \end{array}
		   \right.
	\end{equation}
\end{cor}

Since $U_{\theta}(x)= u_{\theta}\sin\theta$, $x=\cos\theta$ and (\ref{eq_thm1_1}), $\tau=\lim_{\theta\to \pi^-} u_{\theta}(x)\sin\theta=\lim_{x\to-1^+}U_{\theta}(x)$, $\gamma=\lim_{x\to 1}U'_{\theta}(x)=-2\lim_{\theta \to 0}\frac{u_{\theta}}{\sin\theta}=-2\sigma$, and $\mu=\lim_{x\to -1^+}\frac{1}{4}(\frac{1}{2}U^2_{\theta}(x)-2U_{\theta}(x))=\frac{1}{8}\tau^2-\frac{1}{2}\tau$. The relation 
\[
   \mu=\frac{1}{8}\tau^2-\frac{1}{2}\tau, \quad \gamma=-2\sigma
\]
gives a one-one correspondence between $\{u_{\tau,\sigma}|,(\tau, \sigma)\in J\}$ and $\{U^{\mu,\gamma}| (\mu,\gamma)\in I\}$ with $ u_{\tau,\sigma}\sin\theta=U^{\mu,\gamma}$. Moreover,  region $J_1$ corresponds to 
\[
	I_1:=I^0=\{(\mu,\gamma)\in I|,\mu>-\frac{1}{2},  \gamma>-1-\sqrt{1+2\mu}\},
\]
boundary $J_2$ corresponds to 
\[
	I_2:=\{(\mu,\gamma)\in I|\mu=-\frac{1}{2}, \gamma>-1\}, 
\]
$J_3$ corresponds to 
\[
	I_3:=\{(\mu,\gamma)\in I| \mu>-\frac{1}{2}, \gamma=-1-\sqrt{1+2\mu}\}. 
\]
Also, $J_3\cap\{2\le \tau<3\}$ corresponds to $I_3\cap \{-\frac{1}{2}\le \mu<-\frac{3}{8}\}$ and $J_3\cap\{\tau>3\}$ corresponds to $I_3\cap \{\mu>-\frac{3}{8}\}$.

Theorem \ref{thm1_1} follows from the above corollary. 

\begin{rmk}\label{rmk3_1}
From Theorem \ref{thm_smooth_north} and Corollary \ref{cor3_1} we can see that $U_{\theta}(\mu,\gamma)$ exists on all $(-1,1]$ if and only if $(\mu,\gamma)\in I$, which is shown in the first graph of Figure 1, and the behavior of $U_{\theta}$ near the south pole is different when $(\mu,\gamma)\in I_1$, the interior of $I$,  and when $(\mu,\gamma)\in \partial I$.

When $(\mu,\gamma)\in I_1$, $\mu>-1/2$, we have for $-1<x<1$, $i, j \in \mathbb{Z}$
 \begin{equation}\label{eqcor3_1_2}
 \begin{split}
  & U^{\mu,\gamma}_{\theta}(x)=(1-x)(1-\sqrt{1+2\mu}+O(1)(1+x)^b), \\
   & \partial^i_{\mu}U^{\mu,\gamma}_{\theta}(x)=(1-x)\left(-\frac{d^i}{d\mu^i}\sqrt{1+2\mu}+O(1)(1+x)^b\left|\ln\left(\frac{1+x}{2}\right)\right|^i\right),\quad i \geq 1, \\
   & \left|\partial^i_{\mu}\partial^j_{\gamma}U^{\mu,\gamma}_{\theta}(x)\right|=O(1)(1-x)(1+x)^b\left|\ln\left(\frac{1+x}{2}\right)\right|^i, \quad i\ge 0, j\ge 1.
   \end{split}
 \end{equation}

When $(\mu,\gamma)\in I_2$, we have
\begin{equation}\label{eqcor3_1_3}
   \begin{split}
      & U^{\mu,\gamma}_{\theta}(x)=(1-x)\left(1+2\left(\ln\frac{1+x}{3}\right)^{-1}+O(1)\left(\ln\frac{1+x}{3}\right)^{-2} \right), \\
      &  \partial^i_{\gamma}U^{\mu,\gamma}_{\theta}(x)=O(1)(1-x)\left(\ln\frac{1+x}{3}\right)^{-2},\quad  i\ge 1.
   \end{split}
\end{equation}

When $(\mu,\gamma)\in I_3$, $U_{\theta}(x)=(1+b)(1-x)$, which is a linear function, and
\begin{equation}\label{eqcor3_1_4}
  \begin{split}
      & U^{\mu,\gamma}_{\theta}(x)=(1-x)(1+\sqrt{1+2\mu}), \\
      & \partial^i_{\mu}U^{\mu,\gamma}_{\theta}(x)=\frac{\partial^i}{\partial \mu^i}\sqrt{1+2\mu}(1-x), \quad  i\ge 1.
  \end{split}
\end{equation}
\end{rmk}

\section{Existence of axisymmetric solutions with nonzero swirl on $\mathbb{S}^2\setminus\{S\}$}

\subsection{Framework of proofs}

The set of all axisymmetric no swirl solutions of the NSE (1) in $C^\infty(\mathbb{S}^2 \setminus \{S\})$ is classified in Section 3 as the two dimensional surface of solutions
$ \displaystyle{
\big \{ U^{\mu, \gamma}= (U^{\mu, \gamma}_\theta, 0)\ \ \big|\ \
(\mu, \gamma)\in I \big\} }$.
In this section, we will use implicit function theorems in suitably chosen weighted normed spaces to prove the existence of a curve of axisymmetric solutions with non-zero swirl emanating from each
$U^{\mu, \gamma}$ for $(\mu, \gamma) \in I\setminus ( I_3\cap \{\mu\ge -\frac{3}{8}\})$.
 
Since $U_{\theta}(-1)$ affects the behavior of $U_{\theta}$ and $U_{\phi}$ near the singularity $x=-1$, we will need to use
different function spaces according to the values of $U_{\theta}(-1)$.
It is easy to check that
$U^{\mu, \gamma}(-1)\in (-\infty,2)$ for $(\mu, \gamma)\in I_1$,
$U^{\mu, \gamma}(-1)=2$  for $(\mu, \gamma)\in I_2$,
$U^{\mu, \gamma}(-1)\in [2,3)$   for $(\mu, \gamma)\in I_3\cap\{-\frac{1}{2}\le \mu<-\frac{3}{8}\}$.
We will use three different sets of weighted normed spaces based on which of the three sets,
$I_1,I_2$, and $I_3\cap\{-\frac{1}{2}\le \mu<-\frac{3}{8}\}$, $(\mu, \gamma)$ belongs to.

On the other hand, $U^{\mu, \gamma}(-1)>3$ for $(\mu,\gamma)\in I_3\cap \{\mu> -\frac{3}{8}\}$. It will be proved in Section 5 that for every $(\mu,\gamma)\in I_3\cap \{\mu> -\frac{3}{8}\}$, there exists no sequence of axisymmetric solution with nonzero swirl in $C^{\infty}(\mathbb{S}^2\setminus\{S\})$ which converge to $U^{\mu,\gamma}$ in $L^{\infty}(-1,1)$.

For convenience, let us use $\bar{U}$ to denote axisymmetric no-swirl solutions of the stationary NSE.

The equations of axisymmetric solutions in $C^{\infty}(\mathbb{S}^2\setminus\{S\})$ are of the form 
\begin{equation}\label{eq4_0}
\left\{
\begin{split}
	& (1-x^2) U_\theta' + 2x U_\theta +\frac{1}{2} U_\theta^2 - \int _{x}^{1}\int_{l}^{1} \int_{t}^{1} \frac{2 U_\phi(s) U_\phi'(s)}{1-s^2} ds dt dl =\hat{\mu} (1-x)^2, \\
	& (1-x^2) U_\phi'' + U_\theta U_\phi' = 0,
\end{split}
\right. 
\end{equation}
where $\hat{\mu}$ is a constant.

We first introduce the implicit function theorem (IFT) which we use:\\
\textbf{Theorem C (\cite{Nirenberg})} (Implicit Function Theorem) \emph{Let $\mathbf{X},\mathbf{Y},\mathbf{Z}$ be Banach spaces and $f$ a continuous mapping of an open set $U\subset \mathbf{X}\times \mathbf{Y}\to \mathbf{Z}$. Assume that $f$ has a Fr\'{e}chet derivative with respect to $x$, $f_{x}(x,y)$ which is continuous in $U$. Let $(x_0,y_0)\in U$ and $f(x_0,y_0)=0$. If $A=f_{x}(x_0,y_0)$ is an isomorphism of $\mathbf{X}$ onto $\mathbf{Z}$ then}\\
\emph{(1) There is a ball $\{y:||y-y_0||<r\}=B_r(y_0)$ and a unique continuous map $u:B_r(y_0)\to \mathbf{X}$ such that $u(y_0)=x_0$ and $f(u(y),y)\equiv 0$.}\\
\emph{(2) If $f$ is of class $C^1$ then $u(y)$ is of class $C^1$ and $u_y(y)=-(f_{x}(u(y),y))^{-1}\circ f_{y}(u(y),y)$.}\\
\emph{(3) $u_{y}(y)$ belongs to $C^k$ if $f$ is in $C^k$, $k>1$}.

We will work with $\tilde{U}:=U-\bar{U}$,  a calculation gives
\begin{equation*}
    (1-x^2)U'_{\theta}+2xU_{\theta}+\frac{1}{2}U^2_{\theta}-\mu(1-x)^2=(1-x^2)\tilde{U}_{\theta}'+(2x+\bar{U}_{\theta})\tilde{U}_{\theta}+\frac{1}{2}\tilde{U}_{\theta}^2,
\end{equation*}
where $\tilde{U}_{\phi}=U_{\phi}$. Denote
\begin{equation}\label{eqpsi4_1}
   \psi[\tilde{U}_{\phi}](x):=\int_{x}^{1}\int_{l}^{1}\int_{t}^{1}\frac{2\tilde{U}_{\phi}\tilde{U}'_{\phi}}{1-s^2}dsdtdl.
\end{equation}
 Define a map $G$ on $(\mu,\gamma, \tilde{U})$ by
\begin{equation}\label{G}
\begin{split}
   G(\mu,\gamma, \tilde{U}) & =
      \left(
     \begin{matrix}
        (1-x^2)\tilde{U}'_{\theta}+(2x+\bar{U}_{\theta})\tilde{U}_{\theta}+\frac{1}{2}\tilde{U}^2_{\theta}-\psi[\tilde{U}_{\phi}](x)+\frac{1}{4}\psi[\tilde{U}_{\phi}](-1)(1-x)^2\\
        (1-x^2)\tilde{U}''_{\phi}+(\tilde{U}_{\theta}+\bar{U}_{\theta})\tilde{U}'_{\phi}
     \end{matrix}
     \right) \\
     & =:
    \left(
     \begin{matrix}
        \xi_{\theta}\\
        \xi_{\phi}
     \end{matrix}
     \right). 
\end{split}
\end{equation}
If $(\mu,\gamma, \tilde{U})$ satisfies $G(\mu,\gamma, \tilde{U})=0$, then $U=\tilde{U}+\bar{U}$ gives a solution of (\ref{eq4_0})with $\hat{\mu}=\mu-\frac{1}{4}\psi[\tilde{U}_{\phi}](-1)$ satisfying $U_{\theta}(-1)=\bar{U}_{\theta}(-1)$.

Let $A$ and $Q$ be maps of  the form
\begin{equation}\label{eqP4_2}
     A(\mu,\gamma, \tilde{U}) =\left(
	\begin{matrix}
		A_\theta\\
		A_\phi
	\end{matrix}
	\right) 	 :=
	\left(
	\begin{matrix}
		(1-x^2)\tilde{U}'_{\theta}+(2x+\bar{U}_{\theta})\tilde{U}_{\theta}\\
		(1-x^2)\tilde{U}''_\phi+\bar{U}_{\theta}\tilde{U}'_{\phi}
	\end{matrix}
	\right) ,
\end{equation}
and 
 \begin{equation}\label{eqP4_3}
\begin{split}
     & Q(\tilde{U},\tilde{V}) = \left(
	\begin{matrix}
		Q_\theta\\
		Q_\phi
	\end{matrix}
	\right) 	 \\
	& := 
	\left(
	\begin{matrix}
		\frac{1}{2}\tilde{U}_{\theta}\tilde{V}_{\theta}-\int_x^1 \int_l^1 \int_t^1 \frac{2 \tilde{U}_\phi(s) \tilde{V}_\phi'(s)}{1-s^2} ds dt dl+\frac{(1-x)^2}{4}\int_{-1}^1 \int_l^1 \int_t^1 \frac{2 \tilde{U}_\phi(s) \tilde{V}_\phi'(s)}{1-s^2} ds dt dl\\
		\tilde{U}_{\theta}\tilde{V}'_{\phi}
	\end{matrix}
	\right) .	
\end{split}
\end{equation}
Then $G(\mu,\gamma, \tilde{U})=A(\mu,\gamma, \tilde{U})+Q(\tilde{U}, \tilde{U})$.

By computation, the linearized operator of $G$ with respect to $\tilde{U}$ at $(\mu,\gamma,\tilde{U})$ is given by
\begin{equation}\label{eq4_Linear}
   L_{\tilde{U}} ^{\mu,\gamma}\tilde{V}
       :=
     \left(
     \begin{matrix}
        (1-x^2) \tilde{V}'_{\theta}+(2x+\bar{U}_{\theta}) \tilde{V}_{\theta}+\tilde{U}_{\theta} \tilde{V}_{\theta}-\Psi_{\tilde{U}_{\phi}}[\tilde{V}_{\phi}](x)+\frac{1}{4}\Psi_{\tilde{U}_{\phi}}[\tilde{V}_{\phi}](-1)(1-x)^2\\
        (1-x^2) \tilde{V}''_{\phi}+(\tilde{U}_{\theta}+\bar{U}_{\theta}) \tilde{V}'_{\phi}+ \tilde{V}_{\theta} \tilde{U}'_{\phi}
     \end{matrix}
     \right) 
\end{equation}
where 
\[
   \Psi_{\tilde{U}_{\phi}}[\tilde{V}_{\phi}](x):=\int_{x}^{1}\int_{l}^{1}\int_{t}^{1}\frac{\displaystyle 2(\tilde{U}_{\phi}(s) \tilde{V}'_{\phi}(s)+ \tilde{V}_{\phi}(s)\tilde{U}'_{\phi}(s))}{\displaystyle 1-s^2}dsdtdl.
\]
   
In particular, at $\tilde{U}=0$, the linearized operator of $G$ with respect to $\tilde{U}$ is
\begin{equation}\label{eq_LinearAtZero}
   L_{0} ^{\mu,\gamma} \tilde{V}=
		\left(
		\begin{matrix}
			(1-x^2) \tilde{V}_\theta' +( 2x +\bar{U}_\theta) \tilde{V}_\theta \\
			(1-x^2) \tilde{V}_\phi'' + \bar{U}_\theta \tilde{V}_\phi'
		\end{matrix}
		\right). 
		\end{equation}
Let
\begin{equation}\label{eq_ab}
    a_{\mu,\gamma}(x):=\int_{0}^{x}\frac{\displaystyle 2s+\bar{U}_{\theta}}{\displaystyle 1-s^2}ds,\quad
    b_{\mu,\gamma}(x):=\int_{0}^{x}\frac{\displaystyle \bar{U}_{\theta}}{\displaystyle 1-s^2}ds,\quad -1<x<1.
\end{equation}

By Corollary \ref{cor3_1}, for all $(\mu,\gamma)\in I$, $\bar{U}_{\theta}$ is smooth in $(-1,1]$ and $\bar{U}_{\theta}(x)=O(1-x)$. So $a_{\mu,\gamma}\in C^{\infty}(-1,1)$ and $b_{\mu,\gamma}\in C^{\infty}(-1,1]$.

 Note that this definition of $a(x)$ and $b(x)$ are consistent with the definition of $a(\theta)$ and $b(\theta)$ in Section 1, and $a(x)=-\ln(1-x^2)+b(x)$.  A calculation gives
\begin{equation}\label{eq4_W_a}
   a_{\mu,\gamma}'(x)=\frac{2x+\bar{U}_{\theta}(x)}{1-x^2}, \quad a_{\mu,\gamma}''(x)=\frac{2+\bar{U}'_{\theta}(x)}{1-x^2}+\frac{4x^2+2x\bar{U}_{\theta}(x)}{(1-x^2)^2}.
\end{equation}

Consider the following system of ordinary differential equations in $(-1,1)$:
\[
   \left\{
		\begin{matrix}
			(1-x^2) V_\theta' + 2x V_\theta + \bar{U}_\theta V_\theta=0, \\
			(1-x^2) V_\phi'' + \bar{U}_\theta V_\phi' =0.
		\end{matrix}
		\right.
\]
All solutions $V\in C^1((-1,1),\mathbb{R}^2)$   are given by
\begin{equation}\label{eq4_3_ker}
   V=c_1V_{\mu,\gamma}^1+c_2V_{\mu,\gamma}^2+c_3V_{\mu,\gamma}^3
\end{equation}
where $c_1,c_2,c_3\in\mathbb{R}$, and
\begin{equation}\label{eq4_2_ker}
	V_{\mu,\gamma}^{1} := \left(
	\begin{matrix}
		e^{-a_{\mu,\gamma}(x)}  \\  0
	\end{matrix}\right),
	\quad
	V_{\mu,\gamma}^{2} := \left(
	\begin{matrix}
		0  \\  \int_{x}^{1}e^{-b_{\mu,\gamma}(t)}dt
	\end{matrix}\right),
      \quad
       V_{\mu,\gamma}^{3} := \left(
	\begin{matrix}
		0  \\  1
	\end{matrix}\right).
\end{equation}

Next, by computation $(V_{\mu,\gamma}^1)_{\theta}(0)=1$, $(V_{\mu,\gamma}^1)_{\phi}=0$, $(V_{\mu,\gamma}^2)_{\theta}=0$ and 
$$
(V_{\mu,\gamma}^2)_{\phi}(0)=\int_{0}^{1}e^{-b_{\mu,\gamma}(t)}dt>0
$$ 
finite. Introduce the linear functionals $l_1,l_2$ acting on vector functions $V(x)=(V_{\theta}(x), V_{\phi}(x))$ by
\begin{equation}\label{def_l}
	l_1(V) :=  V_\theta(0), \quad l_2(V) :=V_\phi(0).
\end{equation}
It can be seen that for every $(\mu,\gamma)\in I$, the matrix $(l_i(V^{j}_{\mu,\gamma}))$ is a diagonal invertible matrix.

\subsection{Existence of solutions with nonzero swirl near $U^{\mu,\gamma}$ when $(\mu,\gamma)\in I_1$} 

Let us first look at the problem near $U^{\mu,\gamma}$ when $(\mu,\gamma)\in I_1$. For some fixed $(\mu,\gamma)\in I_1$, write $\bar{U}=U^{\mu,\gamma}$, recall that in Corollary \ref{cor3_1} we have
\begin{equation}\label{eqbaru}
   \bar{U}_{\theta}=(1-x)\left(1- b - \frac{2 b(1+\gamma-b)}{(1+\gamma+b) (\frac{1+x}{2})^{-b} - \gamma-1 + b} \right)
\end{equation}
where $b=\sqrt{1+2\mu}$. It satisfies
\[
   (1-x^2)\bar{U}'_{\theta}+2x\bar{U}_{\theta}+\frac{1}{2}\bar{U}^2_{\theta}=\mu (1-x)^2. 
\]

Let us start from constructing the Banach spaces we use. Given a compact subset $K\subset I_1$, from the explicit formula of $U^{\mu,\gamma}$ in Section 3, $\bar{U}:=U^{\mu,\gamma}$ satisfies $\bar{U}_{\theta}(-1)<2$, so there exists an $\epsilon>0$, depending only on $K$, satisfying $\displaystyle \max_{(\mu,\gamma)\in K}\frac{U^{\mu,\gamma}_{\theta}(-1)}{2}< \epsilon<1$ for all $(\mu,\gamma)\in K$. For this fixed $\epsilon$, define
\begin{equation*} 
\begin{split}
 &
    \begin{split}
       \mathbf{M}_1=& \mathbf{M}_1(\epsilon)\\
                           := & \left\{  \tilde{U}_\theta \in C([-1, 1], \mathbb{R}) \cap C^1((-1, 1], \mathbb{R}) \cap C^2((0, 1), \mathbb{R}) \mid  \tilde{U}_\theta(1)=\tilde{U}_\theta(-1)=0, \right.\\
        & \left. ||(1+x)^{-1+\epsilon}\tilde{U}_\theta||_{L^\infty(-1,1)}<\infty, 
	  ||(1+x)^{\epsilon}\tilde{U}'_\theta||_{L^\infty(-1,1)} < \infty, ||\tilde{U}_\theta''||_{L^\infty(0,1)} < \infty  \right\},
    \end{split}\\
	&
   \begin{split}
     \mathbf{M}_2=&\mathbf{M}_2(\epsilon)\\
                         :=& \left\{  \tilde{U}_\phi \in C^1( (-1, 1], \mathbb{R})\cap C^2( (-1, 1), \mathbb{R}) \mid \tilde{U}_\phi(1)=0, ||\tilde{U}_\phi ||_{L^\infty(-1,1)} < \infty, \right.\\
	& \left. ||(1+x)^\varepsilon \tilde{U}_\phi'||_{L^\infty(-1,1)} < \infty, ||(1+x)^{1+\varepsilon} \tilde{U}_\phi'' ||_{L^\infty(-1,1)} <\infty \right\}
   \end{split}
 \end{split}
\end{equation*}
with the following norms accordingly:
\begin{equation*}
  \begin{split}
	& ||\tilde{U}_\theta||_{\mathbf{M}_1}:= ||(1+x)^{-1+\epsilon}\tilde{U}_\theta||_{L^\infty(-1,1)} + ||(1+x)^{\epsilon}\tilde{U}_\theta'||_{L^\infty(-1,1)} + ||\tilde{U}_\theta''||_{L^\infty(0,1)}, \\
	 & ||\tilde{U}_\phi||_{\mathbf{M}_2}:=  ||\tilde{U}_\phi||_{L^\infty(-1,1)} + ||(1+x)^\varepsilon \tilde{U}_\phi'||_{L^\infty(-1,1)}  + ||(1+x)^{1+\varepsilon} \tilde{U}_\phi'' ||_{L^\infty(-1,1)} .
  \end{split}
\end{equation*}
Next, define the following function spaces:
\begin{equation*}
 \begin{split}
	& \begin{split}
      \mathbf{N}_1=\mathbf{N}_{1}(\epsilon):= & \left\{  \xi_\theta \in C( (-1, 1], \mathbb{R}) \cap C^1( (0, 1], \mathbb{R}) \mid  \xi_\theta(1)=\xi'_{\theta}(1)=\xi_{\theta}(-1)=0, \right. \\ 
      & \left. ||(1+x)^{-1+\epsilon}\xi_\theta||_{L^\infty(-1,1)} < \infty, ||\frac{\xi_\theta'}{1-x}||_{L^\infty(0,1)} < \infty \right\},
      \end{split} \\
	& \mathbf{N}_2=\mathbf{N}_2(\epsilon):= \left\{  \xi_\phi \in C( (-1, 1], \mathbb{R}) \mid  \xi_\phi(1)=0, ||\frac{(1+x)^\varepsilon \xi_\phi}{1-x}||_{L^\infty(-1,1)} < \infty  \right\}
  \end{split}
\end{equation*}
with the following norms accordingly:
\begin{equation*}
  \begin{split}
	& ||\xi_\theta||_{\mathbf{N}_1}:= ||(1+x)^{-1+\epsilon}\xi_\theta||_{L^\infty(-1,1)} + ||\frac{\xi_\theta'}{1-x}||_{L^\infty(0,1)}, \\
	& ||\xi_\phi||_{\mathbf{N}_2}:= ||\frac{(1+x)^\varepsilon \xi_\phi}{1-x}||_{L^\infty(-1,1)}.
  \end{split}
\end{equation*}
Let $\mathbf{X}:= \{ \tilde{U} = (\tilde{U}_\theta, \tilde{U}_\phi) \mid \tilde{U}_\theta\in \mathbf{M}_1, \tilde{U}_\phi\in \mathbf{M}_2\}$ with the norm $||\tilde{U}||_\mathbf{X}:= ||\tilde{U}_\theta||_{\mathbf{M}_1} + ||\tilde{U}_\phi||_{\mathbf{M}_2}$, and $\mathbf{Y}:= \{ \xi = (\xi_\theta, \xi_\phi) \mid \xi_\theta\in \mathbf{N}_1, \xi_\phi\in \mathbf{N}_2 \}$ with the norm $||\xi||_\mathbf{Y}:= ||\xi_\theta||_{\mathbf{N}_1} + ||\xi_\phi||_{\mathbf{N}_2}$. It is not difficult to verify that  $\mathbf{M}_1$, $\mathbf{M}_2$, $\mathbf{N}_1$, $\mathbf{N}_2$, $\mathbf{X}$ and $\mathbf{Y}$ are Banach spaces.

Let $l_1,l_2:\mathbf{X}\to \mathbb{R}$ be the bounded linear functionals defined by (\ref{def_l}) for each $V\in \mathbf{X}$. Define 
\begin{equation}\label{eq4_1_1x}
	\mathbf{X}_1:= \ker l_1 \cap \ker l_2.
\end{equation}

\begin{thm}\label{thm4_1}
  For every compact $ K \subset I_1$, with  $\max\{0,U^{\mu,\gamma}_{\theta}(-1)\}< 2\epsilon<2$ for every $(\mu,\gamma)\in K$,  there exist $\delta=\delta(K)>0$, and $V\in C^{\infty}(K\times B_{\delta}(0), \mathbf{X}_1)$ satisfying $V(\mu,\gamma,0,0)=0$ and $\displaystyle \frac{\partial V}{\partial \beta_i}|_{\beta=0}=0$, $i=1,2$, such that 
\begin{equation}\label{eq_thm4_1_1}
   U=U^{\mu,\gamma}+\beta_1V_{\mu,\gamma}^1+\beta_2V_{\mu,\gamma}^2+V(\mu,\gamma, \beta_1,\beta_2)
\end{equation}
satisfies equation (\ref{eq4_0})  with $\displaystyle \hat{\mu}=\mu-\frac{1}{4}\psi[U_{\phi}](-1)$.  Moreover, there exists some $\delta'=\delta'(K)>0$, such that if $||U-U^{\mu,\gamma}||_{\mathbf{X}}<\delta'$, $(\mu,\gamma)\in K$,  and $U$ satisfies  equation (\ref{eq4_0}) with some constant $\hat{\mu}$, then (\ref{eq_thm4_1_1}) holds for some $|(\beta_1,\beta_2)|<\delta$ .
\end{thm}

To prove the theorem, we first study the properties of the Banach spaces $\mathbf{X}$ and $\mathbf{Y}$ we constructed.

With the fixed $\epsilon$,  we have

\begin{lem}\label{lemP4_1}
For every $\tilde{U}\in\mathbf{X}$, it satisfies the following
	\begin{equation}\label{eq4_1}
		|\tilde{U}_\phi(s)|  \leq  ||\tilde{U}_\phi||_{\mathbf{M}_2} (1-s), \quad \forall\; -1<s<1,
	\end{equation}
	
     \begin{equation}\label{eq4_4}
       |\tilde{U}_{\theta}(s)|\le ||\tilde{U}_{\theta}||_{\mathbf{M}_1}(1-s)(1+s)^{1-\epsilon}, \quad   \forall\; -1<s<1.
     \end{equation}
\end{lem}
\begin{proof}
For $s\in(0,1)$, there exists $y\in (s,1)$ such that
	\[
		|\tilde{U}_\phi(s)|  =  |\tilde{U}_\phi(s) - \tilde{U}_\phi(1)|  = |\tilde{U}_\phi'(y)| (1-s) \le (1-s)||\tilde{U}_\phi||_{\mathbf{M}_2} ,
	\]
	while for $s\in(-1,0]$, $|\tilde{U}_{\phi}(s)|\le ||\tilde{U}_\phi||_{\mathbf{M}_2} $. So (\ref{eq4_1}) is proved. 
	
Now we prove (\ref{eq4_4}).  For $0<s<1$, $|(1+s)^{-1+\epsilon}\tilde{U}_{\theta}(s)|\le |\tilde{U}_{\theta}(s)|=|\tilde{U}_{\theta}(s)-\tilde{U}_{\theta}(1)|\le ||\tilde{U}'_{\theta}||_{L^{\infty}(0,1)}(1-s)\le ||\tilde{U}_{\theta}||_{\mathbf{M}_1}(1-s)$, and for $-1<s\le 0$, $\left|(1+s)^{-1+\epsilon}(1-s)^{-1}\tilde{U}_{\theta}(s)\right|\le |(1+s)^{-1+\epsilon}\tilde{U}_{\theta}(s)|\le ||\tilde{U}_{\theta}||_{\mathbf{M}_1}$.\\
\end{proof}

\begin{lem}\label{lem_xithe}
	For every $\xi_\theta \in \mathbf{N}_1$,
	\begin{equation}
		|\xi_\theta(x)| \leq ||\xi_\theta||_{\mathbf{N}_1}(1+x)^{1-\epsilon}(1-x)^2, \quad \forall -1<x\leq 1.
	\end{equation}
\end{lem}
\begin{proof}
	If $\xi_{\theta}\in \mathbf{N}_1$, $\xi_{\theta}(1)=0$. So for every $0<x<1$, there exists $y\in (x,1)$ such that
	\begin{equation*}
		|(1+x)^{-1+\epsilon}\xi_\theta(x)| \le |\xi_\theta(x)| = |\xi_\theta'(y)|(1-x)  \leq  ||\xi_\theta||_{\mathbf{N}_1} (1-y)(1-x) \leq ||\xi_\theta||_{\mathbf{N}_1} (1-x)^2.
	\end{equation*}
For $-1<x\le 0$, $|(1+x)^{-1+\epsilon}\xi_\theta(x)|\le ||\xi_{\theta}||_{\mathbf{N}_1}\le ||\xi_{\theta}||_{\mathbf{N}_1}(1-x)^2$. 
\end{proof}

Near $\bar{U}=(\bar{U}_{\theta}, 0)$, we will prove the existence of a family of solutions $U(\mu, \gamma, \beta)$ in $\mathbf{X}$, $\beta=(\beta_1, \beta_2)\in\mathbb{R}^2$, which are (-1)-homogeneous, axisymmetric, with non-zero swirl when $\beta\ne 0$, and $U(\mu, \gamma,0)=\bar{U}$.\\

For $\tilde{U}_{\phi}\in \mathbf{M}_2$, let $\psi[\tilde{U}_{\phi}](x)$ be defined by (\ref{eqpsi4_1}). Define a map $G$ on $K\times \mathbf{X}$ by (\ref{G}) with $\bar{U}_{\theta}$ given by (\ref{eqbaru}).

\begin{prop}\label{lem_Gcont}
	The map $G$ is in $C^{\infty}(K\times \mathbf{X}, \mathbf{Y})$ in the sense that $G$ has  continuous Fr\'{e}chet derivatives of every order. Moreover, the Fr\'{e}chet derivative of $G$ with respect to $\tilde{U}$ at $(\mu,\gamma, \tilde{U})\in \mathbf{X}$ is given by the linear bounded operator $L^{\mu,\gamma}_{\tilde{U}}: \mathbf{X}\rightarrow \mathbf{Y}$ defined as  in (\ref{eq4_Linear}).
\end{prop}
To prove Proposition \ref{lem_Gcont}, we first prove the following lemmas:\\

\begin{lem}\label{lemP4_2}
  For every $(\mu,\gamma)\in K$,  $A(\mu,\gamma,\cdot): \mathbf{X}\to \mathbf{Y}$ defined by (\ref{eqP4_2})	
 is a well-defined bounded linear operator.
\end{lem}
\begin{proof}
  In the following, $C$ denotes a universal constant which may change from line to line. For convenience we denote $A=A(\mu,\gamma,\cdot)$ for some fixed $(\mu,\gamma)\in K$. We make use of the property of $\bar{U}_{\theta}$ that $\bar{U}_{\theta}(1)=0$ and $\bar{U}_{\theta}\in C^2(-1,1]\cap L^{\infty}(-1,1)$.
  
  $A$ is clearly linear. For every $\tilde{U}\in\mathbf{X}$, we prove that $A\tilde{U}$ defined by (\ref{eqP4_2}) is in $\mathbf{Y}$ and there exists some constant $C$ such that $||A\tilde{U}||_{\mathbf{Y}}\le C||\tilde{U}||_{\mathbf{X}}$ for all $\tilde{U}\in \mathbf{X}$.
   
By the fact that $\tilde{U}_{\theta}\in \mathbf{M}_1$ and (\ref{eq4_4}), we have
\[
  |(1+x)^{-1+\epsilon}A_{\theta}|\le (1-x)|(1+x)^{\epsilon}\tilde{U}'_{\theta}|+(2+|\bar{U}_{\theta}|)(1+x)^{-1+\epsilon}|\tilde{U}_{\theta}|\le C(1-x)||\tilde{U}_{\theta}||_{\mathbf{M}_1}.
\]
We also see from the above that $\displaystyle{\lim_{x\to 1}A_{\theta}(x)=\lim_{x\to -1}A_{\theta}(x)=0}$. By computation $A'_{\theta}=(1-x^2)\tilde{U}''_{\theta}+\bar{U}_{\theta}\tilde{U}'_{\theta}+(2+\bar{U}'_{\theta})\tilde{U}_{\theta}$.  Then, by (\ref{eq4_4}) and (\ref{eqbaru}),
\[
   \frac{|A'_{\theta}|}{1-x}\le (1+x)|\tilde{U}''_{\theta}|+\frac{|\bar{U}_{\theta}|}{1-x}|\tilde{U}'_{\theta}|+(2+|\bar{U}'_{\theta}|)\frac{|\tilde{U}_{\theta}|}{1-x}\le C||\tilde{U}_{\theta}||_{\mathbf{M}_1}, \textrm{ } 0<x<1.
\]
So we have $A_{\theta}\in \mathbf{N}_1$ and $||A_{\theta}||_{\mathbf{N}_1}\le C||\tilde{U}_{\theta}||_{\mathbf{M}_1}$.\\
Next, since $A_{\phi}=(1-x^2)\tilde{U}''_{\phi}+\bar{U}_{\theta}\tilde{U}'_{\phi}$, by the fact that $\tilde{U}_{\phi}\in \mathbf{M}_2$ and (\ref{eqbaru}) we have that
\[
   \left|\frac{(1+x)^{\epsilon}A_{\phi}}{1-x}\right| \le \frac{(1+x)^{\epsilon}}{1-x}(1-x^2)\frac{||\tilde{U}_{\phi}||_{\mathbf{M}_2}}{(1+x)^{1+\epsilon}}+\frac{(1+x)^{\epsilon}|\bar{U}_{\theta}|}{1-x}\cdot \frac{||\tilde{U}_{\phi}||_{\mathbf{M}_2}}{(1+x)^{\epsilon}}\le C||\tilde{U}_{\phi}||_{\mathbf{M}_2}.
\]
We also see from the above that $\lim_{x\to 1}A_{\phi}(x)=0$. So $A_{\phi}\in \mathbf{N}_1$, and $||A_{\phi}||_{\mathbf{N}_1}\le C||\tilde{U}_{\phi}||_{\mathbf{M}_2}$. We have proved that $A\tilde{U}\in \mathbf{Y}$ and $||A\tilde{U}||_{\mathbf{Y}}\le C||\tilde{U}||_{\mathbf{X}}$ for every $\tilde{U}\in \mathbf{X}$. The proof is finished.
\end{proof}

\begin{lem}\label{lemP4_3}
  The map $Q:\mathbf{X}\times\mathbf{X}\to \mathbf{Y}$ defined by (\ref{eqP4_3})
 is a well-defined bounded bilinear operator.
\end{lem}
\begin{proof}
   It is clear that $Q$ is a bilinear operator. For every $\tilde{U},\tilde{V}\in\mathbf{X}$, we will prove that $Q(\tilde{U},\tilde{V})$ is in $\mathbf{Y}$ and there exists some constant $C$ independent of $\tilde{U}$ and $\tilde{V}$ such that $||Q(\tilde{U},\tilde{V})||_{\mathbf{Y}}\le C||\tilde{U}||_{\mathbf{X}}||\tilde{V}||_{\mathbf{X}}$. 
   
For convenience we write
\[
  \psi(\tilde{U},\tilde{V})(x)=\int_x^1 \int_l^1 \int_t^1 \frac{2 \tilde{U}_\phi(s)\tilde{V} _\phi'(s)}{1-s^2} ds dt dl.
\]
For $\tilde{U},\tilde{V}\in \mathbf{X}$, we have, using (\ref{eq4_1}) and the fact that $\tilde{U}_{\phi}, \tilde{V}_{\phi}\in \mathbf{M}_2$, that
\begin{equation}\label{eq4_5}
    \left|\frac{ \tilde{U}_\phi(s)\tilde{V} _\phi'(s)}{1-s^2}\right| \le (1+s)^{-1-\epsilon}||\tilde{U}_{\phi}||_{\mathbf{M}_2}||\tilde{U}_{\phi}||_{\mathbf{M}_2}, \quad \forall -1<s<1.
\end{equation}
It follows that $\psi(\tilde{U},\tilde{V})(x)$ is well-defined and 
\begin{equation}\label{eq4_5new}
  |\psi(\tilde{U},\tilde{V})(x)|\le C(\epsilon)(1-x)^3||\tilde{U}_{\phi}||_{\mathbf{M}_2}||\tilde{U}_{\phi}||_{\mathbf{M}_2}, \quad \forall -1<x<1.
\end{equation}

Moreover, we have, in view of (\ref{eq4_5}), that
\begin{equation*}
  \begin{split}
    & |\psi(\tilde{U},\tilde{V})(x)-\frac{(1-x)^2}{4}\psi(\tilde{U},\tilde{V})(-1)| \\
& =\left|\psi(\tilde{U},\tilde{V})(x)-\psi(\tilde{U},\tilde{V})(-1)+\frac{(1+x)(3-x)}{4}\psi(\tilde{U},\tilde{V})(-1)\right|\\
           & =\left|-\int_{-1}^{x} \int_{l}^{1} \int_{t}^{1} \frac{2 \tilde{U}_\phi(s)\tilde{V} _\phi'(s)}{1-s^2} ds dt dl+\frac{(1+x)(3-x)}{4}\psi(\tilde{U},\tilde{V})(-1)\right|\\
           & \le C(\epsilon)(1+x)||\tilde{U}_\phi||_{\mathbf{M}_2}||\tilde{V} _{\phi}||_{\mathbf{M}_2}, \quad \forall -1<x< 1.
  \end{split}
\end{equation*}
Thus, using (\ref{eq4_5new}), we have for any $x\in(-1,1)$
\begin{equation}\label{eq4_5_1}
   |\psi(\tilde{U},\tilde{V})(x)-\frac{(1-x)^2}{4}\psi(\tilde{U},\tilde{V})(-1)| \le C(\epsilon)(1+x)(1-x)^2||\tilde{U}_\phi||_{\mathbf{M}_2}||\tilde{V} _{\phi}||_{\mathbf{M}_2}.
\end{equation}
So by (\ref{eq4_4})  and (\ref{eq4_5_1}), 
\[
  \begin{split}
  & |(1+x)^{-1+\epsilon}Q_{\theta}(x)|\\
   & \le \frac{1}{2}|(1+x)^{-1+\epsilon}\tilde{U}_{\theta}(x)||\tilde{V}_{\theta}(x)|+(1+x)^{-1+\epsilon}\left| \psi(\tilde{U},\tilde{V})(x)-\frac{(1-x)^2}{4}\psi(\tilde{U},\tilde{V})(-1)\right|\\
                  & \le \frac{1}{2}(1-x)^2||\tilde{U}_{\theta}||_{\mathbf{M}_1}||\tilde{V}_{\theta}||_{\mathbf{M}_1}+C(\epsilon)(1+x)^{\epsilon}(1-x)^{2}||\tilde{U}_{\phi}||_{\mathbf{M}_2}||\tilde{V}_{\phi}||_{\mathbf{M}_2}\\
                  & \le C(\epsilon)(1-x)^2||\tilde{U}||_{\mathbf{X}}||\tilde{V}||_{\mathbf{X}}, \quad \forall -1<x<1.
\end{split}
\]
 From this we also have $\displaystyle \lim_{x\to 1}Q_{\theta}(x)= \lim_{x\to -1}Q_{\theta}(x)=0$.
 
By computation,
\[
   Q'_{\theta}(x)=\frac{1}{2}\tilde{U}_{\theta}\tilde{V}'_{\theta}+\frac{1}{2}\tilde{U}'_{\theta}\tilde{V}_{\theta}+\int_{x}^{1}\int_{t}^{1}\frac{2 \tilde{U}_\phi(s) \tilde{V}_\phi'(s)}{1-s^2} ds dt-\frac{1-x}{2}\psi(\tilde{U},\tilde{V})(-1), \textrm{ for } 0<x<1.
\]
Using $\tilde{U}\in \mathbf{X}$ ,  (\ref{eq4_1}),  (\ref{eq4_4}) and (\ref{eq4_5new}), we see that for $0<x<1$,
\[
  \begin{split}
   |Q'_{\theta}(x)| & \le \frac{1}{2}|\tilde{U}_{\theta}||\tilde{V}'_{\theta}|+\frac{1}{2}|\tilde{V}_{\theta}||\tilde{U}'_{\theta}| +\int_{x}^{1}\int_{t}^{1}\frac{2 |\tilde{U}_\phi(s)| |\tilde{V}_\phi'(s)|}{1-s^2} ds dt+\frac{|\psi(\tilde{U},\tilde{V})(-1)|}{2}(1-x)\\
                  & \le C(1-x)||\tilde{U}_{\theta}||_{\mathbf{M}_1}||\tilde{V}_{\theta}||_{\mathbf{M}_1}+ 2||\tilde{U}_{\phi}||_{\mathbf{M}_2}||\tilde{V}_{\phi}||_{\mathbf{M}_2}\int_x^1 \int_l^1 (1+s)^{-\epsilon-1} dt dl \\
                   & +C(\epsilon)(1-x)||\tilde{U}_{\phi}||_{\mathbf{M}_2}||\tilde{V}_{\phi}||_{\mathbf{M}_2}\\
    & \le C(1-x)||\tilde{U}_{\theta}||_{\mathbf{M}_1}||\tilde{V}_{\theta}||_{\mathbf{M}_1}+C(\epsilon)(1-x)||\tilde{U}_{\phi}||_{\mathbf{M}_2}||\tilde{V}_{\phi}||_{\mathbf{M}_2}\\
    &\le C(\epsilon)(1-x)||\tilde{U}||_{\mathbf{X}}||\tilde{V}||_{\mathbf{X}}.
  \end{split}
\]
So there is $Q_{\theta}\in\mathbf{N}_1$, and $||Q_{\theta}||_{\mathbf{N}_1}\le C(\epsilon)||\tilde{U}||_{\mathbf{X}}||\tilde{V}||_{\mathbf{X}}$.\\
Next, since $Q_{\phi}(x)=\tilde{U}_{\theta}(x)\tilde{V}'_{\phi}(x)$, for $-1<x<1$, 
\[
   \left|\frac{(1+x)^{\epsilon}Q_{\phi}(x)}{1-x}\right|  \le \frac{(1+x)^{\epsilon}}{1-x}|\tilde{U}_{\theta}(x)|\frac{||\tilde{V}_{\phi}||_{\mathbf{M}_2}}{(1+x)^{\epsilon}} \le 2||\tilde{U}_{\theta}||_{\mathbf{M}_1}||\tilde{V}_{\phi}||_{\mathbf{M}_2}.
\]
We also see from the above that $\lim_{x\to 1}Q_{\phi}(x)=0$. So $Q_{\phi}\in \mathbf{N}_2$, and 
$$
||Q_{\phi}||_{\mathbf{N}_2}\le ||\tilde{U}_{\theta}||_{\mathbf{M}_1}||\tilde{V}_{\phi}||_{\mathbf{M}_2}.
$$
Thus we have proved that $Q(\tilde{U}, \tilde{V})\in \mathbf{Y}$ and $||Q(\tilde{U},\tilde{V})||_{\mathbf{Y}}\le C||\tilde{U}||_{\mathbf{X}}||\tilde{V}||_{\mathbf{X}}$ for all $\tilde{U}, \tilde{V}\in \mathbf{X}$. Lemma \ref{lemP4_3} is proved.
\end{proof} 

\noindent \emph{Proof of Proposition \ref{lem_Gcont}:} By definition, $G(\mu,\gamma,\tilde{U})=A(\mu,\gamma,\tilde{U})+Q(\tilde{U},\tilde{U})$ for $(\mu,\gamma,\tilde{U})\in K\times \mathbf{X}$.  Using standard theories in functional analysis, by Lemma \ref{lemP4_3} it is clear that $Q$ is $C^{\infty}$ on $I_1\times  \mathbf{X}$.  By Lemma \ref{lemP4_2}, $A(\mu,\gamma; \cdot): \mathbf{X}\to \mathbf{Y}$ is $C^{\infty}$ for each $(\mu,\gamma)\in I_1$. For all $i,j\ge 0$, $i+j\ne 0$,  we have
\[
    \partial_{\mu}^i\partial_{\gamma}^j A(\mu,\gamma,\tilde{U})= \partial_{\mu}^i\partial_{\gamma}^j U^{\mu,\gamma}_{\theta}\left(
	\begin{matrix}
		\tilde{U}_{\theta}  \\  \tilde{U}'_{\phi} 
	\end{matrix}\right).
\]  
By (\ref{eqcor3_1_2}), for each pair of integers $(i,j)$ where $i,j\ge 0$, $i+j\ne 0$, there exists some constant $C=C(i,j,K)$, depending only on $i,j,K$, such that 
\begin{equation}\label{eq_prop4_1_1}
   |\partial_{\mu}^i\partial_{\gamma}^j U^{\mu,\gamma}_{\theta}(x)|\le C(i,j,K)(1-x), \quad -1<x<1.
\end{equation}
From (\ref{eqbaru}) we can also obtain
\[
     \left|\frac{d}{dx}\partial_{\mu}^i\partial_{\gamma}^j U^{\mu,\gamma}_{\theta}(x)\right|\le C(i,j,K), \quad 0<x<1.
\]
Using the above estimates and the fact that $\tilde{U}_{\theta}\in \mathbf{M}_1$, we have
\[
   |(1+x)^{-1+\epsilon}\partial_{\mu}^i\partial_{\gamma}^j A_{\theta}(\mu,\gamma,\tilde{U})|\le C(i,j,K)(1-x)||\tilde{U}_{\theta}||_{\mathbf{M}_1}, \quad -1<x<1,
\]
and 
\[
\begin{split}
   \left|\frac{d}{dx}\partial_{\mu}^i\partial_{\gamma}^j A_{\theta}(\mu,\gamma,\tilde{U})\right|& \le \left|\frac{d}{dx}\partial_{\mu}^i\partial_{\gamma}^j U^{\mu,\gamma}_{\theta}(x)\right||\tilde{U}_{\theta}(x)|+|\partial_{\mu}^i\partial_{\gamma}^j U^{\mu,\gamma}_{\theta}(x)|\left|\frac{d}{dx}\tilde{U}_{\theta}(x)\right|\\
   & \le C(K)(1-x)||\tilde{U}_{\theta}||_{\mathbf{M}_1}, \quad 0<x<1.
   \end{split}
\]
So $\partial_{\mu}^i\partial_{\gamma}^j A_{\theta}(\mu,\gamma,\tilde{U})\in \mathbf{N}_1$, with $||\partial_{\mu}^i\partial_{\gamma}^j A_{\theta}(\mu,\gamma,\tilde{U})||_{\mathbf{N}_1}\le C(i,j,K)||\tilde{U}_{\theta}||_{\mathbf{M}_1}$ for all $(\mu,\gamma,\tilde{U})\in K\times \mathbf{X}$.

Next, by (\ref{eq_prop4_1_1}) and the fact that $\tilde{U}_{\phi}\in \mathbf{M}_1$, we have
\[
   \frac{(1+x)^{\epsilon}}{1-x}|\partial_{\mu}^i\partial_{\gamma}^j A_{\phi}(\mu,\gamma,\tilde{U})(x)|=\frac{|\partial_{\mu}^i\partial_{\gamma}^j U^{\mu,\gamma}_{\theta}(x)|}{1-x}|(1+x)^{\epsilon}U'_{\phi}|\le C(i,j,K)||\tilde{U}_{\phi}||_{\mathbf{M}_2}.
\]
So $\partial_{\mu}^i\partial_{\gamma}^j A_{\phi}(\mu,\gamma,\tilde{U})\in \mathbf{N}_2$, with $||\partial_{\mu}^i\partial_{\gamma}^j A_{\phi}(\mu,\gamma,\tilde{U})||_{\mathbf{N}_2}\le C(i,j,K)||\tilde{U}_{\phi}||_{\mathbf{M}_2}$ for all $(\mu,\gamma,\tilde{U})\in K\times \mathbf{X}$. Thus $\partial_{\mu}^i\partial_{\gamma}^jA(\mu,\gamma,\tilde{U})\in \mathbf{Y}$, with $||\partial_{\mu}^i\partial_{\gamma}^j A(\mu,\gamma,\tilde{U})||_{\mathbf{Y}}\le C(i,j,K)||\tilde{U}||_{\mathbf{X}}$ for all $(\mu,\gamma,\tilde{U})\in K\times \mathbf{X}$, $i,j\ge 0$, $i+j\ne 0$. 

So for each $(\mu,\gamma)\in K$, $\partial_{\mu}^i\partial_{\gamma}^j A(\mu,\gamma;  \cdot):\mathbf{X}\to \mathbf{Y}$ is a bounded linear map with uniform bounded norm on $K$. Then by standard theories in functional analysis, $A:K\times\mathbf{X}\to \mathbf{Y}$ is $C^{\infty}$. So $G$ is a $C^{\infty}$ map from $K\times\mathbf{X}$ to $Y$. By direct calculation we get its Fr\'{e}chet derivative with respect to $\mathbf{X}$ is given by  the linear bounded operator $L^{\mu,\gamma}_{\tilde{U}}: \mathbf{X}\rightarrow \mathbf{Y}$ defined as  (\ref{eq4_Linear}). The proof is finished.     \qed

By Proposition \ref{lem_Gcont}, $L^{\mu,\gamma}_0:\mathbf{X} \to \mathbf{Y}$, the Fr\'{e}chet derivative of $G$ with respect to $\tilde{U}$ at $\tilde{U}=0$,  is given by (\ref{eq_LinearAtZero}).

Let $a_{\mu,\gamma}(x), b_{\mu,\gamma}(x)$ be the functions defined by (\ref{eq_ab}) with $\bar{U}_{\theta}$ given by (\ref{eqbaru}). For $\xi = (\xi_\theta, \xi_\phi) \in \mathbf{Y}$, let the map $W^{\mu,\gamma}$ be defined as $W^{\mu,\gamma}(\xi) := (W^{\mu,\gamma}_\theta(\xi), W^{\mu,\gamma}_\phi(\xi))$,  where
\begin{equation}\label{eq4_1_6}
  \begin{split}
   &W^{\mu,\gamma}_\theta(\xi)(x) := e^{-a_{\mu,\gamma}(x)}\int_{0}^{x}e^{a_{\mu,\gamma}(s)}\frac{\xi_{\theta}(s)}{1-s^2}ds,\\
    & W^{\mu,\gamma}_\phi(\xi)(x) :=\int_{x}^{1}e^{-b_{\mu,\gamma}(t)}\int_{t}^{1}e^{b_{\mu,\gamma}(s)}\frac{\xi_{\phi}(s)}{1-s^2}dsdt.
  \end{split}
\end{equation}
A calculation gives
\begin{equation}\label{eq4_W_d}
    (W^{\mu,\gamma}_{\theta}(\xi))'(x)=-a'(x)W^{\mu,\gamma}_{\theta}(x)+\frac{\xi_{\theta}(x)}{1-x^2}.
\end{equation}
\begin{lem}\label{lem4_1}
	For every $(\mu,\gamma)\in K$, $W^{\mu,\gamma}: \mathbf{Y}\rightarrow\mathbf{X}$ is  continuous, and is a right inverse of $L^{\mu,\gamma}_{0}$.
\end{lem}
\begin{proof}
We make use of the property that $\bar{U}_{\theta}(1)=0$, $\bar{U}_{\theta}\in C^2(-1,1]\cap C^0[-1,1]$ and $\bar{U}_{\theta}(-1)<2\epsilon<2$.  For convenience let us write  $W:=W^{\mu,\gamma}(\xi)$ for $\xi\in \mathbf{Y}$,  $a(x)=a_{\mu,\gamma}(x)$ and $b(x)=b_{\mu,\gamma}(x)$.

We first prove $W$ is well-defined.  Applying  Lemma \ref{lem_xithe} in the expression of $W_{\theta}$ in (\ref{eq4_1_6}),
	\begin{equation}\label{eq4_1_7}
		|(1+x)^{-1+\epsilon}W_{\theta}(x)|\le (1+x)^{-1+\epsilon}||\xi_{\theta}||_{\mathbf{N}_1}e^{-a(x)}\int_{0}^{x}e^{a(s)}(1-s)(1+s)^{-\epsilon}ds, \quad -1<x<1.
	\end{equation}
	
	We make estimates first for $0<x\le 1$ and then for $-1<x\le 0$.
	
	\textbf{Case} $1$. $0<x\le 1$.
	
      Since $\bar{U}_{\theta}(x)= -\bar{U}_\theta'(1)(1-x) +O((1-x)^{2})$,
      \begin{equation}\label{eq4_1_7_1}
        b(x)=b(1)+\int_{1}^{x}\frac{ \bar{U}_{\theta}}{ 1-s^2}ds=b(1)+\frac{1}{2}\bar{U}_\theta'(1) (1-x)  + O(1)(1-x)^2, \quad 0<x\le 1,
      \end{equation}
      where $b(1):=\int_{0}^{1}\frac{ \bar{U}_{\theta}}{ 1-s^2}ds$ exists and is finite, we have
	\begin{equation}\label{eq4_1_7_2}
		e^{b(x)} = e^{b(1)}\left[1+\frac{1}{2}\bar{U}_\theta'(1) (1-x)  + O(1)(1-x)^2\right].
	\end{equation}
Notice $a(x)=-\ln(1-x^2)+b(x)$, so
	\begin{equation}\label{eq4_1_7_3}
		e^{a(x)} 
                      =\frac{ e^{b(1)}}{2(1-x)}\left( 1+\frac{1}{2}\left( \bar{U}_\theta'(1)+1\right) (1-x)  + O(1)(1-x)^2 \right).
	\end{equation}

	 Then in (\ref{eq4_1_7}), using the estimate of $a(x)$ and $e^{a(x)}$, it is not hard to see that there exists some positive constant $C$ such that
\[
   e^{a(s)}(1-s)(1+s)^{-\epsilon}\le C, \quad e^{a(x)}\ge \frac{1}{C(1-x)},\quad 0<s<x<1.
\]
Thus
	\begin{equation}\label{eq4_1_W_1}
		|W_{\theta}(x)| \le  C||\xi_{\theta}||_{\mathbf{N}_1} (1-x), \quad 0<x\le 1.
	\end{equation}
 In particular, $W_{\theta}(1)=0$.
 
      By (\ref{eq4_W_a}) and (\ref{eq4_W_d}), for $0<x<1$,
\begin{equation}\label{eq4_1_W_2}
   |a'(x)|
   \le  \frac{C}{(1-x)},\quad |a''(x)|\le \frac{C}{(1-x)^2},
\end{equation}
	\begin{equation*}
		|W'_{\theta}(x)|  \le |a'(x)||W_{\theta}(x)|+\frac{|\xi_{\theta}(x)|}{1-x}\le C ||\xi_{\theta}||_{\mathbf{N}_1},\quad 0<x\le 1,
	\end{equation*}
where we have used (\ref{eq4_1_W_1}),  (\ref{eq4_1_W_2}), the fact that $\xi\in \mathbf{N}_1$, and  Lemma \ref{lem_xithe}. Next,

\[
  \begin{split}
  W''_{\theta}(x) & =-a''(x)W_{\theta}-a'(x)W'_{\theta}(x)+\left(\frac{\xi_{\theta}(x)}{1-x^2}\right)'\\
                  & =((a'(x))^2-a''(x))W_{\theta}(x)-a'(x)\frac{\xi_{\theta}(x)}{1-x^2}+\frac{\xi'_{\theta}(x)}{1-x^2}+\frac{2x\xi_{\theta}(x)}{(1-x^2)^2}.
  \end{split}
\]
Thus
	\begin{equation*}
	\begin{split}
		|W''_{\theta}(x)|  \le  |(a'(x))^2-a''(x)| |W_{\theta}| + |a'(x)|\frac{|\xi_{\theta}|}{1-x^2} + \frac{|\xi'_{\theta}|}{(1-x)}+\frac{|\xi_{\theta}|}{(1-x)^2}.
	\end{split}
	\end{equation*}
By computation
	\[
	    (a'(x))^2-a''(x)=\frac{\bar{U}^2_{\theta}+2x\bar{U}_{\theta}}{(1-x^2)^2}-\frac{2+\bar{U}'_{\theta}}{1-x^2}=O\left(\frac{1}{1-x}\right).
	\]
It follows, using (\ref{eq4_1_W_1}),  (\ref{eq4_1_W_2}) and  Lemma \ref{lem_xithe}, that
\[
   |W''_{\theta}(x)|\le  C\left(\frac{|W_{\theta}(x)|}{1-x}+\frac{|\xi_{\theta}|}{(1-x)^2}+\frac{|\xi'_{\theta}|}{1-x}\right)\le C||\xi_{\theta}||_{\mathbf{N}_1}, \quad 0<x<1.
\]

\textbf{Case} $2$. $-1<x\le 0$.

	Recall that $\bar{U}_{\theta}(-1)< 2\epsilon<2$. Moreover,  $\bar{U}_{\theta}(x)=\bar{U}_{\theta}(-1)+O((1+x)^{b})$ with $b=\sqrt{1+2\mu}$. Then we have , for $-1<x\le 0$, that
\begin{equation*}
     b(x)=\frac{\bar{U}_{\theta}(-1)}{2}\ln (1+x)+O(1), \quad a(x)=(\frac{\bar{U}_{\theta}(-1)}{2}-1)\ln (1+x)+O(1),
\end{equation*}
\begin{equation*}
   e^{a(x)}=(1+x)^{\frac{\bar{U}_{\theta}(-1)}{2}-1}e^{O(1)},\quad  e^{-a(x)}=(1+x)^{1-\frac{\bar{U}_{\theta}(-1)}{2}}e^{O(1)}.
\end{equation*}
So there exists some constant $C$ such that
\[
    e^{a(s)}(1-s)(1+s)^{-\epsilon}\le C(1+s)^{\frac{\bar{U}_{\theta}(-1)}{2}-1-\epsilon},\quad e^{-a(s)}\le C(1+s)^{1-\frac{\bar{U}_{\theta}(-1)}{2}}, \quad -1<s\le 0 .
\]
Apply these estimates in (\ref{eq4_1_7}), and use the fact that $\bar{U}_{\theta}(-1)<2\epsilon$, we have
\begin{equation}\label{eq_W_e}
   |(1+x)^{-1+\epsilon}W_{\theta}(x)|\le C\left| \left(1+x\right)^{-\frac{\bar{U}_{\theta}(-1)}{2}+\epsilon} - 1\right| ||\xi_{\theta}||_{\mathbf{N}_1}\le C||\xi_{\theta}||_{\mathbf{N}_1}, \quad -1<x\le 0.
\end{equation}
By (\ref{eq4_W_d}), (\ref{eq4_W_a}) and (\ref{eq_W_e}), we have, for $-1<x\le 0$, that
\[
   |(1+x)^{\epsilon}W'_{\theta}(x)|\le |a'(x)(1+x)^{\epsilon}W_{\theta}(x)|+\frac{|\xi_{\theta}(x)|(1+x)^{\epsilon}}{1-x^2}\le C||\xi_{\theta}||_{\mathbf{N}_1}.
\]
So we have shown that $W_{\theta}\in \mathbf{M}_1$, and $||W_{\theta}||_{\mathbf{M}_1}\le C||\xi_{\theta}||_{\mathbf{N}_1}$ for some constant $C$.\\
By the definition of $W_{\phi}(\xi)$ in (\ref{eq4_1_6}) and the fact that $\xi_{\phi}\in \mathbf{N}_2$, we have, for every $-1<x<1$, 
\[
   |W_{\phi}(x)|\le \int_{x}^{1}e^{-b(t)}\int_{t}^{1}e^{b(s)}\frac{|\xi_{\phi}(s)|}{1-s^2}dsdt\le ||\xi_{\phi}||_{\mathbf{N}_2} \int_{x}^{1}e^{-b(t)}\int_{t}^{1}e^{b(s)}(1+s)^{-1-\epsilon}dsdt.
\]
Since $b(x)=\frac{\bar{U}_{\theta}(-1)}{2}\ln (1+x)+O(1)$ for all $-1<x<1$, there is some constant $C$ such that
\begin{equation}\label{eq4_W_b}
   e^{b(s)}\le C(1+s)^{\frac{\bar{U}_{\theta}(-1)}{2}},\quad e^{-b(t)}\le C(1+t)^{-\frac{\bar{U}_{\theta}(-1)}{2}}, \quad -1<s,t\le 1. 
\end{equation}
So we have, using $\displaystyle \frac{\bar{U}_{\theta}(-1)}{2}<\epsilon<1$, that for $-1<x\le 1$, 
\[
  \begin{split}
   |W_{\phi}(x)|&\le ||\xi_{\phi}||_{\mathbf{N}_2}\int_{x}^{1}e^{-b(t)}\int_{t}^{1}\frac{e^{b(s)}}{(1+s)^{1+\epsilon}}dsdt\\
      & \le C||\xi_{\phi}||_{\mathbf{N}_2}\int_{x}^{1}(1+t)^{-\frac{\bar{U}_{\theta}(-1)}{2}}\int_{t}^{1}(1+s)^{\frac{\bar{U}_{\theta}(-1)}{2}-1-\epsilon}dsdt\\
                 & \le C(1-x)||\xi_{\phi}||_{\mathbf{N}_2}. 
  \end{split}
\]
In particular, $W_{\phi}(1)=0$.

By computation
\[
   W'_{\phi}(x)=-e^{-b(x)}\int_{x}^{1}e^{b(s)}\frac{\xi_{\phi}(s)}{1-s^2}ds.
\]
Thus, using (\ref{eq4_W_b}) and the fact that $\xi_{\phi}\in \mathbf{N}_2$, we have for $-1<x<1$ that 
\[
  \begin{split}
   |(1+x)^{\epsilon}W'_{\phi}(x)| & \le ||\xi_{\phi}||_{\mathbf{N}_2}(1+x)^{\epsilon}e^{-b(x)}\int_{x}^{1} e^{b(s)} (1+s)^{-1-\epsilon}ds\\
                   & \le C||\xi_{\phi}||_{\mathbf{N}_2}(1+x)^{\epsilon}(1+x)^{-\frac{\bar{U}_{\theta}(-1)}{2}}\int_{x}^{1}(1+s)^{\frac{\bar{U}_{\theta}(-1)}{2}-1-\epsilon}ds\\
                   & \le C||\xi_{\phi}||_{\mathbf{N}_2}.
  \end{split}
\]
Similarly
\[
   W''_{\phi}(x)=b'(x)e^{-b(x)}\int_{x}^{1}e^{b(s)}\frac{\xi_{\phi}(s)}{1-s^2}ds+\frac{\xi_{\phi}(x)}{1-x^2}.
\]
Since $\displaystyle |b'(x)|=\left|\frac{\bar{U}_{\theta}(x)}{1-x^2}\right| \le \frac{C}{1+x}$ for all $-1<x<1$,  using (\ref{eq4_W_b}), that
\[
   |(1+x)^{1+\epsilon}W''_{\phi}(x)|\le C||\xi_{\phi}||_{\mathbf{N}_2}, \quad -1<x<1. 
\]
     So $W_{\phi}\in \mathbf{M}_2$, and $||W_{\phi}||_{\mathbf{M}_2}\le C ||\xi_{\phi}||_{\mathbf{N}_2}$ for some constant $C$.
     
Then $W^{\mu,\gamma}(\xi)\in\mathbf{X}$ for all $\xi\in\mathbf{Y}$, and $||W^{\mu,\gamma}(\xi)||_{\mathbf{X}}\le C||\xi||_{\mathbf{Y}}$ for some constant $C$. So $W^{\mu,\gamma}:\mathbf{Y}\rightarrow\mathbf{X}$ is well-defined and continuous. It can be checked directly that $W^{\mu,\gamma}$ is a right inverse of $L^{\mu,\gamma}_{0}$.
\end{proof}

Let $V_{\mu,\gamma}^1, V_{\mu,\gamma}^2, V_{\mu,\gamma}^3$ be vectors defined by (\ref{eq4_2_ker}), we have
\begin{lem}\label{lem4_3}
    $\{V_{\mu,\gamma}^1,V_{\mu,\gamma}^2\}$ is a basis of the kernel of $L^{\mu,\gamma}_{0}:\mathbf{X}\to \mathbf{Y}$.
\end{lem}
\begin{proof}
  Let $V\in\mathbf{X}$, $L^{\mu,\gamma}_{0}V=0$. We know that $V$ is given by (\ref{eq4_3_ker}) for some $c_1,c_2,c_3\in\mathbb{R}$. Since $\bar{U}_{\theta}(-1)<2$, it is not hard to verify that $V_{\mu,\gamma}^1, V_{\mu,\gamma}^2\in \mathbf{X}$, and  $ V_{\mu,\gamma}^3\notin\mathbf{X}$. Since $V\in \mathbf{X}$, we must have $c_3V_{\mu,\gamma}^3\in \mathbf{X}$, so $c_3=0$, and  $V\in \mathrm{span}\{V_{\mu,\gamma}^1,V_{\mu,\gamma}^2\}$. It is clear that $\{V_{\mu,\gamma}^1,V_{\mu,\gamma}^2\}$ is independent. So $\{V_{\mu,\gamma}^1,V_{\mu,\gamma}^2\}$ is a basis of the kernel.
\end{proof}

\begin{cor}\label{cor4_1}
  For any $\xi\in\mathbf{Y}$, all solutions of $L^{\mu,\gamma}_{0}V=\xi$, $V\in\mathbf{X}$, are given by
\begin{equation*}
   V=W^{\mu,\gamma}(\xi)+c_1V_{\mu,\gamma}^1+c_2V_{\mu,\gamma}^2,\quad c_1,c_2\in\mathbb{R}.
\end{equation*}
Namely,
\begin{equation*}
		V_\theta = W^{\mu,\gamma}_\theta(\xi)+ c_1e^{-a_{\mu,\gamma}(x)}, \quad V_\phi = W^{\mu,\gamma}_\phi(\xi)+ c_2 \int_{x}^{1}e^{-b_{\mu,\gamma}(t)}dt,\textrm{ }c_1,c_2\in\mathbb{R}.
	\end{equation*}
\end{cor}
\begin{proof}
  By Lemma \ref{lem4_1}, $V-W^{\mu,\gamma}(\xi)$ is in the kernel of $L_{0}:\mathbf{X}\to \mathbf{Y}$. The conclusion then follows from Lemma \ref{lem4_3}.
\end{proof}

Let $l_1,l_2$ be the functionals on $\mathbf{X}$ defined by (\ref{def_l}), and $\mathbf{X}_1$ be the subspace of $\mathbf{X}$ defined by (\ref{eq4_1_1x}). As shown in Section 4.1, the matrix $(l_i(V^{j}_{\mu,\gamma}))$ is a diagonal invertible matrix, for every $(\mu,\gamma)\in K$. So $\mathbf{X}_1(\mu,\gamma)$ is a closed subspace of $\mathbf{X}$, and
\begin{equation}\label{eq4_1xb} 
	\mathbf{X} = \mbox{span} \{ V_{\mu,\gamma}^{1}, V_{\mu,\gamma}^{2} \} \oplus\mathbf{X}_1(\mu,\gamma),\quad \forall (\mu,\gamma)\in K, 
\end{equation}
with the projection operator $P(\mu,\gamma): \mathbf{X}\rightarrow\mathbf{X}_1$ given by 
\begin{equation*}
    P(\mu,\gamma)V = V-  l_1(V)V_{\mu,\gamma}^{1}-c(\mu,\gamma)l_2(V)V_{\mu,\gamma}^{2} \textrm{ for } V\in \mathbf{X}.
    \end{equation*}
where $c(\mu,\gamma)=\left(\int_{0}^{1}e^{-b_{\mu,\gamma}(t)}dt\right)^{-1}>0$ for all $(\mu,\gamma)\in K$.

\begin{lem}\label{lem4_2}
	For each $(\mu,\gamma)\in K$, the operator $ L^{\mu,\gamma}_{0}: \mathbf{X}_1\rightarrow\mathbf{Y}$ is an isomorphism.
\end{lem}

\begin{proof}
	By Corollary \ref{cor4_1} and Lemma \ref{lem4_3}, $L^{\mu,\gamma}_{0}:\mathbf{X}\rightarrow\mathbf{Y}$ is surjective and $\ker L^{\mu,\gamma}_{0} = $ span $\{V_{\mu,\gamma}^{1}, V_{\mu,\gamma}^{2}\}$. The conclusion of the lemma then follows in view of the direct sum property (\ref{eq4_1xb}).	
\end{proof}

\begin{lem}\label{lemV1V2}
 $V_{\mu,\gamma}^1,V_{\mu,\gamma}^2\in C^{\infty}(K,\mathbf{X})$.
\end{lem}
\begin{proof}
  We know $2\bar{\epsilon}:=\max\{U^{\mu,\gamma}_{\theta}(-1)|(\mu,\gamma)\in K\}<2\epsilon$. For convenience in this proof let us denote $a(x)=a_{\mu,\gamma}(x)$, $b(x)=b_{\mu,\gamma}(x)$ and $V^i=V^i_{\mu,\gamma}$, $i=1,2$.
  
  By computation, using the explicit expression of $U^{\mu,\gamma}_{\theta}(x), a(x), a'(x), b(x), V^1_{\theta}(x)$ and $V^2_{\phi}(x)$ given by (\ref{eqbaru}), (\ref{eq_ab}), (\ref{eq4_W_a}) and (\ref{eq4_2_ker}), and the estimate of  $\partial^i_{\mu}\partial^j_{\gamma}U^{\mu,\gamma}_{\theta}$ in (\ref{eqcor3_1_2}) for all  $i,j\ge 0$. we have, for $(\mu,\gamma)\in K$, that
 \[
    e^{-a(x)}=O(1)(1+x)^{1-\frac{U^{\mu,\gamma}_{\theta}(-1)}{2}}, \quad e^{-b(x)}=O(1)(1+x)^{-\frac{U^{\mu,\gamma}_{\theta}(-1)}{2}}, \quad  -1<x\le 0.
 \]
 So
\[
    \left|V^1_{\theta}(x)\right|= O(1)(1+x)^{1-\frac{U^{\mu,\gamma}_{\theta}(-1)}{2}}=O(1)(1+x)^{1-\bar{\epsilon}}, \quad V^2_{\phi}(x)=O(1), \quad -1<x\le 0,
\]
and 
\[
   \left|\frac{d}{dx}V^1_{\theta}(x)\right|=\left|e^{-a(x)}a'(x)\right|= O(1)(1+x)^{-\frac{U^{\mu,\gamma}_{\theta}(-1)}{2}}=O(1)(1+x)^{-\bar{\epsilon}}, \quad -1<x\le 0,
\]
\[
   \left|\frac{d}{dx}V^2_{\phi}(x)\right|=e^{-b(x)}= O(1)(1+x)^{-\frac{U^{\mu,\gamma}_{\theta}(-1)}{2}}=O(1)(1+x)^{-\bar{\epsilon}}, \quad -1<x\le 0.
\]
Moreover, 
\[
  \begin{split}
  & \frac{\partial^i}{\partial \mu^i}a(x)=\frac{\partial^i}{\partial \mu^i}b(x)=\int_{0}^{x}\frac{1}{1-s^2}\frac{\partial^i}{\partial \mu^i}U^{\mu,\gamma}(s)ds\\
  &=-\left(\frac{d^i}{d \mu^i} \sqrt{1+2\mu}\right)\ln(1+x)+O(1)\int_{0}^{x}(1+s)^{b-1}\left|\ln(1+s)\right|^i ds\\
  &=-\left(\frac{d^i}{d\mu^i}\sqrt{1+2\mu}\right)\ln(1+x)+O(1)(1+x)^{b}\left|\ln(1+x)\right|^i, 
   \end{split}
\]
where $|O(1)|\le C$ depending only on $K$ and $i$.
So we have
\[
   \left|\partial^i_{\mu}V^1_{\theta}(x)\right|=e^{-a(x)}O\left(\left|\ln(1+x)\right|^i\right)=O(1)(1+x)^{1-\bar{\epsilon}}\left|\ln(1+x)\right|^i, \quad -1<x\le 0, i=1,2,3...
\]
Similarly,
\[
   \left|\partial^j_{\gamma}\partial^i_{\mu}V^1_{\theta}(x)\right|=e^{-a(x)}O\left((1+x)^b\left|\ln(1+x)\right|^i\right)=O(1)(1+x)^{1-\bar{\epsilon}}, \quad -1<x\le 0, i=1,2,3...
\]
From the above we can see that for all $(\mu,\gamma)\in K$ and $i,j\ge 0$, there exists some constant $C=C(i,j,K)$, such that
\[
    \left|(1+x)^{-1+\epsilon}\partial^j_{\gamma}\partial^i_{\mu}V^1_{\theta}(x)\right|\le C,\quad \left|(1+x)^{\epsilon}\frac{d}{dx}\partial^j_{\gamma}\partial^i_{\mu}V^1_{\theta}(x)\right|\le C,\quad -1<x\le 0.
\]
We can also show that for $i,j \ge 0$, 
\[
   \partial^j_{\gamma}\partial^i_{\mu}V^1_{\theta}(1)=0,
\]
 and there exists some constant $C$ such that
 \[
    \left|\frac{d^l}{dx^l}\partial^j_{\gamma}\partial^i_{\mu}V^1_{\theta}(x)\right|\le C, \quad l=0,1,2, \quad 0\le x<1.
 \]
 The above imply that  for all $i,j\ge 0$, $\partial^j_{\gamma}\partial^i_{\mu}V^1(x)\in \mathbf{X}$, and $V^1_{\theta}\in C^{\infty}(K,M_1)$.
 
 Similarly, we can show that $V^2_{\phi}\in C^{\infty}(K,M_2)$. So $V^1,V^2\in C^{\infty}(K,\mathbf{X})$.
\end{proof}

\begin{lem}\label{lem4_1_10}
  There exists $C=C(K)>0$ such that for all $(\mu,\gamma)\in K$,  $(\beta_1,\beta_2)\in \mathbb{R}^2$, and $V\in \mathbf{X}_1$, 
  \[
     ||V||_{\mathbf{X}}+|(\beta_1,\beta_2)|\le C||\beta_1V^1_{\mu,\gamma}+\beta_2V^2_{\mu,\gamma}+V||_{\mathbf{X}}.
  \] 
\end{lem}
\begin{proof}
   We prove the lemma by contradiction. Assume there exist a sequence $(\mu^i,\gamma^i)\in K$, and $(\beta_1^i,\beta_2^i)\in \mathbb{R}^2$, $V^i\in \mathbf{X}_1$, such that
   \begin{equation}\label{eq4_1_10_1}
      ||V^i||_{\mathbf{X}}+|(\beta^i_1,\beta^i_2)|\ge i||\beta^i_1V^1_{\mu^i,\gamma^i}+\beta^i_2V^2_{\mu^i,\gamma^i}+V^i||_{\mathbf{X}}.
   \end{equation}
   Without loss of generality we can assume that
   \[
      ||V^i||_{\mathbf{X}}+|(\beta^i_1,\beta^i_2)|=1.
   \] 
   Since $K$ is compact,  there exists a subsequence of $(\mu^i,\gamma^i)$, we still denote it as $(\mu^i,\gamma^i)$ and some $(\mu,\gamma)\in K$ such that $(\mu^i,\gamma^i)\to (\mu,\gamma)\in K$ as $i\to \infty$. Similarly, since $|(\beta^i_1,\beta^i_2)|\le 1$, there exists some subsequence, still denote as $(\beta^i_1,\beta^i_2)$, such that $(\beta^i_1,\beta^i_2)\to (\beta_1,\beta_2)\in \mathbb{R}^2$ as $i \to \infty$. By Lemma \ref{lemV1V2} we have
   \[
      V^j_{\mu^i,\gamma^i}\to V^j_{\mu,\gamma}, \quad j=1,2. 
   \]
By (\ref{eq4_1_10_1}),
   \[
      \beta^i_1V^1_{\mu^i,\gamma^i}+\beta^i_2V^2_{\mu^i,\gamma^i}+V^i \to 0.
   \]

 This implies
 \[
     V^i\to V:=-( \beta_1V^1_{\mu,\gamma}+\beta_2V^2_{\mu,\gamma}).
 \]
 On the other hand, $V^i\in \mathbf{X}_1$. Since $\mathbf{X}_1$ is a closed subspace of $\mathbf{X}$,  we have $V\in \mathbf{X}_1$.
 Thus $V\in \mathbf{X}_1\cap \mbox{span} \{ V_{\mu,\gamma}^{1}, V_{\mu,\gamma}^{2} \}$. So $V=0$. 
 
 Since $V^1_{\mu,\gamma}, V^2_{\mu,\gamma}$ are independent for any $(\mu,\gamma)\in K$. We have $\beta_1=\beta_2=0$. However, $||V^i||_{\mathbf{X}}+|(\beta^i_1,\beta^i_2)|=1$ leads to $||V||_{\mathbf{X}}+|(\beta_1,\beta_2)|=1$, contradiction. The lemma is proved.
\end{proof}

\noindent \emph{Proof of Theorem \ref{thm4_1}:} Define a map $F: K\times\mathbb{R}^2\times \mathbf{X}_1\to \mathbf{Y}$ by
\[
    F(\mu,\gamma,\beta_1,\beta_2,V)=G(\mu,\gamma,\beta_1V^1_{\mu,\gamma}+\beta_2V^2_{\mu,\gamma}+V).
\]
By Proposition \ref{lem_Gcont}, $G$ is a $C^{\infty}$ map from $K\times \mathbf{X}$ to $\mathbf{Y}$. Let $\tilde{U}=\tilde{U}(\mu,\gamma,\beta_1,\beta_2,\bar{V})=\beta_1V^1_{\mu,\gamma}+\beta_2V^2_{\mu,\gamma}+V$. Using Lemma \ref{lemV1V2}, we have $\tilde{U}\in C^{\infty}(K\times\mathbb{R}^2\times \mathbf{X}_1, \mathbf{X})$. So it concludes that $F\in C^{\infty}(K\times\mathbb{R}^2\times \mathbf{X}_1,\mathbf{Y})$. 

Next, by definition $F(\mu,\gamma,0,0,0)=0$ for all $(\mu,\gamma)\in K$. Fix some $(\bar{\mu}, \bar{\gamma})\in K$, using Lemma \ref{lem4_2}, we have $F_{V}(\bar{\mu}, \bar{\gamma}, 0,0,0)=L_0^{\bar{\mu}, \bar{\gamma}}:\mathbf{X}_1\to \mathbf{Y}$ is an isomorphism.

Applying Theorem C, there exist some $\delta>0$ and a unique $V\in C^{\infty}(B_{\delta}(\bar{\mu}, \bar{\gamma})\times B_{\delta}(0), \mathbf{X}_1 )$, such that
\[
    F(\mu,\gamma,\beta_1,\beta_2,V(\mu,\gamma,\beta_1,\beta_2))=0, \quad \forall (\mu,\gamma)\in B_{\delta}(\bar{\mu}, \bar{\gamma}), (\beta_1,\beta_2)\in B_{\delta}(0),
\]
and 
\[
   V(\bar{\mu}, \bar{\gamma},0,0)=0.
\]
The uniqueness part of Theorem C holds in the sense that there exists some $0<\bar{\delta}<\delta$,  such that $B_{\bar{\delta}}(\bar{\mu},\bar{\gamma},0,0,0)\cap F^{-1}(0) \subset  \{(\mu,\gamma,\beta_1,\beta_2,V(\mu,\gamma,\beta_1,\beta_2))|(\mu,\gamma)\in B_{\delta}(\bar{\mu}, \bar{\gamma}), \beta\in B_{\delta}(0)\}$.

\textbf{Claim}: there exists some $0<\delta_1<\frac{\bar{\delta}}{2}$, such that $V(\mu,\gamma,0,0)=0$ for every $(\mu,\gamma)\in B_{\delta_1}(\bar{\mu},\bar{\gamma})$.

\emph{Proof of the claim:}
   Since $V(\bar{\mu},\bar{\gamma},0,0)=0$ and $V(\mu,\gamma,0,0)$ is continuous in $(\mu,\gamma)$, there exists some $0<\delta_1<\frac{\bar{\delta}}{2}$, such that for all $(\mu,\gamma)\in B_{\delta_1}(\bar{\mu},\bar{\gamma})$, $(\mu,\gamma, 0,0, V(\mu,\gamma,0,0))\in B_{\bar{\delta}(\bar{\mu},\bar{\gamma},0,0,0)}$. We know that for all $(\mu,\gamma)\in B_{\delta_1}(\bar{\mu},\bar{\gamma})$,
   \[
      F(\mu,\gamma, 0,0,0)=0,  
      \]
       and 
       \[
          F(\mu,\gamma, 0,0, V(\mu,\gamma,0,0))=0.
       \]
By the above mentioned uniqueness result, $V(\mu,\gamma,0,0)=0$, for every $(\mu,\gamma)\in B_{\delta_1}(\bar{\mu},\bar{\gamma})$.

Now we have $V\in C^{\infty}(B_{\delta_1}(\bar{\mu}, \bar{\gamma})\times B_{\delta_1}(0), \mathbf{X}_1(\bar{\mu}, \bar{\gamma}) )$, and 
\[
     F(\mu,\gamma,\beta_1,\beta_2,V(\mu,\gamma,\beta_1,\beta_2))=0, \quad \forall (\mu,\gamma)\in B_{\delta_1}(\bar{\mu}, \bar{\gamma}), (\beta_1,\beta_2)\in B_{\delta_1}(0).
\]
i.e.
\[
    G(\mu,\gamma, \beta_1V^1_{\mu,\gamma}+\beta_2V^2_{\mu,\gamma}+V(\mu,\gamma,\beta_1,\beta_2) )=0, \quad \forall (\mu,\gamma)\in B_{\delta_1}(\bar{\mu}, \bar{\gamma}), (\beta_1,\beta_2)\in B_{\delta_1}(0).
\]
Take derivative of the above with respect to $\beta_i$ at $(\mu,\gamma, 0)$, $i=1,2$, we have
\[
   G_{\tilde{U}}(\mu,\gamma,0)(V^i_{\mu,\gamma}+\partial_{\beta_i}V(\mu,\gamma,0,0))=0.
\]
Since $G_{\tilde{U}}(\mu,\gamma,0)V^i_{\mu,\gamma}=0$ by Lemma \ref{lem4_3}, we have 
\[
   G_{\tilde{U}}(\mu,\gamma,0)\partial_{\beta_i}V(\mu,\gamma,0,0)=0.
\]
But $\partial_{\beta_i}V(\mu,\gamma,0,0)\in C^{\infty}(\mathbf{X}_1)$, so 
\[
    \partial_{\beta_i}V(\mu,\gamma,0,0)=0, \quad i=1,2.
\]
Since $K$ is compact, we can take $\delta_1$ to be a universal constant for each  $(\mu,\gamma)\in K$. So we have proved the existence of $V$ in Theorem \ref{thm4_1}.

Next, let $(\mu,\gamma)\in B_{\delta_1}(\bar{\mu},\bar{\gamma})$. Let $\delta'$ be a small constant to be determined.  For any $U$ satisfies the equation (\ref{eq4_0}) with $U-U^{\mu,\gamma}\in \mathbf{X}$, and $||U-U^{\mu,\gamma}||_{ \mathbf{X}}\le \delta'$ there exist some $\beta_1,\beta_2\in \mathbb{R}$ and $V^*\in \mathbf{X}_1$ such that
 \[
     U-U^{\mu,\gamma}=\beta_1V^1_{\mu,\gamma}+\beta_2V^2_{\mu,\gamma}+V^*.
 \]
Then by Lemma \ref{lem4_1_10}, there exists some constant $C>0$ such that
\[
    \frac{1}{C}(|(\beta_1,\beta_2)|+||V^*||_{\mathbf{X}})\le ||\beta_1V^1_{\mu,\gamma}+\beta_2V^2_{\mu,\gamma}+V^*||_{\mathbf{X}}\le \delta'.
\]
This gives $||V^*||_{\mathbf{X}}\le C\delta'$.

 Choose $\delta'$ small enough such that $C\delta'<\delta_1$. We have the uniqueness of $V^*$. 
So  $V^*=V(\mu,\gamma,\beta_1,\beta_2)$ in (\ref{eq_thm4_1_1}).
The theorem is proved.
\qed

\subsection{Existence of solutions with nonzero swirl near $U^{\mu,\gamma}$ when $(\mu, \gamma)\in I_2$}

Let us look at the problem near $U^{\mu,\gamma}$ when $\mu=-\frac{1}{2}$ and $\gamma>-1$.  For such a fixed $(\mu,\gamma)$, write $\bar{U}=(\bar{U}_{\theta},0)$. Recall that in Corollary \ref{cor3_1}, we have
\begin{equation}\label{eq_noswirl4_2}
   \bar{U}_{\theta}=(1-x)\left(1+\frac{2(\gamma+1)}{(\gamma+1)\ln\frac{1+x}{2}-2}\right).
\end{equation}
It satisfies
\[
(1-x^2)\bar{U}'_{\theta}+2x\bar{U}_{\theta}+\frac{1}{2}\bar{U}^2_{\theta}=-\frac{1}{2}(1-x)^2.
\]
We will work with $\tilde{U}=U-\bar{U}$. Let $0<\epsilon<\frac{1}{2}$, define
\begin{equation*}
 \begin{split}
 &
 \begin{split} 
 \mathbf{M}_1 := & \left\{  \tilde{U}_\theta \in C([-1, 1], \mathbf{R}) \cap C^1((-1, 1], \mathbf{R}) \cap C^2((0, 1), \mathbf{R}) \mid \right. \\
	 & \tilde{U}_\theta(1)=\tilde{U}_\theta(-1)=0, ||\ln\left(\frac{1+x}{3}\right)\tilde{U}_\theta||_{L^\infty(-1,1)}<\infty, \\
	 & \left. ||(1+x)\left(\ln\frac{1+x}{3}\right)^2\tilde{U}_\theta'||_{L^\infty(-1,1)} < \infty, ||\tilde{U}_\theta''||_{L^\infty(0,1)} < \infty  \right\},
  \end{split}\\
   &
  \begin{split}
\mathbf{M}_2 = & \mathbf{M}_2(\epsilon) \\
:= & \left\{  \tilde{U}_\phi \in C^1( (-1, 1], \mathbf{R})\cap C^2( (-1, 1), \mathbf{R}) \mid \tilde{U}_\phi(1)=0, ||(1+x)^{\epsilon}\tilde{U}_\phi ||_{L^\infty(-1,1)} < \infty, \right.\\
	& \left. ||(1+x)^{1+\epsilon} \tilde{U}_\phi'||_{L^\infty(-1,1)} < \infty, ||(1+x)^{2+\epsilon} \tilde{U}_\phi'' ||_{L^\infty(-1,1)} <\infty \right\}
   \end{split}
\end{split}
\end{equation*}
 with the following norms accordingly:
\begin{equation*}
 \begin{split}
    &||\tilde{U}_{\theta}||_{\mathbf{M}_1}=||\ln\left(\frac{1+x}{3}\right)\tilde{U}_{\theta}||_{L^{\infty}(-1,1)}+||\left(\ln\frac{1+x}{3}\right)^2(1+x)\tilde{U}'_{\theta}||_{L^{\infty}(-1,1)}+||\tilde{U}''_{\theta}||_{L^{\infty}(0,1)}, \\
   &	||\tilde{U}_\phi||_{\mathbf{M}_2}:=  ||(1+x)^{\epsilon}\tilde{U}_\phi||_{L^\infty(-1,1)} + ||(1+x)^{1+\epsilon} \tilde{U}_\phi'||_{L^\infty(-1,1)}  + ||(1+x)^{2+\epsilon} \tilde{U}_\phi'' ||_{L^\infty(-1,1)}.
 \end{split}
\end{equation*}
Next, define
\begin{equation*}
  \begin{split}
  &\begin{split}
   \mathbf{N}_1:= & \left\{  \xi_\theta \in C( [-1, 1], \mathbf{R}) \cap C^1( (0, 1], \mathbf{R}) \mid  \xi_\theta(1)=\xi_{\theta}(-1)=\xi'_{\theta}(1)=0, \right.\\
   & ||\left(\ln\frac{1+x}{3}\right)^2\xi_\theta||_{L^\infty(-1,1)} < \infty,
     \left. ||\frac{\xi_\theta'}{1-x}||_{L^\infty(0,1)} < \infty \right\},
   \end{split}\\
  &   \mathbf{N}_2=\mathbf{N}_2(\epsilon):= \left\{  \xi_\phi \in C( (-1, 1], \mathbb{R}) \mid  \xi_\phi(1)=0, ||\frac{(1+x)^{1+\epsilon} \xi_\phi}{1-x}||_{L^\infty(-1,1)} < \infty  \right\} 
  \end{split}
\end{equation*}
with the following norms accordingly:
\begin{equation*}
  \begin{split}
   & ||\xi_{\theta}||_{\mathbf{N}_1}:=||\left(\ln\frac{1+x}{3}\right)^2\xi_\theta||_{L^\infty(-1,1)}+||\frac{\xi_\theta'}{1-x}||_{L^\infty(0,1)}, \\
   &||\xi_\phi||_{\mathbf{N}_2}:= ||\frac{(1+x)^{1+\varepsilon} \xi_\phi}{1-x}||_{L^\infty(-1,1)}.
  \end{split}
\end{equation*}
Let $\mathbf{X}:= \{ \tilde{U} = (\tilde{U}_\theta, \tilde{U}_\phi) \mid \tilde{U}_\theta\in \mathbf{M}_1, \tilde{U}_\phi\in \mathbf{M}_2\}$ with the norm $||\tilde{U}||_{\mathbf{X}}:= ||\tilde{U}_\theta||_{\mathbf{M}_1} + ||\tilde{U}_\phi||_{\mathbf{M}_2}$,  and $\mathbf{Y}:= \{ \xi = (\xi_\theta, \xi_\phi) \mid \xi_\theta\in \mathbf{N}_1, \xi_\phi\in \mathbf{N}_2 \}$ with the norm $||\xi||_\mathbf{Y}:= ||\xi_\theta||_{\mathbf{N}_1} + ||\xi_\phi||_{\mathbf{N}_2}$.  It is not difficult to verify that $\mathbf{M}_1$, $\mathbf{M}_2$, $\mathbf{N}_1$, $\mathbf{N}_2$, $\mathbf{X}$ and $\mathbf{Y}$ are Banach spaces.

Let $l_1,l_2:\mathbf{X}\to \mathbb{R}$ be the bounded linear functionals defined by (\ref{def_l}) for each $V\in \mathbf{X}$. Define 
\begin{equation}\label{eq4_2x}
	\mathbf{X}_1:= \ker l_1 \cap \ker l_2.
\end{equation}

\begin{thm}\label{thm4_2}
  For every  compact subset  $K$ of $(-1,+\infty)$,  there exists $\delta=\delta(K)>0$, and $V\in C^{\infty}(K\times B_{\delta}(0), \mathbf{X}_1)$ satisfying $V(\gamma,0,0)=0$ and $\displaystyle \frac{\partial V}{\partial \beta_i}|_{\beta=0}=0$, $i=1,2$, such that 
\begin{equation}\label{eq_thm4_2_1}
   U=U^{-\frac{1}{2},\gamma}+\beta_1V_{-\frac{1}{2},\gamma}^1+\beta_2V_{-\frac{1}{2},\gamma}^2+V(\gamma, \beta_1,\beta_2)
\end{equation}
satisfies  equation (\ref{eq4_0})  with $\displaystyle \hat{\mu}=-\frac{1}{2}-\frac{1}{4}\psi[U_{\phi}](-1)$.  Moreover, there exists some $\delta'=\delta'(K)>0$, such that if $||U-U^{-\frac{1}{2},\gamma}||_{\mathbf{X}}<\delta'$, $\gamma\in K$,  and $U$ satisfies   equation (\ref{eq4_0}) with some constant $\hat{\mu}$, then (\ref{eq_thm4_2_1}) holds for some $|(\beta_1,\beta_2)|<\delta$ .
\end{thm}

To prove Theorem \ref{thm4_2}, we first study properties of the Banach spaces $\mathbf{X}$ and $\mathbf{Y}$.

With the fixed $\epsilon\in (0,1)$, we have

\begin{lem}\label{lemP4_2_1}
For every $\tilde{U}\in \mathbf{X}$, it satisfies 
\begin{equation}\label{eq4_2_3'} 
   |\tilde{U}_{\phi}(s)|\le (1-s)(1+s)^{-\epsilon}||\tilde{U}_{\phi}||_{\mathbf{M}_2},\quad \forall -1<s<1,
\end{equation}

\begin{equation}\label{eq4_2_2'}
   |\tilde{U}_{\theta}(s)|\le (\ln 3)\left(\ln \frac{2}{3}\right)^{-2}\left(\ln \frac{1+s}{3}\right)^{-1}(1-s)||\tilde{U}_{\theta}||_{\mathbf{M}_1}, \quad \forall -1<s<1.
\end{equation}
\end{lem}
\begin{proof}
  For $s\in (0,1)$, there exists $y\in (s,1)$ such that
\[
   |\tilde{U}_{\phi}(s)|=|\tilde{U}'_{\phi}(y)|(1-s)\le (1-s)||\tilde{U}_{\phi}||_{\mathbf{M}_2},
\]
while for $s\in (-1,0]$, $|\tilde{U}_{\phi}(s)|\le (1+s)^{-\epsilon}||\tilde{U}_{\phi}||_{\mathbf{M}_2}\le (1-s)(1+s)^{-\epsilon}||\tilde{U}_{\phi}||_{\mathbf{M}_2}$. So (\ref{eq4_2_3'}) is proved.  

  Now we prove (\ref{eq4_2_2'}). For $0\le s<1$, by the fact that $\tilde{U}_{\theta}\in \mathbf{M}_1$, we have $|\tilde{U}'_{\theta}(s)|\le \left(\ln \frac{2}{3}\right)^{-2}||\tilde{U}_{\theta}||_{\mathbf{M}_1}$. So 
  \[
    \begin{split}
    \left|\left(\ln\frac{1+s}{3}\right)\tilde{U}_{\theta}(s)\right|& \le (\ln 3)|\tilde{U}_{\theta}(s)|= (\ln 3)|\tilde{U}_{\theta}(s)-\tilde{U}_{\theta}(1)|\le  (\ln 3)||\tilde{U}'_{\theta}||_{L^{\infty}(0,1)}(1-s)\\
    & \le (\ln 3)\left(\ln \frac{2}{3}\right)^{-2}(1-s)||\tilde{U}_{\theta}||_{\mathbf{M}_1}.
    \end{split}
    \]
     For $-1<s<0$, $\left|\left(\ln\frac{1+s}{3}\right)(1-s)^{-1}\tilde{U}_{\theta}(s)\right|\le \left|\left(\ln\frac{1+s}{3}\right)\tilde{U}_{\theta}(s)\right|\le ||\tilde{U}_{\theta}||_{\mathbf{M}_1}$. So (\ref{eq4_2_2'}) is proved.
\end{proof}

\begin{lem}\label{lemP4_2_xi}
   For every $\xi_{\theta}\in \mathbf{N}_1$,
   \[
       |\xi_{\theta}(x)|\le (\ln 3)^2 \left(\ln \frac{1+x}{3}\right)^{-2} (1-x)^2||\xi_{\theta}||_{\mathbf{N}_1}, \quad -1<x<1.
   \]
\end{lem}
\begin{proof}
   If $\xi_{\theta}\in \mathbf{N}_1$, $\xi_{\theta}(1)=0$. So for every $0<x<1$, there exists $y\in (x,1)$ such that
   \[
     \begin{split}
       \left|\left(\ln\frac{1+x}{3}\right)^2\xi_{\theta}(x)\right| & \le (\ln 3)^2|\xi_{\theta}(x)|=(\ln 3)^2|\xi'_{\theta}(y)(1-x)|\le (\ln 3)^2||\xi_{\theta}||_{\mathbf{N}_1}(1-y)(1-x)\\
       & \le (\ln 3)^2||\xi_{\theta}||_{\mathbf{N}_1}(1-x)^2.
       \end{split}
   \]
   For $-1<x\le 0$, $\left|\left(\ln\frac{1+x}{3}\right)^2\xi_{\theta}(x)\right|\le ||\xi_{\theta}||_{\mathbf{N}_1} \le ||\xi_{\theta}||_{\mathbf{N}_1}(1-x)^2$.
\end{proof}

Now let $K$ be a compact subset of $(-1,+\infty)$. For $\tilde{U}_{\phi}\in \mathbf{M}_2$, let $\psi[\tilde{U}_{\phi}](x)$ be defined by (\ref{eqpsi4_1}). Then define a map $G$ on $K\times\mathbf{X}$ such that for each $(\gamma,\tilde{U})\in K\times\mathbf{X}$, $G(\gamma,\tilde{U})=G(-\frac{1}{2}, \gamma,\tilde{U})$ given by (\ref{G}) with $\bar{U}_{\theta}$ in (\ref{eq_noswirl4_2}). If $\tilde{U}$ satisfies $G(\gamma, \tilde{U})=0$, then $U=\tilde{U}+\bar{U}$ gives a solution of (\ref{eq4_0}) with $\hat{\mu}=-\frac{1}{2}-\frac{1}{4}\psi[\tilde{U}_{\phi}](-1)$, satisfying $U_{\theta}(-1)=\bar{U}_{\theta}(-1)=2$.

\begin{prop}\label{prop4_2}
	The map $G$ is in $C^{\infty}(K\times \mathbf{X},\mathbf{Y})$ in the sense that $G$ has continuous Fr\'{e}chet derivatives of every order.  Moreover, the Fr\'{e}chet derivative of $G$ with respect to $\tilde{U}$ at $(\gamma,\tilde{U})\in K\times \mathbf{X}$ is given by the linear operator $L^{-\frac{1}{2},\gamma}_{\tilde{U}}:\mathbf{X}\rightarrow \mathbf{Y}$ defined as in (\ref{eq4_Linear}).
\end{prop}

To prove Proposition \ref{prop4_2}, we first prove the following lemmas:

\begin{lem}\label{lemP4_2_2'}
For every $\gamma\in K$, the map $A(-\frac{1}{2},\gamma,\cdot):\mathbf{X} \to \mathbf{Y}$  defined by (\ref{eqP4_2}) is a bounded linear operator.
\end{lem}
\begin{proof}
For convenience we denote $A=A(-\frac{1}{2},\gamma,\cdot)$. We make use of the properties of $\bar{U}_{\theta}$ that $\bar{U}_{\theta}(1)=0$, $\bar{U}_{\theta}\in C^2(-1,1]\cap L^{\infty}(-1,1)$ and $\displaystyle \bar{U}_{\theta}-2=O(1)\frac{1}{\ln(1+x)}$.

 $A$ is clearly linear.  For every $\tilde{U}\in\mathbf{X}$, we prove that $A\tilde{U}$ defined by (\ref{eqP4_2}) is in $\mathbf{Y}$ and there exists some constant $C$ such that $||A\tilde{U}||_{\mathbf{Y}}\le C||\tilde{U}||_{\mathbf{X}}$ for all $\tilde{U}\in \mathbf{X}$.
 
 By the fact that $\tilde{U}_{\theta}\in \mathbf{M}_1$ and (\ref{eq4_2_2'}),  we have
\[
\begin{split}
  & \left|\left(\ln\frac{1+x}{3}\right)^2A_{\theta}\right|  \\
\le & (1-x)\left|(1+x)\left(\ln\frac{1+x}{3}\right)^2\tilde{U}'_{\theta}\right|+\left|\left(2x+\bar{U}_{\theta}\right)\ln\frac{1+x}{3}\right| \cdot \left| \tilde{U}_{\theta} \ln\frac{1+x}{3}\right| \\
      \le & C(1-x)||\tilde{U}_{\theta}||_{\mathbf{M}_1}.
\end{split}
\]
From the above we also see that $\lim_{x\to 1}A_{\theta}(x)=\lim_{x\to -1}A_{\theta}(x)=0$. By computation $A'_{\theta}=(1-x^2)\tilde{U}''_{\theta}+\bar{U}_{\theta}\tilde{U}'_{\theta}+(2+\bar{U}'_{\theta})\tilde{U}_{\theta}$. Then by the fact that $\tilde{U}_{\theta}\in \mathbf{M}_1$ and  (\ref{eq4_2_2'}), for $0<x<1$,
\[
   \frac{|A'_{\theta}(x)|}{1-x}  \le (1+x)|\tilde{U}''_{\theta}|+|\tilde{U}'_{\theta}|+\frac{|\tilde{U}_{\theta}|}{1-x}\le C||\tilde{U}_{\theta}||_{\mathbf{M}_1}, \quad 0<x<1.
\]
So $A_{\theta}\in \mathbf{N}_1$ and $||A_{\theta}||_{\mathbf{N}_1}\le C ||\tilde{U}_{\theta}||_{\mathbf{M}_1}$.

Next, by the fact that $\tilde{U}_{\phi}\in \mathbf{M}_2$ and (\ref{eq_noswirl4_2}), 
\[
   \frac{(1+x)^{1+\epsilon}}{1-x}|A_{\phi}|\le |(1+x)^{2+\epsilon} \tilde{U}_\phi'' |+|(1+x)^{1+\epsilon} \frac{|\bar{U}_{\theta}|}{1-x}\tilde{U}_\phi'|\le C||\tilde{U}_{\phi}||_{\mathbf{M}_2}.
\]
We also see from the above that $\lim_{x\to 1}A_{\phi}(x)=0$. So $A_{\phi}\in \mathbf{N}_2$, and $||A_{\phi}||_{\mathbf{N}_2}\le  C||\tilde{U}_{\phi}||_{\mathbf{M}_2}$.
 We have proved that $A\tilde{U}\in\mathbf{Y}$, and $||A\tilde{U}||_{\mathbf{Y}}\le C||\tilde{U}||_{\mathbf{X}}$ for every $\tilde{U}\in \mathbf{X}$.
\end{proof}

\begin{lem}\label{lemP4_2_3'}
The map $ Q:\mathbf{X}\times \mathbf{X}\to \mathbf{Y}$  defined by (\ref{eqP4_3}) is a  bounded bilinear operator.
\end{lem}
\begin{proof}
In the following, $C$ denotes a universal constant which may change from line to line.  
 It is clear that $Q$ is a bilinear operator. For every $\tilde{U},\tilde{V}\in\mathbf{X}$, we will prove that $Q(\tilde{U},\tilde{V})$ is in $\mathbf{Y}$ and there exists some constant $C$ independent of $\tilde{U}$ and $\tilde{V}$ such that $||Q(\tilde{U},\tilde{V})||_{\mathbf{Y}}\le C||\tilde{U}||_{\mathbf{X}}||\tilde{V}||_{\mathbf{X}}$.
 
For convenience we write
\[
  \psi(\tilde{U},\tilde{V})(x)=\int_x^1 \int_l^1 \int_t^1 \frac{2 \tilde{U}_\phi(s)\tilde{V} _\phi'(s)}{1-s^2} ds dt dl.
\]
 For $\tilde{U},\tilde{V}\in \mathbf{X}$, we have, using (\ref{eq4_2_3'}) in Lemma \ref{lemP4_2_1}, that
\begin{equation}\label{eq4_2_6'}
    \left|\frac{ \tilde{U}_\phi(s)\tilde{V} _\phi'(s)}{1-s^2}\right|\le (1+s)^{-2-2\epsilon}||\tilde{U}_{\phi}||_{\mathbf{M}_2}||\tilde{V}_{\phi}||_{\mathbf{M}_2}.
\end{equation}
It follows that $\psi(\tilde{U},\tilde{V})(x)$ is well-defined and 
\begin{equation}\label{eq4_2_6'new}
    |\psi(\tilde{U},\tilde{V})(x)|\le C(\epsilon)(1-x)^3||\tilde{U}_{\phi}||_{\mathbf{M}_2}||\tilde{V}_{\phi}||_{\mathbf{M}_2}, \quad -1<x<1.
\end{equation}
Moreover, we have, in view of (\ref{eq4_2_6'}), that
\begin{equation}\label{eq4_2_6'_1}
  \begin{split}
    & |\psi(\tilde{U},\tilde{V})(x)-\frac{(1-x)^2}{4}\psi(\tilde{U},\tilde{V})(-1)| \\
& =\left|\psi(\tilde{U},\tilde{V})(x)-\psi(\tilde{U},\tilde{V})(-1)+\frac{(1+x)(3-x)}{4}\psi(\tilde{U},\tilde{V})(-1)\right|\\
           & =\left|-\int_{-1}^{x} \int_{l}^{1} \int_{t}^{1} \frac{2 \tilde{U}_\phi(s)\tilde{V} _\phi'(s)}{1-s^2} ds dt dl+\frac{(1+x)(3-x)}{4}\psi(\tilde{U},\tilde{V})(-1)\right|\\
                   & \le C(1+x)^{1-2\epsilon}||\tilde{U}_\phi||_{\mathbf{M}_2}||\tilde{V} _{\phi}||_{\mathbf{M}_2}, \quad \forall -1<x\le 0.
  \end{split}
\end{equation}
Thus, using (\ref{eq4_2_6'new}) and (\ref{eq4_2_6'_1}), we have
\begin{equation}\label{eq4_2_6'_2}
  |\psi(\tilde{U},\tilde{V})(x)-\frac{(1-x)^2}{4}\psi(\tilde{U},\tilde{V})(-1)|\le C(\epsilon)(1+x)^{1-2\epsilon}(1-x)^2||\tilde{U}_\phi||_{\mathbf{M}_2}||\tilde{V} _{\phi}||_{\mathbf{M}_2}, \quad \forall -1<x<1.
\end{equation}
So by (\ref{eq4_2_2'}), (\ref{eq4_2_6'_2}) and the fact that $\tilde{U}_{\theta},\tilde{V}_{\theta}\in \mathbf{M}_1$, we have
\[
  \begin{split}
  & |\left(\ln\frac{1+x}{3}\right)^2Q_{\theta}(x)| \\
  & \le \frac{1}{2}\left(\ln\frac{1+x}{3}\right)^2|\tilde{U}_{\theta}(x)||\tilde{V}_{\theta}(x)|+\left(\ln\frac{1+x}{3}\right)^2\left|\psi(\tilde{U},\tilde{V})(x)-\frac{(1-x)^2}{4}\psi(\tilde{U},\tilde{V})(-1)\right|\\
                  & \le C(1-x)^2||\tilde{U}_{\theta}||_{\mathbf{M}_1}||\tilde{V}_{\theta}||_{\mathbf{M}_1}+ C\left(\ln\frac{1+x}{3}\right)^2(1+x)^{1-2\epsilon}(1-x)^2||\tilde{U}_{\phi}(s)||_{\mathbf{M}_2}||\tilde{V} _{\phi}||_{\mathbf{M}_2}\\
                  & \le C(1-x)^2||\tilde{U}||_{\mathbf{X}}||\tilde{V}||_{\mathbf{X}}, \quad \forall -1<x<1.
\end{split}
\]
From this we also have $\displaystyle \lim_{x\to 1}Q_{\theta}(x)=\lim_{x\to -1}Q_{\theta}(x)=0$.

A calculation gives
\[
   Q'_{\theta}(x)=\frac{1}{2}\tilde{U}_{\theta}\tilde{V}'_{\theta}+\frac{1}{2}\tilde{U}'_{\theta}\tilde{V}_{\theta}+\int_{x}^{1}\int_{t}^{1}\frac{ 2\tilde{U}_\phi(s)\tilde{V} _\phi'(s)}{1-s^2}dsdt-\frac{1-x}{2}\psi(\tilde{U},\tilde{V})(-1), \textrm{ for } 0<x<1.
\]
Using $\tilde{U}\in \mathbf{X}$, (\ref{eq4_2_2'}), (\ref{eq4_2_3'}) and (\ref{eq4_2_6'}), we see that,
\[
   |Q'_{\theta}(x)|\le C(1-x)||\tilde{U}||_{\mathbf{X}}||\tilde{V}||_{\mathbf{X}},\quad \forall 0<x<1.
\]
So $Q_{\theta}\in\mathbf{N}_1$, and $||Q_{\theta}||_{\mathbf{N}_1}\le C||\tilde{U}||_{\mathbf{X}}||\tilde{V}||_{\mathbf{X}}$.

Next, since $Q_{\phi}(x)=\tilde{U}_{\theta}(x)\tilde{V}'_{\phi}(x)$, for $-1<x<1$, 
\[
   \left|\frac{(1+x)^{1+\epsilon}Q_{\phi}}{1-x}\right|  \le \frac{(1+x)^{1+\epsilon}}{1-x}|\tilde{U}_{\theta}(x)|\frac{||\tilde{V}_{\phi}||_{\tilde{\mathbf{M}}_2}}{(1+x)^{1+\epsilon}} \le C||\tilde{U}_{\theta}||_{\tilde{\mathbf{M}}_1}||\tilde{V}_{\phi}||_{\mathbf{M}_2}.
\]
We also see from the above that $\displaystyle \lim_{x\to 1}Q_{\phi}(x)=0$. So $Q_{\phi}\in \mathbf{N}_2$, and $||Q_{\phi}||_{\mathbf{N}_2}\le||\tilde{U}_{\theta}||_{\mathbf{M}_1}||\tilde{V}_{\phi}||_{\mathbf{M}_2}$. Thus we have proved $Q(\tilde{U},\tilde{V})\in \mathbf{Y}$ and $||Q(\tilde{U},\tilde{V})||_{\mathbf{Y}}\le C||\tilde{U}||_{\mathbf{X}}||\tilde{V}||_{\mathbf{X}}$ for all $\tilde{U}, \tilde{V}\in \mathbf{X}$. The proof is finished.
\end{proof}

\noindent \emph{Proof of Proposition \ref{prop4_2}:}
By definition, $G(-\frac{1}{2},\gamma,\tilde{U})=A(-\frac{1}{2},\gamma,\tilde{U})+Q(\tilde{U},\tilde{U})$ for $(\gamma,\tilde{U})\in K\times \mathbf{X}$.  Using standard theories in functional analysis, by Lemma \ref{lemP4_2_3'} it is clear that $Q$ is $C^{\infty}$ on $K\times  \mathbf{X}$.  By Lemma \ref{lemP4_2_2'}, $A(-\frac{1}{2},\gamma, \cdot): \mathbf{X}\to \mathbf{Y}$ is $C^{\infty}$ for each $\gamma\in K$. For all $i\ge 1$, we have
\[
    \partial_{\gamma}^i A(-\frac{1}{2},\gamma,\tilde{U})= \partial_{\gamma}^i U^{-\frac{1}{2},\gamma}_{\theta}\left(
	\begin{matrix}
		\tilde{U}_{\theta}  \\  \tilde{U}'_{\phi}
	\end{matrix}\right).
\]  
By (\ref{eqcor3_1_3}), for each integer $i\ge 1$, there exists some constant $C=C(i,K)$, depending only on $i,K$, such that 
\begin{equation}\label{eq_prop4_1_2}
   |\partial_{\gamma}^i U^{-\frac{1}{2},\gamma}_{\theta}(x)|\le C(i,K)(1-x) \left( \ln \frac{1+x}{3} \right)^{-2}, \quad -1<x<1.
\end{equation}
From (\ref{eq_noswirl4_2}) we also obtain
\[
     \left|\frac{d}{dx}\partial_{\gamma}^i U^{-\frac{1}{2},\gamma}_{\theta}(x)\right|\le C(i,K), \quad 0<x<1.
\]
Using the above estimates and the fact that $\tilde{U}_{\theta}\in \mathbf{M}_1$, we have
\[
   \left|\left(\ln \frac{1+x}{3}\right)^2\partial_{\gamma}^i A_{\theta}(-\frac{1}{2},\gamma,\tilde{U})\right|\le C(i,K)(1-x)||\tilde{U}_{\theta}||_{\mathbf{M}_1}, \quad -1<x<1,
\]
and 
\[
\begin{split}
   \left|\frac{d}{dx}\partial_{\gamma}^i A_{\theta}(-\frac{1}{2},\gamma,\tilde{U})\right|& \le \left|\frac{d}{dx}\partial_{\gamma}^i U^{-\frac{1}{2},\gamma}_{\theta}(x)\right||\tilde{U}_{\theta}(x)|+|\partial_{\gamma}^i U^{-\frac{1}{2},\gamma}_{\theta}(x)|\left|\frac{d}{dx}\tilde{U}_{\theta}(x)\right|\\
   & \le C(i,K)(1-x)||\tilde{U}_{\theta}||_{\mathbf{M}_1}, \quad 0<x<1.
   \end{split}
\]
So $\partial_{\gamma}^i A_{\theta}(-\frac{1}{2},\gamma,\tilde{U})\in \mathbf{N}_1$, with $||\partial_{\gamma}^i A_{\theta}(-\frac{1}{2},\gamma,\tilde{U})||_{\mathbf{N}_1}\le C(i,K)||\tilde{U}_{\theta}||_{\mathbf{M}_1}$ for all $(\gamma,\tilde{U})\in K\times \mathbf{X}$.

Next, by (\ref{eq_prop4_1_2}) and the fact that $\tilde{U}_{\phi}\in \mathbf{M}_1$, we have
\[
   \frac{(1+x)^{1+\epsilon}}{1-x}|\partial_{\gamma}^i A_{\phi}(\mu,\gamma,\tilde{U})(x)|=\frac{|\partial_{\gamma}^i U^{-\frac{1}{2},\gamma}_{\theta}(x)|}{1-x}|(1+x)^{1+\epsilon}U'_{\phi}|\le C(i,K)||\tilde{U}_{\phi}||_{\mathbf{M}_2}.
\]
So $\partial_{\gamma}^i A_{\phi}(-\frac{1}{2},\gamma,\tilde{U})\in \mathbf{N}_2$, with $||\partial_{\gamma}^i A_{\phi}(-\frac{1}{2},\gamma,\tilde{U})||_{\mathbf{N}_2}\le C(i,K)||\tilde{U}_{\phi}||_{\mathbf{M}_2}$ for all $(\gamma,\tilde{U})\in K\times \mathbf{X}$. Thus $\partial_{\gamma}^iA(-\frac{1}{2},\gamma,\tilde{U})\in \mathbf{Y}$, with $||\partial_{\gamma}^i A(-\frac{1}{2},\gamma,\tilde{U})||_{\mathbf{Y}}\le C(i,K)||\tilde{U}||_{\mathbf{X}}$ for all $(\gamma,\tilde{U})\in K\times \mathbf{X}$, $i\ge 1$. 

So for each $\gamma \in K$, $\partial_{\gamma}^i A(-\frac{1}{2},\gamma,  \cdot):\mathbf{X}\to \mathbf{Y}$ is a bounded linear map with uniform bounded norm on $K$. Then by standard theories in functional analysis, $A:K\times\mathbf{X}\to \mathbf{Y}$ is $C^{\infty}$. So $G$ is a $C^{\infty}$ map from $K\times\mathbf{X}$ to $\mathbf{Y}$. By direct calculation we have that its Fr\'{e}chet derivative with respect to $\mathbf{X}$ is given by  the linear bounded operator $L^{-\frac{1}{2},\gamma}_{\tilde{U}}: \mathbf{X}\rightarrow \mathbf{Y}$ defined as  (\ref{eq4_Linear}). The proof is finished.     \qed\\

By Proposition \ref{prop4_2}, $L^{-\frac{1}{2},\gamma}_0:\mathbf{X}\to \mathbf{Y}$, the Fr\'{e}chet derivative of $G$ with respect to $\tilde{U}$ at $\tilde{U}=0$, is given by (\ref{eq_LinearAtZero}).

Next, with $a_{-\frac{1}{2},\gamma}(x)$, $b_{-\frac{1}{2},\gamma}(x)$ defined by (\ref{eq_ab}) with $\bar{U}_{\theta}$ given by (\ref{eq_noswirl4_2}), we define   $W^{-\frac{1}{2},\gamma}(\xi)$   by (\ref{eq4_1_6}) for $\xi\in \mathbf{Y}$. Then $a_{-\frac{1}{2},\gamma}(x)$ and $W^{-\frac{1}{2},\gamma}(\xi)(x)$ satisfy (\ref{eq4_W_a}) and (\ref{eq4_W_d}).

\begin{lem} \label{lem4_2_W}
   For every $\gamma\in K$, $W^{-\frac{1}{2},\gamma}:\mathbf{Y}\to \mathbf{X}$ is continuous and is a right inverse of $L^{-\frac{1}{2},\gamma}_{0}$.
\end{lem}
\begin{proof}
We make use of the properties that $U^{-\frac{1}{2},\gamma}_{\theta}(1)=0$, $U^{-\frac{1}{2},\gamma}_{\theta}\in C^2(-1,1]\cap C^0[-1,1]$ and $\left|\left(\ln \frac{1+x}{3}\right)(\bar{U}_{\theta}(x)-2)\right|\in L^{\infty}(-1,1)$. For convenience, we write $W:=W^{-\frac{1}{2},\gamma}$, $a(x)=a_{-\frac{1}{2},\gamma}(x)$ and $b(x)=b_{-\frac{1}{2},\gamma}(x)$.

We first prove that $W$ is well-defined, denote $W:=W(\xi)$. Applying Lemma \ref{lemP4_2_xi} in the expression of $W_{\theta}$ in (\ref{eq4_1_6}),  we have, for $-1<x<1$, that
\begin{equation}\label{eq4_2_W_1}
   \left|\left(\ln \frac{1+x}{3}\right)W_{\theta}(x)\right|\le C\left(\ln \frac{1+x}{3}\right)||\xi_{\theta}||_{\mathbf{N}_1}e^{-a(x)}\int_{0}^{x}e^{a(s)}(1-s)(1+s)^{-1}\left(\ln \frac{1+s}{3}\right)^{-2}ds.
\end{equation}
We make estimates first for  $0<x\le 1$ and then for $-1<x\le 0$.

\textbf{Case} 1: $0<x\le 1$.

By (\ref{eq_noswirl4_2}), $\bar{U}_{\theta}=-\bar{U}'_{\theta}(1)(1-x)+O((1-x)^2)$. Using similar arguments as in the proof of Lemma \ref{lem4_1}, $b(x)$ and $a(x)$ satisfy (\ref{eq4_1_7_1}),  (\ref{eq4_1_7_2}) and (\ref{eq4_1_7_3}). So there exists some positive constant $C$ such that
\[
   e^{a(s)}(1-s)(1+s)^{-1}\left(\ln \frac{1+s}{3}\right)^{-2}\le C, \quad e^{a(x)}\ge \frac{1}{C(1-x)}, \quad 0<s<x<1.
\]
Then using the above estimate in (\ref{eq4_2_W_1}), we have that
\begin{equation}\label{eq4_2_W_2}
   |W_{\theta}(x)|\le C||\xi_{\theta}||_{\mathbf{N}_1}(1-x), \quad 0<x\le 1.
\end{equation}
In particular, $W_{\theta}(1)=0$.

In (\ref{eq4_W_a}), using $\bar{U}_{\theta}=-\bar{U}'_{\theta}(1)(1-x)+O((1-x)^2)$, we have
\begin{equation}\label{eq4_2_W_3}
   |a'(x)|\le \frac{C}{1-x}, \quad |a''(x)|\le \frac{C}{(1-x)^2}, \quad 0<x<1.
\end{equation}
Then 
\[
   |W'_{\theta}(x)|\le |a'(x)||W_{\theta}(x)|+\frac{|\xi_{\theta}(x)|}{1-x}\le C||\xi_{\theta}||_{\mathbf{N}_1}, \quad 0<x<1, 
\]
where we have used (\ref{eq4_2_W_2}), (\ref{eq4_2_W_3}), the fact that $\xi\in \mathbf{Y}$ and Lemma \ref{lemP4_2_xi}. 

Next, A calculation gives
\[
  W''_{\theta}(x)  =((a'(x))^2-a''(x))W_{\theta}(x)-a'(x)\frac{\xi_{\theta}(x)}{1-x^2}+\frac{\xi'_{\theta}(x)}{1-x^2}+\frac{2x\xi_{\theta}(x)}{(1-x^2)^2}.
\]
So
	\[
		|W''_{\theta}(x)|  \le  |(a'(x))^2-a''(x)| |W_{\theta}| + |a'(x)|\frac{|\xi_{\theta}|}{1-x^2} + \frac{|\xi'_{\theta}|}{(1-x)}+\frac{|\xi_{\theta}|}{(1-x)^2}.
	\]
	By computation
	\[
	    (a'(x))^2-a''(x)=\frac{\bar{U}^2_{\theta}+2x\bar{U}_{\theta}}{(1-x^2)^2}-\frac{2+\bar{U}'_{\theta}}{1-x^2}=O\left(\frac{1}{1-x}\right).
	\]
It follows, using (\ref{eq4_2_W_2}),  (\ref{eq4_2_W_3}) and  Lemma \ref{lemP4_2_xi}, that
\[
   |W''_{\theta}(x)|\le  C\left(\frac{|W_{\theta}(x)|}{1-x}+\frac{|\xi_{\theta}|}{(1-x)^2}+\frac{|\xi'_{\theta}|}{1-x}\right)\le C||\xi_{\theta}||_{\mathbf{N}_1}, \quad 0<x<1.
\]

\textbf{Case} 2: $-1<x\le 0$.

In (\ref{eq_noswirl4_2}), since $\gamma>-1$, we have 
\begin{equation}\label{eq4_2_W_0}
   \bar{U}_{\theta}(x)=2+\frac{4}{\ln\frac{1+x}{3}}+O\left(\left(\ln \frac{1+x}{3}\right)^{-2}\right).
   \end{equation}
    Then we have, for $-1<x\le 0$, that
\[
   b(x)=\ln \frac{1+x}{3}+2\ln\left(-\ln \frac{1+x}{3}\right)+O(1), \quad a(x)=2\ln\left(-\ln \frac{1+x}{3}\right)+O(1),
\]
\[
   e^{a(x)}=\left(\ln \frac{1+x}{3}\right)^2e^{O(1)}, \quad e^{-a(x)}=\left(\ln \frac{1+x}{3}\right)^{-2}e^{O(1)}.
\]
So there exists some constant $C$ such that for $-1<x<s\le 0$
\[
   e^{a(s)}(1-s)(1+s)^{-1}\left(\ln \frac{1+s}{3}\right)^{-2}\le C(1+s)^{-1}, \quad e^{-a(x)}\le C\left(\ln \frac{1+x}{3}\right)^{-2}.
\]
Apply these estimates in (\ref{eq4_2_W_1}), we have
\begin{equation}\label{eq4_2_W_4}
   \left|\left(\ln \frac{1+x}{3}\right)W_{\theta}(x)\right|\le C||\xi_{\theta}||_{\mathbf{N}_1}\left(\ln \frac{1+x}{3}\right)^{-1}\int_{0}^{x}\frac{1}{1+s}ds\le C||\xi_{\theta}||_{\mathbf{N}_1}, \quad -1<x\le 0.
\end{equation}
By (\ref{eq4_W_a}) and (\ref{eq4_2_W_0}),  there exists some $C$ such that
\[
   |a'(x)|\le \frac{C}{(1+x)\ln\frac{1+x}{3}}.
\]
Then by (\ref{eq4_W_d}), (\ref{eq4_2_W_4}) and Lemma \ref{lemP4_2_xi}, we have, for $-1<x\le 0$,  that
\[
\begin{split}
   & \left|(1+x) \left(\ln \frac{1+x}{3}\right)^2W'_{\theta}(x)\right| \\
\le & \left|a'(x)(1+x)\left(\ln \frac{1+x}{3}\right)^2W_{\theta}(x)\right|+\left(\ln \frac{1+x}{3}\right)^2\frac{|\xi_{\theta}(x)|}{1-x}\le C||\xi_{\theta}||_{\mathbf{N}_1}.
\end{split}
\]
So we have shown that $W_{\theta}\in \mathbf{M}_1$, and $||W_{\theta}||_{\mathbf{M}_1}\le C||\xi_{\theta}||_{\mathbf{N}_1}$ for some constant $C$.

By the definition of $W_{\phi}(\xi)$ in (\ref{eq4_1_6}) and the fact that $\xi_{\phi}\in \mathbf{N}_2$, we have, for every $-1<x<1$, that
\[
    |W_{\phi}(x)|\le \int_{x}^{1}e^{-b(t)}\int_{t}^{1}e^{b(s)}\frac{|\xi_{\phi}(s)|}{1-s^2}dsdt\le ||\xi_{\phi}||_{\mathbf{N}_2} \int_{x}^{1}e^{-b(t)}\int_{t}^{1}e^{b(s)}(1+s)^{-\textcolor{red}{2}-\epsilon}dsdt.
\]
Since $b(x)=\ln \frac{1+x}{3}+2\ln\left(-\ln \frac{1+x}{3}\right)+O(1)$ for $-1<x<1$, there exists some constant $C$ such that
\begin{equation}\label{eq4_2_W_5}
   e^{b(s)}\le C(1+s)\left(\ln \frac{1+s}{3}\right)^2, \quad e^{-b(t)}\le \frac{C}{(1+t)\left(\ln \frac{1+t}{3}\right)^2}, \quad -1<s,t\le 1.
\end{equation}
So we have
\[
 \begin{split}
    (1+x)^{\epsilon}|W_{\phi}(x)| & \le C(1+x)^{\epsilon}||\xi_{\phi}||_{\mathbf{N}_2}\int_{x}^{1}(1+t)^{-1}\left(\ln \frac{1+t}{3}\right)^{-2}\int_{t}^{1}(1+s)^{-1-\epsilon}\left(\ln \frac{1+s}{3}\right)^2dsdt\\
       & \le C||\xi_{\phi}||_{\mathbf{N}_2}, \quad -1<x\le 1.
       \end{split}
\]
For $0<x<1$, it can be seen from the above that $|W_{\phi}(x)|\le C||\xi_{\phi}||_{\mathbf{N}_2}(1-x)$. In particular, $W_{\phi}(1)=0$.  By computation
\[
    W'_{\phi}(x)=-e^{-b(x)}\int_{x}^{1}e^{b(s)}\frac{\xi_{\phi}(s)}{1-s^2}ds.
\]
Using (\ref{eq4_2_W_5}) and the fact that $\xi_{\phi}\in \mathbf{N}_2$, we have that for $-1<x<1$,
\[
   |(1+x)^{1+\epsilon}W'_{\phi}(x)|\le C||\xi_{\phi}||_{\mathbf{N}_2}(1+x)^{\epsilon}\left(\ln \frac{1+x}{3}\right)^{-2}\int_{x}^{1}(1+s)^{-1-\epsilon}\left(\ln \frac{1+s}{3}\right)^2ds\le C||\xi_{\phi}||_{\mathbf{N}_2}.
\]
Similarly, 
\[
  W''_{\phi}(x)=b'(x)e^{-b(x)}\int_{x}^{1}e^{b(s)}\frac{\xi_{\phi}(s)}{1-s^2}ds+\frac{\xi_{\phi}(x)}{1-x^2}.
\]
By (\ref{eq4_2_W_0}), $|b'(x)|=\frac{|\bar{U}_{\theta}(x)|}{1-x^2}=O((1+x)^{-1})$. Using (\ref{eq4_2_W_5}), we have
\[
   |(1+x)^{2+\epsilon}W''_{\phi}(x)|\le C||\xi_{\phi}||_{\mathbf{N}_2},\quad -1<x<1.
\]
So $W_{\phi}\in \mathbf{M}_2$, and $||W_{\phi}||_{\mathbf{M}_2}\le C||\xi_{\phi}||_{\mathbf{N}_2}$ for some constant $C$. 

Thus $W^{-\frac{1}{2},\gamma}(\xi)\in \mathbf{X}$ for all $\xi\in \mathbf{Y}$, and $||W^{-\frac{1}{2},\gamma}(\xi)||_{\mathbf{X}}\le C||\xi||_{\mathbf{Y}}$ for some constant $C$. So $W^{-\frac{1}{2},\gamma}: \mathbf{Y}\to \mathbf{X}$ is well-defined and continuous. It can be directly checked that $W$ is a right inverse of $L^{-\frac{1}{2},\gamma}_0$.
\end{proof}

Let $V_{-\frac{1}{2},\gamma}^1,V_{-\frac{1}{2},\gamma}^2,V_{-\frac{1}{2},\gamma}^3$ be defined by (\ref{eq4_2_ker}) with related $a_{-\frac{1}{2},\gamma}(x)$ and $b_{-\frac{1}{2},\gamma}(x)$ in the current case, we have
\begin{lem}\label{lem4_7}
$\{V_{-\frac{1}{2},\gamma}^1,V_{-\frac{1}{2},\gamma}^2\}$ is a basis of the kernel of $L^{-\frac{1}{2},\gamma}_0:\mathbf{X}\to \mathbf{Y}$.
\end{lem}
\begin{proof}
   Let $V\in \mathbf{X}$, $L_0 V=0$. It can be seen that $V$ is given by $V=c_1V_{-\frac{1}{2},\gamma}^1+c_2V_{-\frac{1}{2},\gamma}^2+c_3V_{-\frac{1}{2},\gamma}^3$ for some constants $c_1,c_2,c_3$. It is not hard to verify that $V_{-\frac{1}{2},\gamma}^1, V_{-\frac{1}{2},\gamma}^2\in \mathbf{X}$, and  $ V_{-\frac{1}{2},\gamma}^3\notin\mathbf{X}$. Since $V\in \mathbf{X}$, we must have $c_3V^3\in \mathbf{X}$, so $c_3=0$, and  $V\in \mathrm{span}\{V_{-\frac{1}{2},\gamma}^1,V_{-\frac{1}{2},\gamma}^2\}$. It is clear that $\{V_{-\frac{1}{2},\gamma}^1,V_{-\frac{1}{2},\gamma}^2\}$ is independent. So $\{V_{-\frac{1}{2},\gamma}^1,V_{-\frac{1}{2},\gamma}^2\}$ is a basis of the kernel.
\end{proof}
\begin{cor}\label{cor4_7}
For any $\xi=(\xi_{\theta}, \xi_{\phi})\in \mathbf{Y}$, all solutions  of $L^{-\frac{1}{2},\gamma}_{0}(V)=\xi$, $V\in \mathbf{X}$,  are given by
\begin{equation*}
   V=W^{-\frac{1}{2},\gamma}(\xi)+c_1V_{-\frac{1}{2},\gamma}^1+c_2V_{-\frac{1}{2},\gamma}^2, \quad  c_1,c_2\in \mathbb{R}.
\end{equation*}
Namely,
\begin{equation*}
		V_\theta = W^{-\frac{1}{2},\gamma}_\theta(\xi)+ c_1e^{-a(x)}, \quad V_\phi = W^{-\frac{1}{2},\gamma}_\phi(\xi)+ c_2 \int_{x}^{1}e^{-b(t)}dt,\textrm{ }c_1,c_2\in\mathbb{R}.
	\end{equation*}
\end{cor}
\begin{proof}
  By Lemma \ref{lem4_2_W}, $V-W^{-\frac{1}{2},\gamma}(\xi)$ is in the kernel of $L^{-\frac{1}{2},\gamma}_{0}: \mathbf{X}\to \mathbf{Y}$. The conclusion then follows from Lemma \ref{lem4_7}.
\end{proof}

Let $l_1,l_2$ be the functionals on $\mathbf{X}$ defined by (\ref{def_l}), and $\mathbf{X}_1$ be the subspace of $\mathbf{X}$ defined by (\ref{eq4_2x}). As shown in Section 4.1, the matrix $(l_i(V^{j}_{-\frac{1}{2},\gamma}))$, $i,j=1,2$, is an invertible matrix, for every $\gamma \in K$. So $\mathbf{X}_1$ is a closed subspace of $\mathbf{X}$, and
\begin{equation}\label{eq4_2xb}
	\mathbf{X} = \mbox{span} \{ V_{-\frac{1}{2},\gamma}^{1}, V_{-\frac{1}{2},\gamma}^{2} \} \oplus\mathbf{X}_1,\quad \forall \gamma\in K, 
\end{equation}
with the projection operator $P(\gamma): \mathbf{X}\rightarrow\mathbf{X}_1$ given by 
\begin{equation*}
    P(\gamma)V = V-  l_1(V)V_{-\frac{1}{2},\gamma}^{1}-c(\gamma)l_2(V)V_{-\frac{1}{2},\gamma}^{2} \textrm{ for } V\in \mathbf{X}.
    \end{equation*}
where $c(\gamma)=\left(\int_{0}^{1}e^{-b_{-\frac{1}{2},\gamma}(t)}dt\right)^{-1}>0$ for all $\gamma\in K$.
\begin{lem}\label{lem4_2iso}
The operator $ L^{-\frac{1}{2},\gamma}_0:\mathbf{X}_1 \to \mathbf{Y}$ is an isomorphism.
\end{lem}
\begin{proof}
   By Corollary \ref{cor4_7} and Lemma \ref{lem4_7}, $L^{-\frac{1}{2},\gamma}_0: \mathbf{X}\to \mathbf{Y}$ is surjective and $\ker L_{0} = \mbox{span}\{V^{1}, V^{2}\}$. The conclusion of the lemma then follows in view of the property that $\mathbf{X}=\mbox{span}\{V^1,V^2\}\oplus \mathbf{X}_1$.
\end{proof}

\begin{lem}\label{lem4_8}
$V^1_{-\frac{1}{2},\gamma},V^2_{-\frac{1}{2},\gamma}\in C^{\infty}((-1,\infty),\mathbf{X})$.
\end{lem}
\begin{proof} 
For convenience, in this proof we denote $a(x)=a_{-\frac{1}{2},\gamma}(x)$, $b(x)=b_{-\frac{1}{2},\gamma}(x)$ and $V^i=V^i_{-\frac{1}{2},\gamma}$, $i=1,2$.

  By computation, using the explicit expression of $U^{-\frac{1}{2},\gamma}_{\theta}(x), a(x), a'(x), b(x), V^1_{\theta}(x)$ and $V^2_{\phi}(x)$ given by (\ref{eq_noswirl4_2}), (\ref{eq_ab}), (\ref{eq4_W_a}) and (\ref{eq4_2_ker}), and the estimates of $\partial^i_{\gamma}U^{-\frac{1}{2},\gamma}$ given by (\ref{eqcor3_1_3}) for all $i\ge 0$, we have, for $\gamma\in K$, that
 \[
    e^{-a(x)}=O(1)\left(\ln\frac{1+x}{3}\right)^{-2}, \quad e^{-b(x)}=O(1)\frac{1}{(1+x)\left(\ln\frac{1+x}{3}\right)^{2}}, \quad  -1<x\le 0.
 \]
 and 
 \[
    a'(x)=\frac{2x+U^{-\frac{1}{2},\gamma}_{\theta}(x)}{1-x^2}=O(1)\frac{1}{(1+x)\left(\ln\frac{1+x}{3}\right)}.
 \]
 So
\[
    \left|V^1_{\theta}(x)\right|= O(1)\left(\ln\frac{1+x}{3}\right)^{-2}, \quad V^2_{\phi}(x)=O(1), \quad -1<x\le 0,
\]
and 
\[
   \left|\frac{d}{dx}V^1_{\theta}(x)\right|=\left|e^{-a(x)}a'(x)\right|= O(1)\frac{1}{(1+x)\left(\ln\frac{1+x}{3}\right)^3},\quad -1<x\le 0, 
   \]
   \[
   \left|\frac{d}{dx}V^2_{\phi}(x)\right|=e^{-b(x)}= O(1)\frac{1}{(1+x)\left(\ln\frac{1+x}{3}\right)^{2}}, \quad -1<x\le 0.
\]
Moreover, 
\[
  \begin{split}
  & \frac{\partial^i}{\partial \gamma^i}a(x)=\frac{\partial^i}{\partial \gamma^i}b(x)=\int_{0}^{x}\frac{1}{1-s^2}\frac{\partial^i}{\partial \gamma^i}U^{-\frac{1}{2},\gamma}(s)ds\\
  &=O(1)\int_{0}^{x}\frac{1}{(1+s)\left(\ln\frac{1+x}{3}\right)^{2}} ds=O(1),
   \end{split}
\]
and 
\[
    \frac{\partial^i}{\partial \gamma^i}a'(x)=\frac{\partial^i}{\partial \gamma^i}b'(x)=\frac{1}{1-x^2}\frac{\partial^i}{\partial \gamma^i}U^{-\frac{1}{2},\gamma}(x)=O(1)\frac{1}{(1+x)\left(\ln\frac{1+x}{3}\right)^{2}},
\]
where $|O(1)|\le C$ depending only on $\gamma$ and $i$. So we have
\[
   \left|\partial^i_{\gamma}V^1_{\theta}(x)\right|=e^{-a(x)}O(1)=O(1)\left(\ln\frac{1+x}{3}\right)^{-2}, \quad -1<x\le 0, i=1,2,3...
\]
From the above we can see that for all $\gamma>-1$ and $i\ge 0$, there exists some constant $C$, such that
\[
    \left|\left(\ln\frac{1+x}{3}\right)^{2}\partial^i_{\gamma}V^1_{\theta}(x)\right|\le C,\quad \left|(1+x)\left(\ln\frac{1+x}{3}\right)^{2}\frac{d}{dx}\partial^i_{\gamma}V^1_{\theta}(x)\right|\le C,\quad -1<x\le 0.
\]
We can also show that for $i\ge 0$, 
\[
   \partial^i_{\gamma}V^1_{\theta}(1)=0,
\]
 and there exists some constant $C$ such that
 \[
    \left|\frac{d^l}{dx^l}\partial^i_{\gamma}V^1_{\theta}(x)\right|\le C, \quad l=0,1,2, \quad 0\le x<1.
 \]
 The above imply that  for all $i\ge 0$, $\partial^i_{\gamma}V^1(x)\in \mathbf{X}$, and $V^1_{\theta}\in C^{\infty}((-1,+\infty),\mathbf{M}_1)$.
 
 Similarly, we can show that $V^2_{\phi}\in C^{\infty}((-1,+\infty),\mathbf{M}_2)$. So $V^1,V^2\in C^{\infty}((-1,+\infty),\mathbf{X})$.
\end{proof}

Next, by similar arguments in the proof of Lemma \ref{lem4_1_10}, using Lemma \ref{lem4_8}, we have
\begin{lem}\label{lem4_2_10}
  There exists $C=C(K)>0$ such that for all $\gamma\in K$,  $(\beta_1,\beta_2)\in \mathbb{R}^2$, and $V\in \mathbf{X}_1$, 
  \[
     ||V||_{\mathbf{X}}+|(\beta_1,\beta_2)|\le C||\beta_1V^1_{-\frac{1}{2},\gamma}+\beta_2V^2_{-\frac{1}{2},\gamma}+V||_{\mathbf{X}}.
  \] 
\end{lem}

\noindent \emph{Proof of Theorem \ref{thm4_2}:}
Define a map $F: K\times\mathbb{R}^2\times \mathbf{X}_1\to \mathbf{Y}$ by
\[
    F(\gamma,\beta_1,\beta_2,V)=G(\gamma,\beta_1V^1_{-\frac{1}{2},\gamma}+\beta_2V^2_{-\frac{1}{2},\gamma}+V).
\]
By Proposition \ref{prop4_2}, $G$ is a $C^{\infty}$ map from $K\times \mathbf{X}$ to $\mathbf{Y}$. Let $\tilde{U}=\tilde{U}(\gamma,\beta_1,\beta_2,V)=\beta_1V^1_{-\frac{1}{2},\gamma}+\beta_2V^2_{-\frac{1}{2},\gamma}+V$. Using Lemma \ref{lem4_8}, we have $\tilde{U}\in C^{\infty}(K\times\mathbb{R}^2\times \mathbf{X}_1, \mathbf{X})$. So $F\in C^{\infty}(K\times\mathbb{R}^2\times \mathbf{X}_1,\mathbf{Y})$. 

Next, by definition $F(\gamma,0,0,0)=0$ for all $\gamma\in K$. Fix some $\bar{\gamma}\in K$, using Lemma \ref{lem4_2iso}, we have $F_{V}( \bar{\gamma}, 0,0,0)=L_0^{-\frac{1}{2}, \bar{\gamma}}:\mathbf{X}_1\to \mathbf{Y}$ is an isomorphism.

Applying Theorem C, there exist some $\delta>0$ and a unique $V\in C^{\infty}(B_{\delta}(\bar{\gamma})\times B_{\delta}(0), \mathbf{X}_1 )$, such that
\[
    F(\gamma,\beta_1,\beta_2,V(\gamma,\beta_1,\beta_2))=0, \quad \forall \gamma\in B_{\delta}(\bar{\gamma}), (\beta_1,\beta_2)\in B_{\delta}(0),
\]
and 
\[
   V(\bar{\gamma},0,0)=0.
\]
The uniqueness part of Theorem C holds in the sense that there exists some $0<\bar{\delta}<\delta$,  such that $B_{\bar{\delta}}(\bar{\gamma},0,0,0)\cap F^{-1}(0) \subset  \{(\gamma,\beta_1,\beta_2,V(\gamma,\beta_1,\beta_2))|(\gamma)\in B_{\delta}(\bar{\gamma}), \beta\in B_{\delta}(0)\}$.

\textbf{Claim}: there exists some $0<\delta_1<\frac{\bar{\delta}}{2}$, such that $V(\gamma,0,0)=0$ for every $\gamma\in B_{\delta_1}(\bar{\gamma})$. 

\emph{Proof of the claim:}
   Since $V(\bar{\gamma},0,0)=0$ and $V(\gamma,0,0)$ is continuous in $\gamma$, there exists some $0<\delta_1<\frac{\bar{\delta}}{2}$, such that for all $\gamma\in B_{\delta_1}(\bar{\gamma})$, $(\gamma, 0,0, V(\gamma,0,0))\in B_{\bar{\delta}(\bar{\gamma},0,0,0)}$. We know that for all $\gamma\in B_{\delta_1}(\bar{\gamma})$,
   \[
      F(\gamma, 0,0,0)=0,  
      \]
       and 
       \[
          F(\gamma, 0,0, V(\gamma,0,0))=0.
       \]
By the above mentioned uniqueness result, $V(\gamma,0,0)=0$, for every $\gamma\in B_{\delta_1}(\bar{\gamma})$.

Now we have $V\in C^{\infty}(B_{\delta_1}(\bar{\gamma})\times B_{\delta_1}(0), \mathbf{X}_1 )$, and 
\[
     F(\gamma,\beta_1,\beta_2,V(\gamma,\beta_1,\beta_2))=0, \quad \forall \gamma\in B_{\delta_1}( \bar{\gamma}), (\beta_1,\beta_2)\in B_{\delta_1}(0).
\]
i.e.
\[
    G(\gamma, \beta_1V^1_{-\frac{1}{2}, \gamma}+\beta_2V^2_{-\frac{1}{2}, \gamma}+V(\gamma,\beta_1,\beta_2) )=0, \quad \forall \gamma\in B_{\delta_1}(\bar{\gamma}), (\beta_1,\beta_2)\in B_{\delta_1}(0).
\]
Take derivative of the above with respect to $\beta_i$ at $(\gamma, 0)$, i=1,2, we have
\[
   G_{\tilde{U}}(\gamma,0)(V^i_{-\frac{1}{2},\gamma}+\partial_{\beta_i}V(\gamma,0,0))=0.
\]
Since $G_{\tilde{U}}(\gamma,0)V^i_{-\frac{1}{2},\gamma}=0$ by Lemma \ref{lem4_7}, we have 
\[
   G_{\tilde{U}}(\gamma,0)\partial_{\beta_i}V(\gamma,0,0)=0.
\]
But $\partial_{\beta_i}V(\gamma,0,0)\in \mathbf{X}_1$, so 
\[
    \partial_{\beta_i}V(\gamma,0,0)=0, \quad i=1,2.
\]
Since $K$ is compact, we can take $\delta_1$ to be a universal constant for each  $\gamma\in K$. So we have proved the existence of $V$ in Theorem \ref{thm4_2}.

Next, let $\gamma\in B_{\delta_1}(\bar{\gamma})$. Let $\delta'$ be a small constant to be determined.  For any $U$ satisfying the equation (\ref{eq4_0}) with $U-U^{-\frac{1}{2},\gamma}\in \mathbf{X}$, and $||U-U^{-\frac{1}{2},\gamma}||_{ \mathbf{X}}\le \delta'$ there exist some $\beta_1,\beta_2\in \mathbb{R}$ and $V^*\in \mathbf{X}_1$ such that
 \[
     U-U^{-\frac{1}{2},\gamma}=\beta_1V^1_{-\frac{1}{2},\gamma}+\beta_2V^2_{-\frac{1}{2},\gamma}+V^*.
 \]
Then by Lemma \ref{lem4_2_10}, there exists some constant $C>0$ such that
\[
    \frac{1}{C}(|(\beta_1,\beta_2)|+||V^*||_{\mathbf{X}})\le ||\beta_1V^1_{-\frac{1}{2},\gamma}+\beta_2V^2_{-\frac{1}{2},\gamma}+V^*||_{\mathbf{X}}\le \delta'.
\]
This gives $||V^*||_{\mathbf{X}}\le C\delta'$.

 Choose $\delta'$ small enough such that $C\delta'<\delta_1$. We have the uniqueness of $V^*$. So  $V^*=V(\gamma,\beta_1,\beta_2)$ in (\ref{eq_thm4_2_1}). The theorem is proved.  \qed

\subsection{Existence of solutions with nonzero swirl near $U^{\mu,\gamma}$ when $(\mu, \gamma)\in I_3\cap \{ - \frac{1}{2} \le \mu < - \frac{3}{8} \}$}

Next we look at the problem near $U^{\mu,\gamma}$ when $\mu\ge -\frac{1}{2}$ and $\gamma=-(1+\sqrt{1+2\mu})$. For such a fixed $(\mu,\gamma)$, write $\bar{U}=U^{\mu,\gamma}$. Recall that in Corollary \ref{cor3_1} we have
\begin{equation}\label{eq4_3_0}
  \bar{U}_{\theta}=(1-x)(1+\sqrt{1+2\mu}).
\end{equation}
It satisfies 
\[
     (1-x^2)\bar{U}'_{\theta}+2x\bar{U}_{\theta}+\frac{1}{2}\bar{U}^2_{\theta}=\mu (1-x)^2.
\]

We will work with $\tilde{U}=U-\bar{U}$.  Given a compact subset $K\in (-\frac{1}{2}, -\frac{3}{8})$ or $K=\{-\frac{1}{2}\}$, there exists  an $\epsilon>0$, depending only on $K$, satisfying $\displaystyle \max_{\mu\in K}\sqrt{1+2\mu}<\epsilon<\frac{1}{2}$. For this fixed $\epsilon$, define

\begin{equation*}
 \begin{split}
 &
 \begin{split}
 \mathbf{M}_1 = & \mathbf{M}_1(\epsilon) \\
:= & \left\{  \tilde{U}_\theta \in C([-1, 1], \mathbb{R}) \cap C^1((-1, 1], \mathbb{R}) \cap C^2((0, 1), \mathbb{R}) \mid  \tilde{U}_\theta(1)=\tilde{U}_\theta(-1)=0,\right.\\
	 & \left.  ||(1+x)^{-1+2\epsilon}\tilde{U}_\theta||_{L^\infty(-1,1)}<\infty,  ||(1+x)^{2\epsilon}\tilde{U}'_\theta||_{L^\infty(-1,1)} < \infty, ||\tilde{U}_\theta''||_{L^\infty(0,1)} < \infty  \right\},
  \end{split}\\
   &
  \begin{split}
\mathbf{M}_2 = & \mathbf{M}_2(\epsilon) \\
:=& \left\{  \tilde{U}_\phi \in C^1( (-1, 1], \mathbf{R})\cap C^2( (-1, 1), \mathbf{R}) \mid \tilde{U}_\phi(1)=0, ||(1+x)^{\epsilon}\tilde{U}_\phi ||_{L^\infty(-1,1)} < \infty, \right.\\
	& \left. ||(1+x)^{1+\epsilon} \tilde{U}_\phi'||_{L^\infty(-1,1)} < \infty, ||(1+x)^{2+\epsilon} \tilde{U}_\phi'' ||_{L^\infty(-1,1)} <\infty \right\} 
   \end{split}
\end{split}
\end{equation*}
 with the following norms accordingly:
\begin{equation*}
 \begin{split}
    & ||\tilde{U}_\theta||_{\mathbf{M}_1}:= ||(1+x)^{-1+2\epsilon}\tilde{U}_\theta||_{L^\infty(-1,1)} + ||(1+x)^{2\epsilon}\tilde{U}_\theta'||_{L^\infty(-1,1)} + ||\tilde{U}_\theta''||_{L^\infty(0,1)}, \\
   &	||\tilde{U}_\phi||_{\mathbf{M}_2}:=  ||(1+x)^{\epsilon}\tilde{U}_\phi||_{L^\infty(-1,1)} + ||(1+x)^{1+\epsilon} \tilde{U}_\phi'||_{L^\infty(-1,1)}  + ||(1+x)^{2+\epsilon} \tilde{U}_\phi'' ||_{L^\infty(-1,1)}.
 \end{split}
\end{equation*}
Next, define
\begin{equation*}
  \begin{split}
  &\begin{split}
   \mathbf{N}_1= \mathbf{N}_1(\epsilon):=  & \left\{  \xi_\theta \in C( (-1, 1], \mathbb{R}) \cap C^1( (0, 1], \mathbb{R}) \mid  \xi_\theta(1)=\xi'_{\theta}(1)=\xi_{\theta}(-1)=0, \right.\\
    & \left.  ||(1+x)^{-1+2\epsilon}\xi_\theta||_{L^\infty(-1,1)} < \infty, 
          ||\frac{\xi_\theta'}{1-x}||_{L^\infty(0,1)} < \infty \right\},
   \end{split}\\
  &   \mathbf{N}_2=\mathbf{N}_2(\epsilon):= \left\{  \xi_\phi \in C( (-1, 1], \mathbb{R}) \mid  \xi_\phi(1)=0, ||\frac{(1+x)^{1+\epsilon} \xi_\phi}{1-x}||_{L^\infty(-1,1)} < \infty  \right\} 
  \end{split}
\end{equation*}
with the following norms accordingly:
\begin{equation*}
  \begin{split}
   & ||\xi_\theta||_{\mathbf{N}_1}:= ||(1+x)^{-1+2\epsilon}\xi_\theta||_{L^\infty(-1,1)} + ||\frac{\xi_\theta'}{1-x}||_{L^\infty(0,1)}, \\
   &||\xi_\phi||_{\mathbf{N}_2}:= ||\frac{(1+x)^{1+\varepsilon} \xi_\phi}{1-x}||_{L^\infty(-1,1)}.
  \end{split}
\end{equation*}

Let $\mathbf{X}:= \{ \tilde{U} = (\tilde{U}_\theta, \tilde{U}_\phi) \mid \tilde{U}_\theta\in \mathbf{M}_1, \tilde{U}_\phi\in \mathbf{M}_2\}$ with the norm $||\tilde{U}||_\mathbf{X}:= ||\tilde{U}_\theta||_{\mathbf{M}_1} + ||\tilde{U}_\phi||_{\mathbf{M}_2}$, and $\mathbf{Y}:= \{ \xi = (\xi_\theta, \xi_\phi) \mid \xi_\theta\in \mathbf{N}_1, \xi_\phi\in \mathbf{N}_2 \}$ with the norm $||\xi||_\mathbf{Y}:= ||\xi_\theta||_{\mathbf{N}_1} + ||\xi_\phi||_{\mathbf{N}_2}$. It is not difficult to verify that $\mathbf{M}_1$, $\mathbf{M}_2$, $\mathbf{N}_1$, $\mathbf{N}_2$, $\mathbf{X}$ and $\mathbf{Y}$ are Banach spaces.

Let $l_2:\mathbf{X}\to \mathbb{R}$ be the bounded linear functional defined by (\ref{def_l}) for each $V\in \mathbf{X}$. Define 
\begin{equation}\label{eq4_3x}
	\mathbf{X}_1:= \ker l_2.
\end{equation}

\begin{thm}\label{thm4_3}
  For every  compact subset  $K$ of $(-\frac{1}{2},-\frac{3}{8})$ or $K=\{-\frac{1}{2}\}$,  there exist $\delta=\delta(K)>0$, and $V\in C^{\infty}(K\times B_{\delta}(0), \mathbf{X}_1)$ satisfying $V(\mu,0)=0$ and $\displaystyle \frac{\partial V}{\partial \beta}|_{\beta=0}=0$, such that 
\begin{equation}\label{eq_thm4_3_1}
   U=U^{\mu, -1-\sqrt{1+2\mu}}+\beta V_{\mu,-1-\sqrt{1+2\mu}}^2+V(\mu, \beta)
\end{equation}
satisfies  equation (\ref{eq4_0})  with $\displaystyle \hat{\mu}=\mu-\frac{1}{4}\psi[U_{\phi}](-1)$.  Moreover, there exists some $\delta'=\delta'(K)>0$, such that if $||U-U^{\mu,-1-\sqrt{1+2\mu}}||_{\mathbf{X}}<\delta'$, $\mu\in K$,  and $U$ satisfies  equation (\ref{eq4_0}) with some constant $\hat{\mu}$, then (\ref{eq_thm4_3_1}) holds for some $|\beta|<\delta$ .
\end{thm}

To prove Theorem \ref{thm4_3}, we first study the properties of the Banach spaces $\mathbf{X}$ and $\mathbf{Y}$.

With the fixed $\epsilon$, we have

\begin{lem}\label{lem4_3_1}
   For every $\tilde{U}\in \mathbf{X}$, it satisfies 
   \begin{equation}\label{eq4_3_1}
      |\tilde{U}_{\phi}(s)|\le (1-s)(1+s)^{-\epsilon}||\tilde{U}_{\phi}||_{\mathbf{M}_2},  \quad \forall -1<s<1,
   \end{equation}
   \begin{equation}\label{eq4_3_2}
      |\tilde{U}_{\theta}(s)|\le (1-s)(1+s)^{1-2\epsilon}||\tilde{U}_{\theta}||_{\mathbf{M}_1},  \quad \forall -1<s<1.
   \end{equation}
\end{lem}
\begin{lem}\label{lem4_3_2}
  For every $\xi_{\theta}\in \mathbf{N}_1$, 
  \begin{equation}\label{eq4_3_3}
     |\xi_{\theta}(s)|\le (1-s)^2(1+s)^{1-2\epsilon}||\xi_{\theta}||_{\mathbf{N}_1}, \quad \forall -1<s<1.
  \end{equation}
\end{lem}

Now let $K$ be a compact subset of $(-\frac{1}{2},-\frac{3}{8})$ or $K=\{-\frac{1}{2}\}$. For $\tilde{U}_{\phi}\in \mathbf{M}_2$, let $\psi[\tilde{U}_{\phi}](x)$ be defined by (\ref{eqpsi4_1}). Then define a map $G$ on $K\times\mathbf{X}$ such that for each $(\mu,\tilde{U})\in K\times\mathbf{X}$, $G(\mu,\tilde{U})=G(\mu,-1-\sqrt{1+2\mu},\tilde{U})$ given by (\ref{G}) with $\bar{U}_{\theta}$ in (\ref{eq4_3_0}). If $\tilde{U}$ satisfies $G(\mu, \tilde{U})=0$, then $U=\tilde{U}+\bar{U}$ gives a solution of (\ref{eq4_0}) with $\hat{\mu}=\mu - \frac{1}{4}\psi[\tilde{U}_{\phi}](-1)$, satisfying $U_{\theta}(-1)=\bar{U}_{\theta}(-1)$.

\begin{prop}\label{prop4_3}
The map $G$ is in $C^{\infty}(K\times \mathbf{X},\mathbf{Y})$ in the sense that $G$ has continuous Fr\'{e}chet derivatives of every order.  Moreover, the Fr\'{e}chet derivative of $G$ with respect to $\tilde{U}$ at $(\mu,\tilde{U})\in K\times \mathbf{X}$ is given by the linear operator $L^{\mu}_{\tilde{U}}:\mathbf{X}\rightarrow \mathbf{Y}$ where $L^{\mu}:=L^{\mu,-1-\sqrt{1+2\mu}}$ defined as in (\ref{eq4_Linear}).
\end{prop}

To prove Proposition \ref{prop4_3}, we first have the following lemmas:
\begin{lem}\label{lemP4_3_1}
For every $\mu\in K$, the map $A(\mu,-1-\sqrt{1+2\mu},\cdot):\mathbf{X} \to \mathbf{Y}$  defined by (\ref{eqP4_2}) is a bounded linear operator.
\end{lem}
\begin{proof}
For convenience we denote $A=A(\mu,-1-\sqrt{1+2\mu},\cdot)$. We make use of the properties of $\bar{U}_{\theta}$ that $\bar{U}_{\theta}(1)=0$ and $\bar{U}_{\theta}\in C^2(-1,1]\cap L^{\infty}(-1,1)$.

 $A$ is clearly linear.  For every $\tilde{U}\in\mathbf{X}$, we prove that $A\tilde{U}$ defined by (\ref{eqP4_2}) is in $\mathbf{Y}$ and there exists some constant $C$ such that $||A\tilde{U}||_{\mathbf{Y}}\le C||\tilde{U}||_{\mathbf{X}}$ for all $\tilde{U}\in \mathbf{X}$.
 
 By the fact that $\tilde{U}_{\theta}\in \mathbf{M}_1$ and (\ref{eq4_3_1}),  we have
\[
  \left|(1+x)^{-1+2\epsilon}A_{\theta}\right|  \le (1-x)(1+x)^{2\epsilon}|\tilde{U}'_{\theta}|+(2+|\bar{U}_{\theta}|)(1+x)^{-1+2\epsilon}|\tilde{U}_{\theta}|
      \le C(1-x)||\tilde{U}_{\theta}||_{\mathbf{M}_1}.
\]
We also see from the above that $\lim_{x\to 1}A_{\theta}(x)=\lim_{x\to -1}A_{\theta}(x)=0$. By computation $A'_{\theta}=(1-x^2)\tilde{U}''_{\theta}+\bar{U}_{\theta}\tilde{U}'_{\theta}+(2+\bar{U}'_{\theta})\tilde{U}_{\theta}$. Then by (\ref{eq4_3_0}), (\ref{eq4_3_2}) and the fact that $\tilde{U}_{\theta}\in \mathbf{M}_1$, 
\[
   \frac{|A'_{\theta}(x)|}{1-x}\le C||\tilde{U}_{\theta}||_{\mathbf{M}_1}, \quad 0<x<1.
\]
So $A_{\theta}\in \mathbf{N}_1$ and $||A_{\theta}||_{\mathbf{N}_1}\le C ||\tilde{U}_{\theta}||_{\mathbf{M}_1}$.

Next, by the fact that $\tilde{U}_{\phi}\in \mathbf{M}_2$ and (\ref{eq4_3_0}), with similar arguments in the proof of Lemma \ref{lemP4_2_2'}, we have
\[
   \frac{(1+x)^{1+\epsilon}}{1-x}|A_{\phi}|\le C||\tilde{U}_{\phi}||_{\mathbf{M}_2},\quad -1<x<1.
\]
In particular, $\lim_{x\to 1}A_{\phi}(x)=0$. So $A_{\phi}\in \mathbf{N}_2$, and $||A_{\phi}||_{\mathbf{N}_2}\le  C||\tilde{U}_{\phi}||_{\mathbf{M}_2}$.
 We have proved that $A\tilde{U}\in\mathbf{Y}$, and $||A\tilde{U}||_{\mathbf{Y}}\le C||\tilde{U}||_{\mathbf{X}}$ for every $\tilde{U}\in \mathbf{X}$.
\end{proof}

\begin{lem}\label{lemP4_3_2}
The map $ Q:\mathbf{X}\times \mathbf{X}\to \mathbf{Y}$  defined by (\ref{eqP4_3}) is a  bounded bilinear operator.
\end{lem}
\begin{proof}
 It is clear that $Q$ is a bilinear operator. For every $\tilde{U},\tilde{V}\in\mathbf{X}$, we will prove that $Q(\tilde{U},\tilde{V})$ is in $\mathbf{Y}$ and there exists some constant $C$ independent of $\tilde{U}$ and $\tilde{V}$ such that $||Q(\tilde{U},\tilde{V})||_{\mathbf{Y}}\le C||\tilde{U}||_{\mathbf{X}}||\tilde{V}||_{\mathbf{X}}$.
 
For convenience we write
\[
  \psi(\tilde{U},\tilde{V})(x)=\int_x^1 \int_l^1 \int_t^1 \frac{2 \tilde{U}_\phi(s)\tilde{V} _\phi'(s)}{1-s^2} ds dt dl.
\]
 By the same proof as that of Lemma \ref{lemP4_2_3'}, for $\tilde{U}_{\phi},\tilde{V}_{\phi}\in \mathbf{M}_2$,  we have for any $-1<x<1$
\begin{equation}\label{eq4_3_5}
  |\psi(\tilde{U},\tilde{V})(x)-\frac{(1-x)^2}{4}\psi(\tilde{U},\tilde{V})(-1)|\le C(\epsilon)(1+x)^{1-2\epsilon}(1-x)^2||\tilde{U}_\phi||_{\mathbf{M}_2}||\tilde{V} _{\phi}||_{\mathbf{M}_2}. 
\end{equation}
So by (\ref{eq4_3_2}), (\ref{eq4_3_5}) and the fact that $\tilde{U}_{\theta},\tilde{V}_{\theta}\in \mathbf{M}_1$, we have
\[
  \begin{split}
  & |(1+x)^{-1+2\epsilon}Q_{\theta}(x)| \\
  & \le \frac{1}{2}(1+x)^{-1+2\epsilon}|\tilde{U}_{\theta}(x)||\tilde{V}_{\theta}(x)|+(1+x)^{-1+2\epsilon}|\psi(\tilde{U},\tilde{V})(x)-\frac{(1-x)^2}{4}\psi(\tilde{U},\tilde{V})(-1)|\\
                  & \le C(1-x)^2||\tilde{U}_{\theta}||_{\mathbf{M}_1}||\tilde{V}_{\theta}||_{\mathbf{M}_1}+ C(1-x)^2||\tilde{U}_{\phi}(s)||_{\mathbf{M}_2}||\tilde{V} _{\phi}||_{\mathbf{M}_2}\\
                  & \le C(1-x)^2||\tilde{U}||_{\mathbf{X}}||\tilde{V}||_{\mathbf{X}}, \quad \forall -1<x<1.
\end{split}
\]
Since $\epsilon<\frac{1}{2}$, from the above we also see that $\lim_{x\to 1}Q_{\theta}(x)=\lim_{x\to -1}Q_{\theta}(x)=0$.

Using (\ref{eq4_3_1}), (\ref{eq4_3_2}) and the fact that $\tilde{U}\in \mathbf{X}$, with the same argument in the proof of Lemma \ref{lemP4_2_3'}, it can be shown that
\[
   |Q'_{\theta}(x)|\le C(1-x)||\tilde{U}||_{\mathbf{X}}||\tilde{V}||_{\mathbf{X}},\quad \forall 0<x<1.
\]
So $Q_{\theta}\in\mathbf{N}_1$, and $||Q_{\theta}||_{\mathbf{N}_1}\le C||\tilde{U}||_{\mathbf{X}}||\tilde{V}||_{\mathbf{X}}$.

Next, using (\ref{eq4_3_2}) and similar proof of Lemma \ref{lemP4_2_3'}, we can prove
\[
   \left|\frac{(1+x)^{1+\epsilon}Q_{\phi}}{1-x}\right|  \le C||\tilde{U}_{\theta}||_{\tilde{\mathbf{M}}_1}||\tilde{V}_{\phi}||_{\tilde{\mathbf{M}}_2}, \quad -1<x<1, 
\]
and $\displaystyle \lim_{x\to 1}Q_{\phi}(x)=0$. So $Q_{\phi}\in \mathbf{N}_2$, and $||Q_{\phi}||_{\mathbf{N}_2}\le||\tilde{U}_{\theta}||_{\mathbf{M}_1}||\tilde{V}_{\phi}||_{\mathbf{M}_2}$. Thus we have proved $Q(\tilde{U},\tilde{V})\in \mathbf{Y}$ and $||Q(\tilde{U},\tilde{V})||_{\mathbf{Y}}\le C||\tilde{U}||_{\mathbf{X}}||\tilde{V}||_{\mathbf{X}}$ for all $\tilde{U}, \tilde{V}\in \mathbf{X}$. The proof is finished.
\end{proof}

\noindent \emph{Proof of Proposition \ref{prop4_3}:}
By definition, $G(\mu,\tilde{U})=A(\mu,-1-\sqrt{1+2\mu},\tilde{U})+Q(\tilde{U},\tilde{U})$ for $(\mu,\tilde{U})\in K\times \mathbf{X}$.  Using standard theories in functional analysis, by Lemma \ref{lemP4_3_2} it is clear that $Q$ is $C^{\infty}$ on $K\times  \mathbf{X}$.  By Lemma \ref{lemP4_3_1}, $A(\mu,-1-\sqrt{1+2\mu}, \cdot): \mathbf{X}\to \mathbf{Y}$ is $C^{\infty}$ for each $\mu \in K$. For all $i\ge 1$, we have
\[
    \partial_{\mu}^i A(\mu,-1-\sqrt{1+2\mu},\tilde{U})= \partial_{\mu}^i U^{\mu,-1-\sqrt{1+2\mu}}_{\theta}\left(
	\begin{matrix}
		\tilde{U}_{\theta}  \\  \tilde{U}'_{\phi}
	\end{matrix}\right).
\]  
By (\ref{eqcor3_1_4}), for each integer $i\ge 1$, there exists some constant $C=C(i,K)$, depending only on $i,K$, such that 
\begin{equation}\label{eq_prop4_3_1}
   |\partial_{\mu}^i U^{\mu,-1-\sqrt{1+2\mu}}_{\theta}(x)|\le C(i,K)(1-x), \quad -1<x<1.
\end{equation}
From (\ref{eq4_3_0}) we can also obtain
\[
     \left|\frac{d}{dx}\partial_{\mu}^i U^{\mu,-1-\sqrt{1+2\mu}}_{\theta}(x)\right|\le C(i,K), \quad 0<x<1.
\]
Using the above estimates and the fact that $\tilde{U}_{\theta}\in \mathbf{M}_1$, we have
\[
   \left|(1+x)^{-1+2\epsilon}\partial_{\mu}^i A_{\theta}(\mu,-1-\sqrt{1+2\mu},\tilde{U})\right|\le C(i,K)(1-x)||\tilde{U}_{\theta}||_{\mathbf{M}_1}, \quad -1<x<1,
\]
and 
\[
\begin{split}
  & \left|\frac{d}{dx}\partial_{\mu}^i A_{\theta}(\mu,-1-\sqrt{1+2\mu},\tilde{U})\right| \\
 \le & \left|\frac{d}{dx}\partial_{\mu}^i U^{\mu,-1-\sqrt{1+2\mu}}_{\theta}(x)\right||\tilde{U}_{\theta}(x)|+|\partial_{\mu}^i U^{\mu,-1-\sqrt{1+2\mu}}_{\theta}(x)|\left|\frac{d}{dx}\tilde{U}_{\theta}(x)\right|\\
   \le &  C(i,K)(1-x)||\tilde{U}_{\theta}||_{\mathbf{M}_1}, \quad 0<x<1.
   \end{split}
\]
So $\partial_{\mu}^i A_{\theta}(\mu,-1-\sqrt{1+2\mu},\tilde{U})\in \mathbf{N}_1$, with $||\partial_{\mu}^i A_{\theta}(\mu,-1-\sqrt{1+2\mu},\tilde{U})||_{\mathbf{N}_1}\le C(i,K)||\tilde{U}_{\theta}||_{\mathbf{M}_1}$ for all $(\mu,\tilde{U})\in K\times \mathbf{X}$.

Next, by (\ref{eq_prop4_3_1}) and the fact that $\tilde{U}_{\phi}\in \mathbf{M}_1$, we have
\[
   \frac{(1+x)^{1+\epsilon}}{1-x}|\partial_{\mu}^i A_{\phi}(\mu,-1-\sqrt{1+2\mu},\tilde{U})|=\frac{|\partial_{\mu}^i U^{\mu,-1-\sqrt{1+2\mu}}_{\theta} |}{1-x}|(1+x)^{1+\epsilon}U'_{\phi}|\le C(i,K)||\tilde{U}_{\phi}||_{\mathbf{M}_2}.
\]
So $\partial_{\mu}^i A_{\phi}(\mu,-1-\sqrt{1+2\mu},\tilde{U})\in \mathbf{N}_2$, with 
$$
	||\partial_{\mu}^i A_{\phi}(\mu,-1-\sqrt{1+2\mu},\gamma,\tilde{U})||_{\mathbf{N}_2}\le C(i,K)||\tilde{U}_{\phi}||_{\mathbf{M}_2}
$$ 
for all $(\mu,\tilde{U})\in K\times \mathbf{X}$. Thus $\partial_{\mu}^iA(\mu,-1-\sqrt{1+2\mu},\tilde{U})\in \mathbf{Y}$, with 
$$
	||\partial_{\mu}^i A(\mu,-1-\sqrt{1+2\mu},\tilde{U})||_{\mathbf{Y}} \le C(i,K)||\tilde{U}||_{\mathbf{X}}
$$
for all $(\mu,\tilde{U})\in K\times \mathbf{X}$, $i\ge 1$. 

So for each $\mu \in K$, $\partial_{\mu}^i A(\mu,-1-\sqrt{1+2\mu},  \cdot):\mathbf{X}\to \mathbf{Y}$ is a bounded linear map with uniform bounded norm on $K$. Then by standard theories in functional analysis, $A:K\times\mathbf{X}\to \mathbf{Y}$ is $C^{\infty}$. So $G$ is a $C^{\infty}$ map from $K\times\mathbf{X}$ to $\mathbf{Y}$. By direct calculation we get its Fr\'{e}chet derivative with respect to $\mathbf{X}$ is given by  the linear bounded operator $L^{\mu,-1-\sqrt{1+2\mu}}_{\tilde{U}}: \mathbf{X}\rightarrow \mathbf{Y}$ defined as  (\ref{eq4_Linear}). The proof is finished.   \qed \\

By Proposition \ref{prop4_3},  $L^{\mu}_0: \mathbf{X}\to \mathbf{Y}$, the Fr\'{e}chet derivative of $G$ at $\tilde{U}=0$ is given by (\ref{eq_LinearAtZero}).

Next, let $a_{\mu}(x)=a_{\mu, -1-\sqrt{1+2\mu}}(x)$, $b_{\mu}(x)=b_{\mu, -1-\sqrt{1+2\mu}}(x)$ be the functions defined by (\ref{eq_ab}) with $\bar{U}_{\theta}$ given by (\ref{eq4_3_0}).

Since $\bar{U}=(1-x)(1+\sqrt{1+2\mu})$, we have
\begin{equation}\label{eq4_3_ab}
 \begin{split}
   & a_{\mu}(x)=-\ln(1-x^2)+(1+\sqrt{1+2\mu})\ln(1+x), \\
   &b_{\mu}(x)=(1+\sqrt{1+2\mu})\ln(1+x).
 \end{split}
\end{equation}
For $\xi = (\xi_\theta, \xi_\phi) \in \mathbf{Y}$,  by (\ref{eq4_3_3}) and (\ref{eq4_3_ab}), we have
\[
    \int_{-1}^{1}e^{a_{\mu}(s)}\frac{|\xi_{\theta}(s)|}{1-s^2}ds\le ||\xi_{\theta}||_{\mathbf{N}_1}\int_{-1}^{1}(1+s)^{\sqrt{1+2\mu}-2\epsilon}ds<\infty.
\]
Let the map $W^{\mu}$ be defined as $W^{\mu}(\xi):=(W^{\mu}_{\theta}(\xi), W^{\mu}_{\phi}(\xi))$ by
\begin{equation}\label{eq4_3_W}
  \begin{split}
   & W^{\mu}_{\theta}(\xi)(x)= e^{-a_{\mu}(x)}\int_{-1}^{x}e^{a_{\mu}(s)}\frac{\xi_{\theta}(s)}{1-s^2}ds,\\
   & W^{\mu}_{\phi}(\xi)(x)=\int_{x}^{1}e^{-b_{\mu}(t)}\int_{t}^{1}e^{b_{\mu}(s)}\frac{\xi_{\phi}(s)}{1-s^2}dsdt.
   \end{split}
\end{equation}
Then $W^{\mu}$ satisfies (\ref{eq4_W_d}).

\begin{lem} \label{lem4_3_W}
   $W^{\mu}:\mathbf{Y}\to \mathbf{X}$ is  continuous and is a right inverse of $L^{\mu}_{0}$.
\end{lem}
\begin{proof}
For convenience we write $W=W^{\mu}(\xi)$, $a(x)=a_{\mu}(x)$ and $b(x)=b_{\mu}(x)$.

We first prove that $W$ is well-defined. For $\xi\in \mathbf{Y}$, denote $W:=W(\xi)$. Applying Lemma \ref{lem4_3_2} in the expression of $W_{\theta}$ in (\ref{eq4_3_W}),  we have
\begin{equation}\label{eq4_3_6}
   \left|(1+x)^{-1+2\epsilon}W_{\theta}(x)\right|\le C(1+x)^{-1+2\epsilon}||\xi_{\theta}||_{\mathbf{N}_1}e^{-a(x)}\int_{-1}^{x}e^{a(s)}(1-s)(1+s)^{-2\epsilon}ds,\quad -1<x<1.
\end{equation}
Using (\ref{eq4_3_ab}), we have
\[
    e^{a(s)}=(1+s)^{\sqrt{1+2\mu}}(1-s)^{-1}, \quad e^{-a(x)}=(1+s)^{-\sqrt{1+2\mu}}(1-s), \quad -1<s<x<1.
\]
Apply this in (\ref{eq4_3_6}), it is not hard to see that
\begin{equation}\label{eq4_3_7}
   |W_{\theta}(x)|\le C||\xi_{\theta}||_{\mathbf{N}_1}(1+x)^{1-2\epsilon}(1-x), \quad -1<x\le 1.
\end{equation}
In particular $W_{\theta}(1)=0$. Since $\epsilon<\frac{1}{2}$, $\displaystyle \lim_{x\to -1}W_{\theta}(x)=0$, 

By (\ref{eq4_3_ab}),  
\begin{equation}\label{eq4_3_8}
   |a'(x)|\le \frac{C}{1-x^2}, \quad |a''(x)|\le \frac{C}{(1-x^2)^2}, \quad -1<x<1.
\end{equation}
Using the above estimate of $|a'(x)|$, (\ref{eq4_3_3}), (\ref{eq4_3_7}) and (\ref{eq4_W_d}), we have
\[
    |(1+x)^{2\epsilon}W'_{\theta}|\le  (1+x)^{2\epsilon}|a'(x)||W_{\theta}(x)|+\frac{|\xi_{\theta}(x)|(1+x)^{2\epsilon}}{1-x^2}\le C||\xi_{\theta}||_{\mathbf{N}_1}, \quad -1<x<1.
\]

Next, A calculation gives
\[
  W''_{\theta}(x)  =((a'(x))^2-a''(x))W_{\theta}(x)-a'(x)\frac{\xi_{\theta}(x)}{1-x^2}+\frac{\xi'_{\theta}(x)}{1-x^2}+\frac{2x\xi_{\theta}(x)}{(1-x^2)^2}.
\]
So
	\begin{equation*}
		|W''_{\theta}(x)|  \le  |(a'(x))^2-a''(x)| |W_{\theta}| + |a'(x)|\frac{|\xi_{\theta}|}{1-x^2} + \frac{|\xi'_{\theta}|}{(1-x)}+\frac{|\xi_{\theta}|}{(1-x)^2}.
	\end{equation*}
	By (\ref{eq4_3_ab}), we have the estimate
	\[
	    (a'(x))^2-a''(x)=O\left(\frac{1}{1-x}\right).
	\]
It follows, using (\ref{eq4_3_7}),  (\ref{eq4_3_8}) and  Lemma \ref{lem4_3_2}, that
\[
   |W''_{\theta}(x)|\le  C\left(\frac{|W_{\theta}(x)}{1-x}+\frac{|\xi_{\theta}|}{(1-x)^2}+\frac{|\xi'_{\theta}|}{1-x}\right)\le C||\xi_{\theta}||_{\mathbf{N}_1}, \quad 0<x<1.
\]

So we have shown that $W_{\theta}\in \mathbf{M}_1$, and $||W_{\theta}||_{\mathbf{M}_1}\le C||\xi_{\theta}||_{\mathbf{N}_1}$ for some constant $C$.

By definition of $W_{\phi}(\xi)$ in (\ref{eq4_3_W}) and the fact that $\xi_{\phi}\in \mathbf{N}_2$, we have, for every $-1<x<1$, that
\[
    |W_{\phi}(x)|\le  ||\xi_{\phi}||_{\mathbf{N}_2} \int_{x}^{1}e^{-b(t)}\int_{t}^{1}e^{b(s)}(1+s)^{-2-\epsilon}dsdt.
\]
Using (\ref{eq4_3_ab}), we have
\begin{equation}\label{eq4_3_9}
    e^{b(s)}=(1+s)^{1+\sqrt{1+2\mu}},\quad e^{-b(t)}=(1+t)^{-1-\sqrt{1+2\mu}}, \quad -1<s,t<1.
\end{equation}
So we have, using $\sqrt{1+2\mu}<\epsilon<\frac{1}{2}$, 
\[
 \begin{split}
    |W_{\phi}(x)| & \le C||\xi_{\phi}||_{\mathbf{N}_2}\int_{x}^{1}(1+t)^{-1-\sqrt{1+2\mu}}\int_{t}^{1}(1+s)^{\sqrt{1+2\mu}-1-\epsilon}dsdt\\
       & \le C||\xi_{\phi}||_{\mathbf{N}_2}(1+x)^{-\epsilon}, \quad -1<x\le 1.
       \end{split}
\]
For $0<x<1$, it can be seen from the above that $|W_{\phi}(x)|\le C||\xi_{\phi}||_{\mathbf{N}_2}(1-x)$. In particular, $W_{\phi}(1)=0$.  By computation
\[
    W'_{\phi}(x)=-e^{-b(x)}\int_{x}^{1}e^{b(s)}\frac{\xi_{\phi}(s)}{1-s^2}ds.
\]
Using (\ref{eq4_3_9}), $\epsilon>\sqrt{1+2\mu}$ and the fact that $\xi_{\phi}\in \mathbf{N}_2$, we have,
\[
   |(1+x)^{1+\epsilon}W'_{\phi}(x)|\le  C||\xi_{\phi}||_{\mathbf{N}_2}\quad -1<x<1.
\]
Similarly, 
\[
  W''_{\phi}(x)=b'(x)e^{-b(x)}\int_{x}^{1}e^{b(s)}\frac{\xi_{\phi}(s)}{1-s^2}ds+\frac{\xi_{\phi}(x)}{1-x^2}.
\]
By (\ref{eq4_3_ab}), $b'(x)=\frac{1+\sqrt{1+2\mu}}{1+x}=O((1+x)^{-1})$. Using (\ref{eq4_3_9}), we have
\[
   |(1+x)^{2+\epsilon}W''_{\phi}(x)|\le C||\xi_{\phi}||_{\mathbf{N}_2},\quad -1<x<1.
\]
So $W_{\phi}\in \mathbf{M}_2$, and $||W_{\phi}||_{\mathbf{M}_2}\le C||\xi_{\phi}||_{\mathbf{N}_2}$ for some constant $C$. 

Thus $W^{\mu}(\xi)\in \mathbf{X}$ for all $\xi\in \mathbf{Y}$, and $||W(\xi)||_{\mathbf{X}}\le C||\xi||_{\mathbf{Y}}$ for some constant $C$. So $W^{\mu}: \mathbf{X}\to \mathbf{Y}$ is well-defined and continuous. It can be directly checked that $W^{\mu}$ is a right inverse of $L^{\mu}_0$.
\end{proof}

Let $V^i_{\mu}:=V^i_{\mu,-1-\sqrt{1+2\mu}}$, $i=1,2,3$,  be defined by (\ref{eq4_2_ker}) with related $a_{\mu,-1-\sqrt{1+2\mu}}=a_{\mu}(x)$ and $b_{\mu,-1-\sqrt{1+2\mu}}=b_{\mu}(x)$ given by (\ref{eq4_3_ab}) , we have
\begin{lem}\label{lem4_3_6}
$\{V_{\mu}^2\}$ is a basis of the kernel of $L^{\mu}_0:\mathbf{X}\to \mathbf{Y}$.
\end{lem}
\begin{proof}
   By (\ref{eq4_3_ab}), it is not hard to verify that $V_{\mu}^2\in \mathbf{X}$, and  $ V_{\mu}^1,V_{\mu}^3\notin\mathbf{X}$. Then by similar proof as Lemma \ref{lem4_3}, we obtain the conclusion.
\end{proof}
\begin{cor}\label{cor4_3_7}
For any $\xi=(\xi_{\theta}, \xi_{\phi})\in \mathbf{Y}$, all solutions  of $L^{\mu}_{0}(V)=\xi$, $V\in \mathbf{X}$,  are given by
\begin{equation*}
   V=W^{\mu}(\xi)+cV_{\mu}^2, \quad  c\in \mathbb{R}.
\end{equation*}
Namely,
\begin{equation*}
		V_\theta = W^{\mu}_\theta(\xi), \quad V_\phi = W^{\mu}_\phi(\xi)+ c \int_{x}^{1}e^{-b_{\mu}(t)}dt,\textrm{ }c \in\mathbb{R}.
	\end{equation*}
\end{cor}
\begin{proof}
  By Lemma \ref{lem4_3_W}, $V-W^{\mu}(\xi)$ is in the kernel of $L^{\mu}_{0}: \mathbf{X}\to \mathbf{Y}$. The conclusion then follows from Lemma \ref{lem4_3_6}.
\end{proof}

Let $l_2$ be the functionals on $\mathbf{X}$ defined by (\ref{def_l}), and $\mathbf{X}_1$ be the subspace of $\mathbf{X}$ defined by (\ref{eq4_3x}). As shown in Section 4.1, $l_2(V^2_{\mu})>0$ for every $\mu \in K$. So $\mathbf{X}_1$ is a closed subspace of $\mathbf{X}$, and
\begin{equation*} 
	\mathbf{X} = \mbox{span} \{ V_{\mu}^{2} \} \oplus\mathbf{X}_1,\quad \forall \mu \in K, 
\end{equation*}
with the projection operator $P(\mu): \mathbf{X}\rightarrow\mathbf{X}_1$ given by 
\begin{equation*}
    P(\mu)V = V-  c(\mu)l_2(V)V_{\mu}^{2} \textrm{ for } V\in \mathbf{X}.
    \end{equation*}
where $c(\mu)=\left(\int_{0}^{1}e^{-b_{\mu}(t)}dt\right)^{-1}>0$ for all $\mu\in K$.

By Lemma \ref{lem4_3_6} and Corollary \ref{cor4_3_7}, using similar proof as Lemma \ref{lem4_2}, we have
\begin{lem}\label{lem4_3iso}
	The operator $ L^{\mu}_{0}: \mathbf{X}_1\rightarrow\mathbf{Y}$ is an isomorphism.
\end{lem}

\begin{lem}\label{lem4_3V2}
  $V^2_{\mu}\in C^{\infty}(K, \mathbf{X})$.
\end{lem}
\begin{proof}
For convenience, in this proof we denote $a(x)=a_{\mu}(x)$, $b(x)=b_{\mu}(x)$ and $V^2=V^2_{\mu}$.

   By computation, using the explicit expression of $U^{\mu,-1-\sqrt{1+2\mu}}_{\theta}(x), a(x), a'(x), b(x)$ and $V^2_{\phi}(x)$ given by (\ref{eq4_3_0}), (\ref{eq4_3_ab}) and (\ref{eq4_2_ker}), and the estimates of $\partial_{\mu}^{i}U^{\mu,-1-\sqrt{1+2\mu}}_{\theta}$ in (\ref{eqcor3_1_4}), we have, for $\mu\in (-\frac{1}{2},-\frac{3}{8})$, that
 \[
    e^{-b(x)}=(1+x)^{-1-\sqrt{1+2\mu}},  \quad  -1<x<1.
 \]
 So
\[
    V^2_{\phi}(x)=O(1)(1-x)(1+x)^{-\sqrt{1+2\mu}},  \quad \left| \frac{d}{dx}V^2_{\phi}(x) \right| = e^{-b(x)}=(1+x)^{-1-\sqrt{1+2\mu}},
    \]
    \[
     \left| \frac{d^2}{dx^2}V^2_{\phi}(x) \right| = \left| b'(x) \right| e^{-b(x)}=O(1)(1+x)^{-2-\sqrt{1+2\mu}},\quad  -1<x<1.
\]
Moreover, 
\[
  \frac{\partial^i}{\partial \mu^i}b(x)=\frac{\partial^i}{\partial \mu^i}\sqrt{1+2\mu}\ln(1+x).
\]
So we have, for $-1<x<1, i=1,2,3...$, that
\[
   \left|\partial^i_{\mu}V^2_{\phi}(x)\right|=O(1)(1-x)(1+x)^{-\sqrt{1+2\mu}}(\ln(1+x))^i, 
\]
\[
   \left|\partial^i_{\mu}\frac{d}{dx}V^2_{\phi}(x)\right|=O(1)(1+x)^{-1-\sqrt{1+2\mu}}(\ln(1+x))^i, 
\]
\[
   \left|\partial^i_{\mu}\frac{d^2}{dx^2}V^2_{\phi}(x)\right|=O(1)(1+x)^{-2-\sqrt{1+2\mu}}(\ln(1+x))^i.
\]
 The above imply that  for all $i\ge 0$, $\partial^i_{\mu}V^2(x)\in \mathbf{X}$, and $V^2_{\phi}\in C^{\infty}(K,M_2)$. So $V^2\in C^{\infty}(K,\mathbf{X})$.
\end{proof}

Next, by similar arguments  in the proof of Lemma \ref{lem4_1_10}, using Lemma \ref{lem4_3V2}, we have
\begin{lem}\label{lem4_3_10}
  There exists $C=C(K)>0$ such that for all $\mu\in K$,  $\beta \in \mathbb{R}^2$, and $V\in \mathbf{X}_1$, 
  \[
     ||V||_{\mathbf{X}}+|\beta |\le C||\beta V^2_{\mu}+V||_{\mathbf{X}}.
  \] 
\end{lem}

\noindent \emph{Proof of Theorem \ref{thm4_3}:}
Define a map $F: K\times\mathbb{R}\times \mathbf{X}_1\to \mathbf{Y}$ by
\[
    F(\mu,\beta,V)=G(\mu,\beta V^2_{\mu}+V).
\]
By Proposition \ref{prop4_3}, $G$ is a $C^{\infty}$ map from $K\times \mathbf{X}$ to $\mathbf{Y}$. Let $\tilde{U}=\tilde{U}(\mu,\beta,V)=\beta_2V^2_{\mu}+V$. Using Lemma \ref{lem4_3V2}, we have $\tilde{U}\in C^{\infty}(K\times\mathbb{R}\times \mathbf{X}_1, \mathbf{X})$. So it concludes that $F\in C^{\infty}(K\times\mathbb{R}\times \mathbf{X}_1,\mathbf{Y})$. 

Next, by definition $F(\mu,0,0)=0$ for all $\mu\in K$. Fix some $\bar{\mu}\in K$, using Lemma \ref{lem4_3iso}, we have $F_{V}( \bar{\mu}, 0,0)=L_0^{ \bar{\mu}}:\mathbf{X}_1\to \mathbf{Y}$ is an isomorphism.

Applying Theorem C, there exist some $\delta>0$ and a unique $V\in C^{\infty}(B_{\delta}(\bar{\mu})\times B_{\delta}(0), \mathbf{X}_1 )$, such that
\[
    F(\mu,\beta,V(\mu,\beta))=0, \quad \forall \mu\in B_{\delta}(\bar{\mu}), \beta\in B_{\delta}(0),
\]
and 
\[
   V(\bar{\mu},0)=0.
\]
The uniqueness part of Theorem C holds in the sense that there exists some $0<\bar{\delta}<\delta$,  such that $B_{\bar{\delta}}(\bar{\mu},0,0)\cap F^{-1}(0) \subset  \{(\mu,\beta,V(\mu,\beta))|(\gamma)\in B_{\delta}(\bar{\mu}), \beta\in B_{\delta}(0)\}$.

\textbf{Claim}: there exists some $0<\delta_1<\frac{\bar{\delta}}{2}$, such that $V(\mu,0)=0$ for every $\mu\in B_{\delta_1}(\bar{\mu})$.

\emph{Proof of the claim:}
Since $V(\bar{\mu},0)=0$ and $V(\mu,0)$ is continuous in $\mu$, there exists some $0<\delta_1<\frac{\bar{\delta}}{2}$, such that for all $\mu\in B_{\delta_1}(\bar{\mu})$, $(\mu, 0,V(\mu,0))\in B_{\bar{\delta}(\bar{\mu},0,0)}$. We know that for all $\mu\in B_{\delta_1}(\bar{\mu})$,
   \[
      F(\mu, 0,0)=0,  
      \]
       and 
       \[
          F(\mu, 0, V(\mu,0))=0.
       \]
By the above mentioned uniqueness result, $V(\mu,0)=0$, for every $\mu\in B_{\delta_1}(\bar{\mu})$.

Now we have $V\in C^{\infty}(B_{\delta_1}(\bar{\mu})\times B_{\delta_1}(0), \mathbf{X}_1 )$, and 
\[
     F(\mu, \beta,V(\mu,\beta))=0, \quad \forall \mu\in B_{\delta_1}( \bar{\mu}), \beta\in B_{\delta_1}(0).
\]
i.e.
\[
    G(\mu, \beta V^2_{\mu}+V(\mu,\beta) )=0, \quad \forall \mu\in B_{\delta_1}(\bar{\mu}), \beta\in B_{\delta_1}(0).
\]
Take derivative of the above with respect to $\beta$ at $(\mu, 0)$, we have
\[
   G_{\tilde{U}}(\mu,0)(V^2_{\mu}+\partial_{\beta}V(\mu,0))=0.
\]
Since $G_{\tilde{U}}(\mu,0)V^2_{\mu}=0$ by Lemma \ref{lem4_3_6}, we have 
\[
   G_{\tilde{U}}(\mu,0)\partial_{\beta}V(\mu,0)=0.
\]
But $\partial_{\beta_i}V(\mu,0)\in C^{\infty}(\mathbf{X}_1)$, so 
\[
    \partial_{\beta}V(\mu,0)=0.
\]
Since $K$ is compact, we can take $\delta_1$ to be a universal constant for each  $\mu\in K$. So we have proved the existence of $V$ in Theorem \ref{thm4_3}.

Next, let $\mu\in B_{\delta_1}(\bar{\mu})$. Let $\delta'$ be a small constant to be determined.  For any $U$ satisfies the equation (\ref{eq4_0}) with $U-U^{\mu,-1-\sqrt{1+2\mu}}\in \mathbf{X}$, and $||U-U^{\mu,-1-\sqrt{1+2\mu}}||_{ \mathbf{X}}\le \delta'$ there exist some $\beta\in \mathbb{R}$ and $V^*\in \mathbf{X}_1$ such that
 \[
     U-U^{\mu,-1-\sqrt{1+2\mu}}=\beta V^2_{\mu}+V^*.
 \]
Then by Lemma \ref{lem4_3_10}, there exists some constant $C>0$ such that
\[
    \frac{1}{C}(|\beta|+||V^*||_{\mathbf{X}})\le ||\beta V^2_{\mu}+V^*||_{\mathbf{X}}\le \delta'.
\]
This gives $||V^*||_{\mathbf{X}}\le C\delta'$.

 Choose $\delta'$ small enough such that $C\delta'<\delta_1$. We have the uniqueness of $V^*$. So  $V^*=V(\mu,\beta)$ in (\ref{eq_thm4_3_1}). The theorem is proved.  \qed

Now with Theorem \ref{thm4_1}-\ref{thm4_3} we can give the \\

\noindent \emph{Proof of the existence part of Theorem \ref{thm1_2}:}
Recall the relation between the parameters $(\mu,\gamma)$ and $(\tau,\sigma)$
\[
   \mu=\frac{1}{8}\tau^2-\frac{1}{2}\tau,\quad \gamma=-2\sigma.
\]
Let $K$ be a compact subset of one of the four sets $J_1$, $J_2$, $J_3\cap\{2< \tau<3\}$ and $J_3\cap \{\tau=2\}$, where $J_1,J_2,J_3$ are the sets defined by (\ref{eqJ}). 

For $(\tau,\sigma)\in K\cap J_1$, let
\[
    u(\tau,\sigma,\beta)=\frac{1}{\sin\theta}\left(U^{\mu,\gamma}+\beta V_{\mu,\gamma}^2+V(\mu,\gamma,0, \beta)\right) ,\quad \beta \in (-\delta,\delta),
\]
where $\delta,  V_{\mu,\gamma}^2$ and $V(\mu,\gamma,0, \beta)$ are as in Theorem \ref{thm4_1}.

For $(\tau,\sigma)\in K\cap  J_2$, let
\[
    u(\tau,\sigma,\beta)=\frac{1}{\sin\theta}\left(U^{-\frac{1}{2},\gamma}+\beta V_{-\frac{1}{2},\gamma}^2+V(\gamma,0, \beta)\right) ,\quad \beta \in (-\delta,\delta),
\]
where $\delta, V_{-\frac{1}{2},\gamma}^2$ and $V(\gamma,0, \beta)$ are as in  Theorem \ref{thm4_2}.

For $(\tau,\sigma)\in K\cap (J_3\cap\{2\le \tau<3\})$, let
\[
    u(\tau,\sigma,\beta)=\frac{1}{\sin\theta}\left(U^{\mu,-1-\sqrt{1+2\mu}}+\beta V_{\mu,-1-\sqrt{1+2\mu}}^2+V(\mu,\beta)\right), \quad \beta \in (-\delta,\delta),
\]
where $\delta, V_{\mu,-1-\sqrt{1+2\mu}}^2$ and $V(\mu,\beta)$ are as in Theorem \ref{thm4_3}.

With  $u(\tau,\sigma,\beta)$ defined as the above, the existence part of Theorem \ref{thm1_2} follows from Theorem \ref{thm4_1}-\ref{thm4_3}.

\section{Asymptotic behavior of solutions}
In this section we study the asymptotic behavior of (-1)-homogeneous axisymmetric solutions of (\ref{eq_homo}) in a punctured ball around the north or south pole of $\mathbb{S}^2$. In particular we prove Theorem \ref{thm1_5} and Theorem \ref{thm1_6}.

Recall that the Navier-Stokes equations for (-1)-homogeneous solutions have been converted to the system

\begin{equation}\label{eq5_2}
\left\{
\begin{split}
	& (1-x^2) U_\theta' + 2x U_\theta +\frac{1}{2} U_\theta^2 + \int_{x_0}^{x} \int_{x_0}^{l} \int_{x_0}^{t} \frac{2 U_\phi(s) U_\phi'(s)}{1-s^2} ds dt dl = c_1 x^2 + c_2 x + c_3, \\
	& (1-x^2) U_\phi'' + U_\theta U_\phi' = 0. 
\end{split}
\right. 
\end{equation}
where $x_0$ is some fixed number in $(-1,1)$, $c_1,c_2,c_3$ are constants.

It follows from the second line of the above that
\begin{equation}\label{eq5_10}
   U'_{\phi}(x)=Ce^{-\int_{-1+\delta_1}^{x}\frac{U_{\theta}}{1-s^2}ds}.
\end{equation}

Let $\delta>0$ be a real number, $H$ be a function of $x$, we consider the equation 
\begin{equation}\label{5_1}
  (1-x^2)U_{\theta}'(x)+2xU_{\theta}+\frac{1}{2}U^2_{\theta}=H(x), \quad -1<x\le -1+\delta.
\end{equation}

Define, with $x_0=-1+\delta$, 
\begin{equation}\label{eq5_int}
   I(x):=\int_{x_0}^{x} \int_{x_0}^{l} \int_{x_0}^{t} \frac{2 U_\phi(s) U_\phi'(s)}{1-s^2} ds dt dl.
\end{equation}

We can write $I$ as 
\begin{equation}\label{eqI_1}
  \begin{split}
     I(x) & =\int_{x_0}^{x} \int_{s}^{x} \int_{t}^{x}\frac{2 U_\phi(s) U_\phi'(s)}{1-s^2} dl dt ds
	= \int_{x_0}^{x} \frac{U_\phi(s) U_\phi'(s)(s-x)^2}{1-s^2} ds\\
           & = -\frac{(x-x_0)^2}{2(1-x_0^2)}U^2_{\phi}(x_0)+I_1
     \end{split}
\end{equation}
where 
\begin{equation}\label{eqI_2}
   I_1(x)=- \int_{x_0}^{x}  \frac{U_\phi^2(s) (s-x)(1-sx)}{(1-s^2)^2} ds.
\end{equation}
By computation
\begin{equation}\label{eqI_3}
  I_1'(x)= -\int_{x_0}^{x}\frac{U^2_{\phi}(s)(-s^2+2xs-1)}{(1-s^2)^2}ds<0, \quad -1<x<x_0.
  \end{equation}
Indeed, the first inequality in the above follows from $-s^2+2xs-1\le -s^2+s^2+x^2-1=x^2-1<0$, for all $-1<x<s<x_0$.

\noindent \emph{Proof of (i) and (ii) of Theorem \ref{thm1_5}:} We write the first equation  of (\ref{eq5_2}) as (\ref{5_1}) with
\[
  H(x)=-I(x)+c_1x^2+c_2x+c_3,
\]
where $I(x)$ is defined in (\ref{eq5_int}).

By (\ref{eqI_1}) and (\ref{eqI_3}), $H(x)$ is the sum of a bounded function and a monotonically increasing function in $(-1,-1+\delta]$. It follows that $H^+\in L^{\infty}(-1,-1+\delta)$. 

Let $g(x):=U_{\theta}(x)$, $a(x):=1-x^2$ and $b(x):=2x$. An application of Proposition \ref{lemA_1} yields part (i) and (ii) of the theorem. \qed \\

For $H\in C[-1,-1+\delta]$, denote $\tau_1=2-\sqrt{4+2H(-1)}$, and $\tau_2=2+\sqrt{4+2H(-1)}$

\begin{lem}\label{lem5_4}
 For $\delta>0$, $H\in C[-1,-1+\delta]$, let  $U_{\theta}\in C^1(-1, -1+\delta]$ be a solution of (\ref{5_1}) in $(-1,-1+\delta)$. Then 
\[
      U_{\theta}(-1): =\lim_{x\to -1^+}U_{\theta}(x)=\tau_1 \textrm{  or  } \tau_2, 
\]
and $H(-1)=-2U_{\theta}(-1)+\frac{1}{2}U^2_{\theta}(-1)\ge -2$.
\end{lem}
\begin{proof}
  Let $g(x):=U_{\theta}(x)$, $a(x):=1-x^2$ and $b(x):=2x$. By Proposition \ref{lemA_1},  $U_{\theta}(-1):=\lim_{x\to -1^+}U_{\theta}(x)$ exists and is finite, and $\lim_{x\to -1^+}(1-x^2) U_{\theta}'(x)=0$. Sending $x\to -1$ in (\ref{5_1}) leads to 
\[
   H(-1)=-2U_{\theta}(-1)+\frac{1}{2}U^2_{\theta}(-1)=\frac{1}{2}\left(U_{\theta}(-1)-2\right)^2-2. 
\]
Lemma \ref{lem5_4} follows from the above.
\end{proof}

Now we are ready to give some further local asymptotic behavior of local solutions $U$ of (\ref{eq5_2}) as $x\to -1^+$. By part (i) of Theorem \ref{thm1_5} we know that $\lim_{x\to -1^+}U_{\theta}(x)=U_{\theta}(-1)$ exists and is finite. 

Now let us prove part (iii) of Theorem \ref{thm1_5}.

\begin{lem} \label{lem5_7}
 For $\delta>0$,  $x_0\in (-1,-1+\delta]$, let $U=(U_{\theta},U_{\phi})$ be a solution of system (\ref{eq5_2}) in $(-1,-1+\delta)$, and $U_{\theta}\in C^1(-1,-1+\delta]$, $U_{\phi}\in C^2(-1,-1+\delta]$, with $U_{\theta}(-1)<2$. \\
Then if $U_{\theta}(-1)\ne 0$,  there exist some constants $a_1,a_2$ and $b_1,b_2,b_3$, such that for any $\epsilon>0$,
\begin{equation*} 
\begin{split}
  U_{\theta}(x) = & U_{\theta}(-1)+ a_1(1+x)^{\alpha_0}+a_2(1+x)+O((1+x)^{2\alpha_0-\epsilon})+O((1+x)^{2-\epsilon}),\\
   U_{\phi}(x) = & U_{\phi}(-1)+b_1(1+x)^{\alpha_0}+b_2(1+x)^{2\alpha_0}+b_3(1+x)^{1+\alpha_0} \\
   & +O((1+x)^{\alpha_0+2-\epsilon})+O((1+x)^{3\alpha_0-\epsilon}) 
   \end{split}
\end{equation*}
where $\alpha_0=1-\frac{U_{\theta}(-1)}{2}$.

If $U_{\theta}(-1)=0$, there exist some constants $a_1,a_2$ and $b_1,b_2,b_3$ such that for any $\epsilon>0$, 
\begin{equation*}
\begin{split}
  & U_{\theta}(x)=a_1(1+x)\ln (1+x)+a_2(1+x)+O((1+x)^{2-\epsilon}),\\
  & U_{\phi}(x)=U_{\phi}(-1)+b_1(1+x)+b_2(1+x)^{2}\ln(1+x)+b_3(1+x)^{2}+O((1+x)^{3-\epsilon}).
  \end{split}
\end{equation*}
\end{lem}
\begin{proof}
Let $I(x)$ be defined by (\ref{eq5_int}). The first equation of (\ref{eq5_2}) can be written as
\begin{equation*}
    (1-x^2)U_{\theta}'+2xU_{\theta}+\frac{1}{2}U_{\theta}^2=\lambda+h(x),
\end{equation*}
where by Lemma \ref{lem5_4}, $\lambda$ is a constant , $\lambda=-2U_{\theta}(-1)+\frac{1}{2}U^2_{\theta}(-1)=-\frac{\tau_1\tau_2}{2}$, and $h(x)=-I(x)+I(-1)+c_1(1+x)+c_2(1+x)^2$ for some constants $c_1$ and $c_2$.

Since $U_{\theta}(-1)<2$, there exist  $\delta_1,\epsilon>0$ such that $\displaystyle \frac{U_{\theta}(x)}{1-x}\le \frac{U_{\theta}(-1)+\epsilon}{2}<1$ for $-1<x\le -1+\delta_1$.

For convenience denote $\tau_1=U_{\theta}(-1)$ and let $\tau_2=4-U_{\theta}(-1)$.  It follows from (\ref{eq5_10}) that for some constant $C_1$,  $ |U'_{\phi}|\le C_1(1+x)^{-\frac{\tau_1+\epsilon}{2}}$ and $|U_{\phi}(x)|\le C_1$ for $-1<x<-1+\delta_1$. Then $I'''(x)=O((1+x)^{-1-\frac{\tau_1+\epsilon}{2}})$. Therefore both $I(-1)$ and $I'(-1)$ exist and are finite, and $I(x)=I(-1)+I'(-1)(1+x)+O((1+x)^2)+O((1+x)^{2-\frac{\tau_1}{2}-\epsilon})$. So $h(x)=(c_1-I'(-1))(1+x)+O((1+x)^2)+O((1+x)^{2-\frac{\tau_1}{2}-\epsilon})$.

Rewriting the above equation as 
\begin{equation*} 
   (1-x^2)(U_{\theta}-\tau_1)'+\frac{1}{2}(U_{\theta}-\tau_1)(U_{\theta}-\tau_2)=\tilde{h}(x):=h(x)-2(1+x)U_{\theta}. 
\end{equation*}
Let $V:=U_{\theta}-\tau_1$, $B:=\frac{U_{\theta}-\tau_2}{2(1-x^2)}$, $H:=\frac{\tilde{h}}{1-x^2}$. It can be checked that $B,H\in C(-1,-1+\delta]$, $H\in L^{\infty}(-1,-1+\delta)$  and $\lim_{x\to -1^+}(1+x)B(x)=\frac{\tau_1-\tau_2}{4}=-\alpha_0<0$, and  $V,B,H$ satisfy
\begin{equation*}
  V'(x)+B(x)V(x)=H(x), \quad -1<x<-1+\delta.
\end{equation*}
So we can apply Lemma \ref{lemA_3} with $\beta=\alpha_0$ and $b=1$ to obtain $U_{\theta}-\tau_1=O((1+x)^{\min\{\alpha_0,1\}-\epsilon})$ for any $\epsilon>0$. 

Next, use this estimate in (\ref{eq5_10}), we have $U'_{\phi}=O(1)(1+x)^{-\frac{\tau_1}{2}}$. So $U_{\phi}=U_{\phi}(-1)+O(1)(1+x)^{1-\frac{\tau_1}{2}}$ and $I(x)=I(-1)+I'(-1)(1+x)+O((1+x)^{2-\frac{\tau_1}{2}-\epsilon})$ for any $\epsilon>0$. Then by the estimate of $I(x)$ and $U_{\theta}$, notice $\alpha_0=1-\frac{\tau_1}{2}$, there is some constant $d_1$ such that $\tilde{h}(x)=d_1(1+x)+O((1+x)^{1+\min\{\alpha_0,1\}-\epsilon})$ for any $\epsilon>0$. So $H=d_1+O((1+x)^{\min\{\alpha_0,1\}-\epsilon})$.  Moreover,
\[
   (1+x)B+\alpha_0=O((1+x)^{\min\{\alpha_0,1\}-\epsilon}).
\]
So we can apply Lemma \ref{lemA_5} .  If $\alpha_0\ne 1$, there exist some constants $a_1, a_2$ such that
\[
   U_{\theta}-\tau_1=a_1(1+x)^{\alpha_0}+a_2(1+x)+O((1+x)^{1+\min\{\alpha_0,1\}-\epsilon})+O((1+x)^{\alpha_0+\min\{\alpha_0,1\}-\epsilon}).
\]
Then by (\ref{eq5_10}), we have estimate of $U'_{\phi}$ and $U_{\phi}(-1)$ exists and finite, and there exist some constants $b_1,b_2,b_3$ such that 
\[
\begin{split}
    U_{\phi} = & U_{\phi}(-1)+b_1(1+x)^{\alpha_0}+b_2(1+x)^{2\alpha_0}+b_3(1+x)^{1+\alpha_0} \\
  & + O((1+x)^{\alpha_0+1+\min\{\alpha_0,1\}-\epsilon})+O((1+x)^{2\alpha_0+\min\{\alpha_0,1\}-\epsilon})
\end{split}
\]
for any $\epsilon>0$. 

If $\alpha_0=1$, $U_{\theta}(-1)=0$, there exist some constants $a_1,a_2$ such that
\[
   U_{\theta}=a_1(1+x)\ln (1+x)+a_2(1+x)+O((1+x)^{1+\min\{\alpha_0,1\}-\epsilon})
\]
By (\ref{eq5_10}), $U_{\phi}(-1)$ exists and there exist some constants $b_1,b_2,b_3$ such that 
\[
    U_{\phi}=U_{\phi}(-1)+b_1(1+x)+b_2(1+x)^{2}\ln(1+x)+b_3(1+x)^{2}+O((1+x)^{2+\min\{\alpha_0,1\}-\epsilon})
 \]
for any $\epsilon>0$. 
\end{proof}

\begin{lem}\label{lem5_8}
   Let $U=(U_{\theta},U_{\phi})$ be a solution of system (\ref{eq5_2}), and $U_{\theta}\in C^1(-1,-1+\delta]$, $U_{\phi}\in C^2(-1,-1+\delta]$, for some $\delta>0$ and $x_0\in (-1,-1+\delta]$, with $U_{\theta}(-1)=2$. Then for some constants $b_1$ and $b_2$, and for any $\epsilon\in (0,1)$, either
\begin{equation}\label{eq5_9}
\begin{split}
   & U_{\theta}=2+\frac{4}{\ln(1+x)}+O((\ln (1+x))^{-2+\epsilon}),\\
   & U_{\phi}=U_{\phi}(-1)+\frac{b_1}{\ln(1+x)}+O((\ln (1+x))^{-2+\epsilon}), 
   \end{split}
\end{equation}
or 
\begin{equation}\label{eq5_9or}
\begin{split}
   & U_{\theta}=2+O((1+x)^{1-\epsilon}),\\
   & U_{\phi}=b_1\ln(1+x)+b_2+b_1O((1+x)^{1-\epsilon}).
   \end{split}
\end{equation}
\end{lem}
\begin{proof}

Let $I$ be the triple integral defined by (\ref{eq5_int}). The equation (\ref{5_1}) can be written as
\begin{equation}\label{eqtau2}
(1-x^2)(U_{\theta}-2)'+\frac{1}{2}(U_{\theta}-2)^2=\tilde{h}:=-I(x)+c_1x^2+c_2x+c_3+2-2(1+x)U_{\theta}.
\end{equation}
Since $U_{\theta}(-1)=2$, for any $\epsilon>0$,  
\[
   |U'_{\phi}|\le C(1+x)^{-\frac{2+\epsilon}{2}}, 
\]
and  $|U_{\phi}|\le C(1+x)^{-\frac{\epsilon}{2}}$ for some constant $C>0$. Thus $I(x)=I(-1)+O((1+x)^{1-\epsilon})$. So $\tilde{h}=O((1+x)^{1-\epsilon})$.

By (\ref{eqtau2}), $g:=(U_{\theta}-2)\ln(1+x)$ satisfies
\begin{equation*}
    (1-x^2)\ln(1+x)g'-(1-x)g+\frac{1}{2}g^2=\tilde{h}(x)(\ln(1+x))^2.
\end{equation*}
By Proposition \ref{lemA_1} ,  $g\in L^{\infty}(-1,-1+\frac{\delta}{2})$, $\displaystyle \lim_{x\to -1^+}g(x)$ exists and is finite, $\displaystyle \lim_{x\to -1^+}(1-x^2)\ln(1+x)g'=0$, and $-2g(1)+\frac{1}{2}g^2(1)=0$. So $g(1)=0$ or $4$

Let us write
\[
    U_{\theta}(x)=2+\frac{\eta}{\ln(1+x)}+V. 
\]
We can see that $\eta=0$ or $4$, $V(-1)=0$ and $V=o(\frac{1}{\ln(1+x)})$.

By (\ref{eqtau2}), $V$ satisfies
\begin{equation*}
   (1-x^2)V'+\frac{\eta}{\ln(1+x)}V+\frac{1}{2}V^2 =\hat{h},
\end{equation*}
where $\hat{h}:=-I(x)+c_1x^2+c_2x+c_3-\frac{\frac{1}{2}\eta^2-\eta(1-x)}{(\ln(1+x))^2}-2(1+x)V-4x-2-\frac{2\eta (1+x)}{\ln (1+x)}$. Since  $\eta=0$ or $4$,  $\hat{h}=O((1+x)^{1-\epsilon})$.

Let $B=\frac{\frac{1}{2}V+\frac{\eta}{\ln (1+x)}}{1-x^2}$, $H(x)=\frac{\hat{h}}{1-x^2}$. Then $B,H\in C(-1,-1+\delta]$ satisfy $H(x)=O((1+x)^{-\epsilon})$,  $\lim_{x\to -1^+}(1+x)\ln(1+x)B=\frac{\eta}{2}$,  $V=o(\frac{1}{\ln(1+x)})$. So we can apply Lemma \ref{lemA_7} to conclude that 
$V=O((\ln(1+x))^{-2+\epsilon})$ if $\eta=4$ and $V=O((1+x)^{1-\epsilon})$ if $\eta=0$.  We have established the estimates of $U_{\theta}$ in (\ref{eq5_9}) and (\ref{eq5_9or}). 

With estimates of  $U_{\theta}$ in (\ref{eq5_9}) and (\ref{eq5_9or}), we obtain from (\ref{eq5_10}) the estimates of $U_{\phi}$ in (\ref{eq5_9}) and (\ref{eq5_9or}). The lemma is proved.
\end{proof}

\begin{rmk} 
 This case does occur. For example, as given by Corollary \ref{cor3_1}, for all $\gamma>-1$, $(U_{\theta}, U_{\phi})=((1-x)(1+\frac{2(\gamma+1)}{(\gamma+1) \ln \frac{1+x}{2}-2}),0)$  are smooth solutions on $\mathbb{S}^2\setminus\{S\}$.  
\end{rmk}

\begin{lem}\label{lem5_8_1}
   Let $U=(U_{\theta},U_{\phi})$ be a solution of the system (\ref{eq5_2}), and $U_{\theta}\in C^1(-1,-1+\delta]$, $U_{\phi}\in C^2(-1,-1+\delta]$, for some $\delta>0$ and $x_0\in (-1,-1+\delta]$. If $2<U_{\theta}(-1)<3$,  there exist constants $ a_1,a_2$ and $b_1,b_2,b_3,b_4$ such that for any $\epsilon>0$, 
\begin{equation}\label{eqlem5_8_1}
  \begin{split}
   U_{\theta}(x) = & U_{\theta}(-1)+a_1(1+x)^{3-U_{\theta}(-1)}+a_2(1+x)+O((1+x)^{2(3-U_{\theta}(-1))-\epsilon}),\\
   U_{\phi}(x) = & b_1(1+x)^{1-\frac{U_{\theta}(-1)}{2}}+b_2+b_1b_3(1+x)^{4-\frac{3U_{\theta}(-1)}{2}}+b_1b_4(1+x)^{2-\frac{U_{\theta}(-1)}{2}}\\
 & + b_1O((1+x)^{7-\frac{5U_{\theta}(-1)}{2}-\epsilon}).
   \end{split}
\end{equation}
\end{lem}
\begin{proof}
  Let $\tau_2=U_{\theta}(-1)$, and $I(x)$ be the triple integral defined by (\ref{eq5_int}). Using the fact $2<U_{\theta}(-1)<3$ and (\ref{eq5_10}), for any $\epsilon>0$, there exists some constant $C_1$ such that $|U'_{\phi}(x)|\le C_1(1+x)^{-\frac{\tau_2+\epsilon}{2}}$. Then by (\ref{eq5_int}) we obtain that in the current situation  $I(x)=I(-1)+O((1+x)^{3-\tau_2-\epsilon})$.  So $U_{\theta}$ satisfies
  \[
     (1-x^2)(U_{\theta}-\tau_2)'+\frac{1}{2}(U_{\theta}-\tau_1)(U_{\theta}-\tau_2)=\tilde{h}:=-I(x)+I(-1)+c_1(1+x)+c_2(1+x)^2-2(1+x)U_{\theta}
  \]
  where $c_1,c_2$ are constants. By the estimate of $I(x)$,  $\tilde{h}=O((1+x)^{3-\tau_2-\epsilon})$.
  Let $V=U_{\theta}-\tau_2$, $B=\frac{U_{\theta}-\tau_1}{2(1-x^2)}$, $H=\frac{\tilde{h}}{1-x^2}$. Then $V\in C^1(-1,-1+\delta]$, $B,H\in C(-1,-1+\delta]$, satisfy $V'+BV=H$,  and $H(x)=O((1+x)^{2-\tau_2-\epsilon})$, $\lim_{x\to -1^+}(1+x)B=\alpha_0>0$, and $\lim_{x\to -1^+}V(x)e^{\int_{-1+\delta}^{x}B(s)ds}=0$. So we can apply Lemma \ref{lemA_4} to obtain
\[
   U_{\theta}(x)-\tau_2=O((1+x)^{3-\tau_2-\epsilon}). 
\]
With this estimate, we derive from (\ref{eq5_10}) that  $U'_{\phi}=C(1+x)^{-\frac{\tau_2}{2}}(1+O((1+x)^{3-\tau_2-\epsilon}))$. So $U_{\phi}=\frac{2}{2-\tau_2}C(1+x)^{1-\frac{\tau_2}{2}}(1+O((1+x)^{3-\tau_2-\epsilon}))$ and $I(x)=I(-1)+c'_1(1+x)^{3-\tau_2}+c'_2(1+x)+O((1+x)^{2(3-\tau_2)-\epsilon})$ for some constants $c'_1,c'_2$. Let $\bar{b}=3-\tau_2$. Then by the estimate of $I(x)$ and $U_{\theta}$, there is some constant $d_1$ such that $\tilde{h}(x)=c'_1(1+x)^{\bar{b}}+d_1(1+x)+O((1+x)^{2\bar{b}-\epsilon})$. So $H=c'_1(1+x)^{\bar{b}-1}+d_1+O((1+x)^{2\bar{b}-1-\epsilon})$.  Moreover,
\[
   (1+x)B-\alpha_0=O((1+x)^{\bar{b}-\epsilon}).
\]
So we can apply Lemma \ref{lemA_6} to obtain the first estimate of $U_{\theta}$ in (\ref{eqlem5_8_1}). Then by (\ref{eq5_10}), we have the estimate of $U_{\phi}$ in (\ref{eqlem5_8_1}), using the first estimate in (\ref{eqlem5_8_1}). 
\end{proof}

Part (iii) of Theorem \ref{thm1_5} and part (i), (ii) and (iv) of Theorem \ref{thm1_6} follow from Lemma \ref{lem5_7}-\ref{lem5_8_1}. So Theorem \ref{thm1_5} is proved. Next let us prove part (iii) of Theorem \ref{thm1_6}.

\begin{lem}\label{lem5_9}
  If $U = (U_{\theta}, U_\phi)$ is a solution of (\ref{eq5_2}) and $U_{\theta}\in C^{1}(-1,-1+\delta)$, $0<\delta<2$, $U_{\theta}(-1)\ge 3$, then $U_{\phi}$ is a constant in $(-1,-1+\delta)$.
\end{lem}
\begin{proof}
We prove it by contradiction. Assume that $U_{\phi}$ is not a constant, then (\ref{eq5_10}) holds for a nonzero constant $C$ and we may assume  that  $C$  is positive. Let $I(x)$ be given by (\ref{eq5_int}) with  $x_0=-1+\delta$. Since $U_{\theta}$ and $(1-x^2)U'_{\theta}$ are bounded according to Theorem \ref{thm1_5}, $I(x)$ is bounded in view of (\ref{eq5_2}). We divide the proof into two cases.

   \textbf{Case} 1. $U_{\theta}(-1)>3$. 
   
    If $U_{\theta}(-1)>3$,  there exist $a>3$ such that  $U_{\theta}(x)>a>3$ for $x$ close to $-1$.
   So by (\ref{eq5_10}), there exists $c>0$ such that $U'_{\phi}\ge c (1+x)^{-\frac{a}{2}}$   and  $-U_{\phi}\ge c(1+x)^{-\frac{a}{2}+1}$ for $x$ close to $-1$ .  Then, using (\ref{eq5_int}), we have $-I(x)\to +\infty \textrm{ as }x \to -1^+$, a contradiction.

  \textbf{Case} 2. $U_{\theta}(-1)=3$.
  
   Since $U_{\theta}(-1)=3$, we rewrite the first line of (\ref{eq5_2}) as 
   \begin{equation*}
   (1-x^2) (U_\theta-3)'+\frac{1}{2}(U_{\theta}-1)(U_{\theta}-3)=\tilde{h}(x):=-2(1+x)U_{\theta}+Q(x)+I_1(-1)-I_1(x), 
\end{equation*}
where  $I_1$ is given by (\ref{eqI_2}) , and $Q(x)$ is a quadratic polynomial with $Q(-1)=0$. 

By (\ref{eqI_3}),   $I_1(-1)-I_1(x)\ge 0$ in $(-1,-1+\delta)$.  Thus, using the boundedness of $U_{\theta}$ and the fact that $Q(-1)=0$, $\tilde{h}(x)\ge -C(1+x)$ in $(-1,-1+\delta)$ for some constant  $C>0$.

Let $V(x)=U_{\theta}(x)-3$, $B(x)=\frac{U_{\theta}-1}{2(1-x^2)}$ and $H(x)=\frac{\tilde{h}(x)}{1-x^2}$.  Then (\ref{eqA_3_4}), (\ref{eqA_3_5}), (\ref{eqA_4_2}) and (\ref{eqA_4_1}) hold with $b=1$, $\beta=-\frac{1}{2}$. By Lemma \ref{lemA_4}, see also Remark \ref{rmkA_4}, we have, for some positive constant $C$, and for any $\epsilon>0$,  $U_{\theta}-3 \ge -C(1+x)^{1-\epsilon}$ in $(-1,-1+\delta)$.

Next, in (\ref{eq5_10}),  apply the estimate of $U_{\theta}(x)$, in $(-1,-1+\delta)$ there is
\begin{equation*}
   U'_{\phi}(x)  \ge ce^{-\frac{3}{2}\ln (1+x)}\ge c(1+x)^{-\frac{3}{2}}, \textrm{ for }x \textrm{ close to }-1.
\end{equation*}
Then $ -U_{\phi}(x)\ge c(1+x)^{-\frac{1}{2}}$ for $x$ close to $-1$.  
\begin{equation*}
   -I'''(x)=-\frac{2 U_\phi(x) U_\phi'(x)}{1-x^2}\ge C(1+x)^{-3}. 
\end{equation*}
Thus $I\ge C|\ln (1+x)|$ is unbounded, contradiction. So  $U_{\phi}$ is a constant.
\end{proof}

\noindent \emph{Completion of the proof of Theorem \ref{thm1_2}:}
We have proved the existence part of the theorem in Section 4 for $(\tau,\sigma)\in J_1\cup J_2 \cup (J_3\cap\{2\le \tau<3\})$. Now we prove the nonexistence part of the theorem. 

For $(\tau,\sigma)\in J_3\cap\{\tau>3\}$, let $\{u^i\}$ be a sequence of solutions of (\ref{NS}) satisfying $||\sin\frac{\theta+\pi}{2}(u^i-u_{\tau,\sigma})||_{L^{\infty}(\mathbb{S}^2\setminus\{S\})}\to 0$ as $i\to \infty$. Let $U^i=\sin\theta u^i$ for all $i\in \mathbb{N}$. Recall that $U^{\mu,\gamma}=\sin\theta u_{\tau,\sigma}$ with $(\mu,\gamma)=(\frac{1}{8}\tau^2-\frac{1}{2}\tau,-2\sigma)$. We have $||U_{\theta}^i-U^{\mu,\gamma}_{\theta}||_{L^{\infty}(-1,1]}\to 0$. By Theorem \ref{thm1_5} part (a), $U^i(-1)$ must exists and is finite for every $i$. Since $U^{\mu,\gamma}(-1)>3$, $U^{i}_{\theta}(-1)>3$ for large $i$. Then by Theorem \ref{thm1_6}, $U^i_{\phi}$ must be constant for large $i$. Since $u^i\in C^{\infty}(\mathbb{S}^2\setminus\{S\})$, $U^i_{\phi}(1)=0$, so $U^i_{\phi}=0$ for large $i$. The theorem is proved.

\section{Pingpong ball on top of a fountain}\label{sec_pingpong}

As mentioned in the introduction, the pressure of Landau solutions at the center of north pole is greater than the pressure nearby. In this section, we identify all (-1) homogeneous, axisymmetric, no-swirl solutions which describe outward jets with lower pressure in the center. We tend to believe that the pressure profiles are of interest and modification of these solutions is more likely to support a pingpong ball. 

Set $\alpha:=\gamma+1$, consider below the exact form solutions in Theorem \ref{thm_smooth_north}: 

When $\mu > - \frac{1}{2}$, the solutions are expressed as
\begin{equation*}
	U_\theta(x) =  (1-x)\left(1- b - \frac{2 b(\alpha-b)}{(\alpha+b) (\frac{1+x}{2})^{-b} - \alpha + b} \right)
\end{equation*}
where $b = \sqrt{1+ 2\mu}$. Then $u_r|_{x=1} = U_\theta'(1)=  \gamma=\alpha-1$. By L'Hospital's rule, 
\begin{equation*}
	\lim_{x\rightarrow 1^-}\frac{U_\theta(x)}{1-x^2} = \lim_{x\rightarrow 1^-} \frac{U_\theta'(x)}{-2x} = -\frac{1}{2}U_\theta'(1). 
\end{equation*}
From the second line of (\ref{eq_NSwithpresure}) with $U_\phi\equiv 0$, we have 
\begin{equation*}
\begin{split}
	\lim_{x\rightarrow 1^-} p' & =  \lim_{x\rightarrow 1^-} \left( U_\theta''  -  \frac{1}{1-x^2}U_\theta U_\theta' - \frac{x}{(1-x^2)^2} U_\theta^2 \right)   =  \frac{1}{2}(\alpha+b)(\alpha-b) + \frac{1}{2} U_\theta'^2 - \frac{1}{4} U_\theta'^2 \\
	& =  \frac{1}{2}(\alpha+b)(\alpha-b) + \frac{1}{4} ( \alpha - 1)^2  = \frac{3}{4} \alpha^2 - \frac{1}{2}\alpha + \frac{1}{4} - \frac{1}{2}b^2. 
\end{split}
\end{equation*}
Since $b = \sqrt{1+2\mu}>0$, it can be proved that $u_r|_{x=1}=\alpha - 1 > 0$ and $\displaystyle \lim_{x\to 1^-}p'(x)=\frac{1}{2}(\alpha+b)(\alpha-b) +\frac{1}{4} ( \alpha - 1)^2 < 0$ if and only if $b>1$, $1<\alpha<\frac{1}{3}+\sqrt{\frac{2}{3}b^2-\frac{2}{9}}$. Notice that $b>1$, $1<\alpha<\frac{1}{3}+\sqrt{\frac{2}{3}b^2-\frac{2}{9}}$ implies
\begin{equation*}
	\mu>0, \quad 0 < \gamma <  \frac{2}{3}(\sqrt{1+3\mu} - 1). 
\end{equation*}
Therefore, under the condition $\mu>0$, $0 < \gamma < \frac{2}{3}(\sqrt{1+3\mu} -1)$, we have $u_r\mid_{x=1}>0$, $\frac{dp}{dx}\mid_{x=1}<0$. The corresponding solutions describe fluid jets with lower pressure at north pole than nearby. 

It remains to check the case when $u_r\mid_{x=1}>0$, $\frac{dp}{dx}\mid_{x=1}=0$. This condition implies 
\begin{equation*}
	b>1, \quad \alpha = \frac{1}{3} + \sqrt{\frac{2}{3}b^2 - \frac{2}{9}}, 
\end{equation*} 
or equivalently, 
\begin{equation}\label{eq_dpequal0}
	\mu>0, \quad \gamma = \frac{2}{3} \left( \sqrt{1+3\mu} - 1 \right). 
\end{equation}

Notice that $\{(\mu, \gamma) \mid \mu>0, \gamma = \frac{2}{3} \left( \sqrt{1+3\mu} - 1 \right)\} \subset I$. We substitute (\ref{eq_dpequal0}) into the the first line of (\ref{eq_temp110}) in Theorem 3.1, then use the first line of (\ref{eq_NSwithpresure}) to derive the pressure $p$. Direct computation shows that 
\begin{equation*}
	p(x) = C + f(b) (1-x)^2 + O(1) (1-x)^3, 
\end{equation*} 
where function 
$$
	f(b) = \frac{1}{432} \left(54 b^2 - 22 -  \sqrt{2(3 b^2 - 1)} (15b^2+1) \right). 
$$
It can be checked that 
$$
	f(1) = f'(1)=0; \quad \quad f'(b)<0, \forall b>1; \quad \quad f''(1)<0. 
$$
So 
$f(b)<0$ for all $b>1$. It means that when $p'\mid_{x=1}=0$, the pressure at the center of north pole is greater than the pressure nearby.

When $\mu = - \frac{1}{2}$, the solutions are expressed as
\begin{equation*}
	U_\theta(x) = (1-x)  \left( 1+ \frac{2 \alpha  }{  \alpha \ln \frac{1+x}{2}  - 2 } \right),  
\end{equation*}
and there is $\lim_{x\rightarrow 1^-} u_r =  \alpha - 1$. Similarly, by L'Hospital's rule, we get
\begin{equation*}
	\lim_{x\rightarrow 1^-} p'  =  \lim_{x\rightarrow 1^-} \left(U_\theta''  +  \frac{1}{4} U_\theta'^2 \right) = \frac{1}{2}\alpha^2 +  \frac{1}{4} (\alpha-1)^2. 
\end{equation*}
It is not hard to see that $\lim_{x\rightarrow 1^-} p' >0$ for any $\alpha\in \mathbb{R}$. 

When $\mu < - \frac{1}{2}$, the solution can be exactly expressed as
\begin{equation*}
	U_\theta(x) = (1-x) \left( 1 + \frac{ b (b\tan\frac{\beta(x)}{2} + \alpha) }{\alpha \tan\frac{\beta(x)}{2} - b } \right) ,
\end{equation*}
where $\beta(x)$ is determined by $\beta(x)= b \ln \frac{1+x}{2}$. There is $u_r|_{x=1}=  \alpha - 1$, and
\begin{equation*}
	\lim_{x\rightarrow 1^-}p'  =  U_\theta''  -  \frac{1}{1-x^2}U_\theta U_\theta' - \frac{x}{(1-x^2)^2} U_\theta^2   =  \frac{\alpha^2}{2} + \frac{b^2}{2}  +\frac{1}{4}  (\alpha-1)^2. 
\end{equation*}
It is not hard to see that $p'_x|_{x=1} >0$ for any $\alpha\in \mathbb{R}$. 

According to the above computation, if $\mu\le 0$, the fluid does not fit our pressure profile to support a pingpong ball. In particular, Landau solutions correspond to $\mu=0$, and they have greater pressure in the center. 

Define the open set $I_p\subset I$ by
\begin{equation*}
	I_p:= \{ (\mu,\gamma)\subset \mathbb{R}^2 | \mu>0,  0 < \gamma <  \frac{2}{3}(\sqrt{1+3\mu} - 1) \}. 
\end{equation*} 

\begin{thm}
	For any $(\mu,\gamma)\in I_p$, $u_r|_{x=1} > 0$, $p'|_{x=1} < 0$. For any $(\mu,\gamma)\in \mathbb{R}^2\setminus I_p$, 
	
	\noindent either
	$$
		u_r|_{x=1} \leq 0,
	$$
	or there exists $\delta>0$ such that
	$$
		p(x) < p(1), \mbox{ in } (1-\delta,1). 
	$$
\end{thm}

\begin{rmk}
	We have therefore identified all (-1) homogeneous, axisymmetric, no-swirl solutions of NSE, which describe outward jets with lower pressure in the center. They are $\{u(\mu, \gamma)\mid (\mu,\gamma)\in I_p \}$.

	In particular, those solutions which can not be extended to solutions in $C^{\infty}(\mathbb{S}^2\setminus \{S\})$ are not in this set. There are also many solutions in $C^\infty(\mathbb{S}^2\setminus\{S\})$, including Landau solutions, not in this set. 
\end{rmk}

\section{Asymptotic behavior of certain type of ODE}

In Section 5, we have analyzed several equations of the following form:\\
Let $\delta>0$ and $g\in C^1(-1,-1+\delta]$ be a solution of
\begin{equation}\label{eq_g}
   a(x)g'(x)+b(x)g(x)+\frac{1}{2}g^2(x)=H(x),\quad -1<x<-1+\delta.
\end{equation}
We require $a(x), b(x)\in C(-1,-1+\delta]$ and $a(x)$ satisfy:\\
either (i) $a(x)>0$  for every $ x\in (-1,-1+\delta]$, and $\displaystyle{\lim_{x\to -1^+}\int_{x}^{-1+\delta}\frac{1}{a(x)}=+\infty}$,\\
or (ii) $a(x)<0$  for every $ x\in (-1,-1+\delta]$, and $\displaystyle{\lim_{x\to -1^+}\int_{x}^{-1+\delta}\frac{1}{a(x)}=-\infty}$.

Introduce $H^+(x)=\max\{H(x),0\}$ and $H^-(x)=\max\{-H(x),0\}$, so $H(x)=H^+(x)-H^-(x)$. This is for $b^+(x),b^-(x)$ as well.
\begin{prop}\label{lemA_1}
For $\delta>0$, let $H,a,b\in C(-1,-1+\delta]$ with $b,H^+\in L^{\infty}(-1,-1+\delta)$  and $a(x)$ satisfies (i) or (ii) above. Suppose that $g\in C^1(-1,-1+\delta]$ is a solution of (\ref{eq_g}). Then $g\in L^{\infty}(-1,-1+\delta)$.  If in addition, $\displaystyle{\lim_{x\to -1^+}H(x)}$  is assumed to exist, either finite or infinite,  and $\displaystyle{\lim_{x\to -1^+}b(x)}$ exists and is finite, then $\displaystyle{\lim_{x\to -1^+}g(x)}$ exists and is finite, 

$\displaystyle{\lim_{x\to -1^+}a(x)g'(x)=0}$.
\end{prop}
\begin{lem}\label{lem7_1_A}  For $\delta>0$, let $H,a,b\in C(-1,-1+\delta]$ with $a(x)>0$ for $x\in (-1,-1+\delta)$. Suppose that $g\in C^1(-1,-1+\delta]$ is a solution of (\ref{eq_g}). Then
\[
   g(x)\ge -A_1:= -\max\{4||b^+||_{L^{\infty}(-1,-1+\delta)},\sqrt{8||H^+||_{L^{\infty}(-1,-1+\delta)}}, -g(-1+\delta)\}, \forall x\in (-1,-1+\delta).
\]
\end{lem}
\begin{proof}  
  If $A_1=\infty$, done. So we assume $A_1<\infty$.
  If $g(x)<-A_1$ for some $x\in (-1,-1+\delta)$, we have 
  \[
     a(x)g'(x)=H(x)-\frac{1}{2}g^2(x)-b(x)g(x)\le H(x)-\frac{1}{4}g^2(x)\le -\frac{1}{8}g^2(x)<0.
  \]
  Thus $g'(x)<0$.  This implies, given $g(-1+\delta)\ge -A_1$, that $g\ge -A_1$ on $(-1,-1+\delta)$.
\end{proof}

\begin{lem}\label{lem7_1_B} In addition to the assumption of Lemma \ref{lem7_1_A}, we assume that 
$$
	\displaystyle{\lim_{x\to -1^+}\int_{x}^{-1+\delta}\frac{1}{a(x)}}=+\infty.
$$
Then
\[
   g(x)\le A_2:=\max\{4||b^-||_{L^{\infty}(-1,-1+\delta)}, \sqrt{8||H^+||_{L^{\infty}(-1,-1+\delta)}} \},\quad \forall x\in (-1,-1+\delta).
\]
\end{lem}
\begin{proof}
  If $g(\bar{x})>A_2$ for some $\bar{x}\in (-1,-1+\delta)$, we have 
  \[
    a(\bar{x})g'(\bar{x})=H(\bar{x})-\frac{1}{2}g^2(\bar{x})-b(\bar{x})g(\bar{x})\le H(\bar{x})-\frac{1}{4}g^2(\bar{x})\le -\frac{1}{8}g^2(\bar{x})<0.
  \]
  Thus $g'(\bar{x})<0$, and therefore for some $\epsilon>0$, $g(x)>g(\bar{x})>A_2$ for $\bar{x}-\epsilon<x<\bar{x}$. It follows that $g(x)>A_2$ for all $x\in (-1,\bar{x})$. Thus as shown above, $a(x)g'(x)<-\frac{1}{8}g^2(x)$ for all $-1<x<\bar{x}$. It follows that $(g^{-1})'(x)\ge \frac{1}{8a(x)}$ and
  \[
      \frac{1}{8}\int_{x}^{\bar{x}}\frac{ds}{a(s)} \le  g^{-1}(\bar{x})-g^{-1}(x)\le g^{-1}(\bar{x})\le  \frac{1}{A_2}, \quad  \forall -1<x<\bar{x}.
  \]
  This violates $\int_{x}^{-1+\delta} \frac{ds}{a(s)}=\infty$, a contradiction.
\end{proof}

\begin{lem}\label{lem7_1_A1} For $\delta>0$, let $H,a,b\in C(-1,-1+\delta]$ with $a(x)<0$ for $x\in (-1,-1+\delta)$. Suppose that $g\in C^1(-1,-1+\delta)$ is a solution of (\ref{eq_g}). Then
\[
   g(x)\le \hat{A}_1:=\max\{4||b^-||_{L^{\infty}(-1,-1+\delta)}, \sqrt{8||H^+||_{L^{\infty}(-1,-1+\delta)}}, g(-1+\delta)\},\quad \forall x\in (-1,-1+\delta).
\]
\end{lem}
\begin{proof}
Rewriting (\ref{eq_g}) as
\begin{equation}\label{eq_g1}
   (-a)(-g)'+(-b)(-g)+\frac{1}{2}(-g)^2=H.
\end{equation}
The conclusion follows from Lemma \ref{lem7_1_A} with $a, b$ and $g$ there replaced by $-a,-b$ and $-g$.
\end{proof}

\begin{lem}\label{lem7_1_B1} In addition to the assumption of Lemma \ref{lem7_1_A1}, we assume that 
$$
\displaystyle{\lim_{x\to -1^+}\int_{x}^{-1+\delta}\frac{1}{a(x)}}=-\infty.
$$
Then
\[
   g(x)\ge -\hat{A}_2 : =-\max\{4||b^+||_{L^{\infty}(-1,-1+\delta)}, \sqrt{8||H^+||_{L^{\infty}(-1,-1+\delta)}}\},\quad \forall x\in (-1,-1+\delta).
\]
\end{lem}
\begin{proof}
This follows from Lemma \ref{lem7_1_B} as the way Lemma \ref{lem7_1_A1} being deduced from Lemma \ref{lem7_1_A}.
\end{proof}

\begin{lem}\label{lem_M1}
  For $\delta>0$, let $b\in C^0(-1,-1+\delta]\cap L^{\infty}(-1,-1+\delta)$, $H\in C^0(-1,-1+\delta]$, and let $a\in  C^0(-1,-1+\delta]$ be either positive or negative in the interval and satisfies $\displaystyle{\lim_{x\to -1^+}\left|\int_{x}^{-1+\delta}\frac{ds}{a(s)}\right|=\infty}$. Assume that  $g\in C^1(-1,-1+\delta]$ is a solution of (\ref{eq_g}).  Then
  \[
    \lambda:= \sup_{-1<x\le -1+\delta}\left(H(x)+\frac{1}{2}(b(x))^2\right)\ge 0.
  \]
\end{lem}
\begin{proof}
   We only need to treat the case that $a(x)>0$ since the other case can be converted to this case by rewriting (\ref{eq_g}) as (\ref{eq_g1}).  We prove it by contradiction. If not, then
   \[
      a(x)g'(x)=H(x)-\frac{1}{2}b(x)^2-\frac{1}{2}(g(x)+b(x))^2\le \lambda<0, \quad \forall -1<x<-1+\delta.
   \]
   It follows that 
      \[
      g(-1+\delta)-g(x)=\int_{x}^{-1+\delta}g'(s)ds\le \lambda \int_{x}^{-1+\delta} \frac{ds}{a(x)}\to -\infty \textrm{ as } x\to -1^+.
   \]
   This implies
   \begin{equation}\label{eqM_1}
       \lim_{x\to -1^+}g(x)=+\infty.
   \end{equation}
   On the other hand, $\lambda$ being negative implies that $H^+\in L^{\infty}(-1,-1+\delta)$. An application of  Lemma \ref{lem7_1_B} gives that $g^+\in L^{\infty}(-1,-1+\delta)$, violating  (\ref{eqM_1}). 
\end{proof}

\begin{lem}\label{lemA_2} 
For $\delta>0$, let $b\in C^0[-1,-1+\delta]$ and $H\in C^0(-1,-1+\delta]$ such that $\displaystyle{\lim_{x\to -1^+}H(x)}$ exists, is either finite or infinite, and let $a\in C^0(-1,-1+\delta]$ be either positive or negative in the interval. Assume that $g\in C^1(-1,-1+\delta]$ is a solution of (\ref{eq_g}). Then $\displaystyle{\lim_{x\to -1^+}g(x)}$ exists and $b(-1)g(-1)+\frac{1}{2}g(-1)^2=H(-1)$.

If in addition, $\displaystyle{\lim_{x\to -1^+}\left|\int_{x}^{-1+\delta}\frac{ds}{a(s)}\right|=\infty}$, then $\displaystyle{\lim_{x\to -1^+}g(x)}$ is finite if and only if $\displaystyle{\lim_{x\to -1^+}H(x)}$ is finite, and in this case $\displaystyle{\lim_{x\to -1^+}a(x)g'(x)=0}$.
\end{lem}

\begin{proof}
As before, we will only prove it when $a>0$, since the $a<0$ case follows after rewriting (\ref{eq_g}) as (\ref{eq_g1}). We prove it by contradiction.

Assume that $\displaystyle{\lim_{x\to -1^+}g(x)}$ does not exist, then there exist $-\infty<\alpha_1<\alpha_2<\infty$ and two sequences $\{x_i\}$ and $\{y_i\}$ such that $ x_1>y_1>x_2>y_2>\cdots>-1$, $\displaystyle{\lim_{i\to\infty}x_i=\lim_{i\to\infty}y_i=-1}$, $ g(x_i)=\alpha_1$ and $g(y_i)=\alpha_2$. Then for any $\alpha \in (\alpha_1, \alpha_2)$, there exists a $x_i>z_i>y_i$  such that $g(z_i)=\alpha$ and $g(z)<\alpha, \forall x_i\ge z>z_i$. Clearly $\displaystyle{\lim_{i\to \infty}z_i=1}$ and $g'(z_i)\le 0$. This leads to, in view of (\ref{eq_g}), $b(z_i)g(z_i)+\frac{1}{2}g^2(z_i)\ge H(z_i)$. Sending $i\to\infty$, we have $b(-1)\alpha+\frac{1}{2}\alpha^2\ge \displaystyle{\lim_{x\to -1^+}}H(x)$.

Similarly, we can find $y_i>\hat{z_i}>x_{i+1}$ satisfying $g(\hat{z_i})=\alpha$ and $g'(\hat{z_i})\ge 0$, which leads to $b(-1)\alpha+\frac{1}{2}\alpha^2\le\displaystyle{\lim_{x\to -1^+}}H(x)$. 
So for any $\alpha\in (\alpha_1, \alpha_2)$, $b(-1)\alpha+\frac{1}{2}\alpha^2=\displaystyle{\lim_{x\to -1^+}}H(x)$. Contradiction. We have proved that $\displaystyle{\lim_{x\to -1^+}g(x)}$ exists, either finite or infinite.

If $\displaystyle{\lim_{x\to -1^+}}H(x)$ is finite, then, in view of  Lemma \ref{lem7_1_A} and Lemma \ref{lem7_1_B}, $\displaystyle{\lim_{x\to -1^+}g(x)}$ is finite. 

If $\displaystyle{\lim_{x\to -1^+}H(x)}$ is infinite, then, in view of Lemma \ref{lem_M1},  $\displaystyle{\lim_{x\to -1^+}H(x)}=+\infty$.  We will show by contradiction that $\displaystyle{\lim_{x\to -1^+}}g$ is infinite. Suppose that the limit is finite, then $a(x)g'(x)=H(x)-b(x)g(x)-\frac{1}{2}g^2(x)\to +\infty$ as $x\to -1^+$. It follows that there exists $0<\epsilon<\delta$, such that $ g'(x)\ge \displaystyle{\frac{1}{a(x)}}$, for $-1<x<-1+\epsilon$. It follows that 
\[
   g(-1+\epsilon)-g(x)\ge \int_{x}^{-1+\epsilon}\frac{ds}{a(s)}\to \infty \textrm{ as }x\to -1^+,
\]
a contradiction to the finiteness of $\displaystyle{\lim_{x\to -1^+}}g(x)$.

We have proved that $\displaystyle{\lim_{x\to -1^+}}g(x)$ is finite if and only if $\displaystyle{\lim_{x\to -1^+}}H(x)$ is finite.\\

If $\displaystyle{\lim_{x\to -1^+}g}$ is finite, we see by sending $x$ to $-1^+$ in (\ref{eq_g})  that $\displaystyle{\lim_{x\to -1^+}}a(x)g'=\mu$ for some $\mu\in \mathbb{R}$. Since $g$ is bounded, $\mu=0$.  Indeed, if $\mu\ne 0$, we would have
\[
   \frac{2g'(x)}{\mu}\ge \frac{1}{a(x)} 
\]
for $x$ close to $-1$, and an argument above would lead to a contradiction to the boundedness of $g$.
\end{proof} 

Proposition \ref{lemA_1} follows from Lemma \ref{lem7_1_A}-\ref{lem7_1_B1} and Lemma \ref{lemA_2}.

Next, we study asymptotic behavior of  solution $V\in C^1(-1, -1+\delta]$  of
\begin{equation}\label{eqA_3_1}
V'+BV=H\qquad \mbox{in}\ (-1, -1+\delta)
\end{equation}
under various hypothesis on $B$ and $H$.

Let $w:=\int_{-1+\delta}^{x}B(s)ds$, then $V$ can be expressed as 
\begin{equation}\label{eqA_3_2}
		V(x)=V(x_0)e^{w(x_0)-w(x)}+e^{-w(x)}\int_{x_0}^{x}e^{w(s)}H(s)ds,
\end{equation}
for every $x_0\in (-1,-1+\delta]$.

\begin{lem}\label{lemA_3}
 For $\delta>0$, $0\le b\le 1$ and $\beta\ge 0$, let $B, H\in C(-1, -1+\delta]$ satisfy
\begin{equation}\label{eqA_3_4}
\inf_{-1<x\le-1+\delta}(1+x)^{1-b}H(x)>-\infty,
\end{equation}
and
\begin{equation}\label{eqA_3_5}
\lim_{ x\to -1^+} (1+x)B(x)=-\beta.
\end{equation} 
Assume that $V\in  C^1(-1, -1+\delta]$  and satisfies (\ref{eqA_3_1}).  Then for every  $\epsilon>0$,  there exists some constant $C$, such that
\begin{equation}\label{eqA_3_6}
V(x)\le 
C(1+x)^{   \min\{ b, \beta\}  -\epsilon}, \quad \textrm{for all } -1<x\le -1+\delta.
\end{equation}
\end{lem}

\begin{proof}
By (\ref{eqA_3_5})
\begin{equation}\label{eqA_3_3}
  w(x)=(-\beta+o(1))\ln (1+x),
\end{equation}
where $o(1)$ denotes some quantity which tends to $0$ as $x\to -1^+$.
  
Since $V\in  C^1(-1, -1+\delta]$  is a solution of  (\ref{eqA_3_1}), (\ref{eqA_3_2}) holds for every $x_0\in (-1,-1+\delta]$. It follows from  (\ref{eqA_3_3}), (\ref{eqA_3_4}) and (\ref{eqA_3_2}) , with $x_0=-1+\delta$, that 
\begin{equation*}
  \begin{split}
     V(x) & \le (1+x)^{\beta+o(1)}+(1+x)^{\beta+o(1)}\int_{x}^{-1+\delta}(1+s)^{-\beta+b-1+o(1)}ds\\
               & \le (1+x)^{\beta+o(1)}+(1+x)^{b+o(1)}\le C(1+x)^{\min\{b,\beta\}-\epsilon}. 
  \end{split}
\end{equation*} 
\end{proof}

\begin{rmk}
  In Lemma \ref{lemA_3}, if we replace (\ref{eqA_3_4}) by 
  \begin{equation}\label{eqA_3_7}
     \sup_{-1<x\le-1+\delta}(1+x)^{1-b}|H(x)|<\infty, 
  \end{equation}
  then we have, instead of (\ref{eqA_3_6}), for any $\epsilon>0$, 
  \[
     |V(x)|\le 
C(1+x)^{   \min\{ b, \beta\}  -\epsilon}, \quad \textrm{for all } -1<x\le -1+\delta 
  \]
  instead of (\ref{eqA_3_6})
\end{rmk}

\begin{lem}\label{lemA_4}
 For $\delta>0$, $0<b\le 1$ and $\beta<0$, let $B, H\in C(-1, -1+\delta]$ satisfy (\ref{eqA_3_4}) and  (\ref{eqA_3_5}). Assume that $V\in  C^1(-1, -1+\delta]$  and satisfies (\ref{eqA_3_1}) and 
\begin{equation}\label{eqA_4_1}
\limsup_{x\to -1^+}V(x)e^{\int_{-1+\delta}^{x}B(s)ds}\ge 0.
\end{equation}
Then for every  $\epsilon>0$,  there exists some constant $C$, such that
\begin{equation}\label{eqA_4_3}
-V(x)\le C(1+x)^{b-\epsilon}, \textrm{for all } -1<x\le -1+\delta.
\end{equation}
\end{lem}
\begin{proof}
 Estimate (\ref{eqA_3_3})  still holds by the assumption of $B$. For all $-1<x_0<x$, we obtain from (\ref{eqA_3_2}) and (\ref{eqA_3_4}) that
 \[
   V(x)\ge V(x_0)e^{w(x_0)-w(x)}-Ce^{-w(x)}\int_{x_0}^{x}e^{w(s)}(1+s)^{b-1}ds.
 \]
 Sending $x_0 \to -1$ along a subsequence in (\ref{eqA_3_2}), we have, in view of (\ref{eqA_4_1}) 
 \[
   -V(x)\le  Ce^{-w(x)}\int_{-1}^{x}e^{w(s)}(1+s)^{b-1}ds.
 \]
 By (\ref{eqA_3_3}), for every $\epsilon>0$, there exists some constant $C$, such that  
 \[
    -V(x)  \le  (1+x)^{\beta+o(1)}\int_{-1}^{x}(1+s)^{-\beta+b-1+o(1)}ds \le C(1+x)^{b-\epsilon}.
 \]
\end{proof}

\begin{rmk}\label{rmkA_4}
  In Lemma \ref{lemA_4}, if we replace (\ref{eqA_3_4}) and (\ref{eqA_4_1}) respectively by (\ref{eqA_3_7}) and 
  \begin{equation}\label{eqA_4_2}
     \lim_{x\to -1^+}V(x)e^{\int_{-1+\delta}^{x}B(s)ds}=0,
  \end{equation}
  then we have, instead of (\ref{eqA_4_3}), that for any $\epsilon>0$, 
  \[
     |V(x)|\le C(1+x)^{b-\epsilon}, \quad \textrm{for all } -1<x\le -1+\delta.
  \]
\end{rmk}

\begin{lem}\label{lemA_5}
For $\delta, \beta,c_1,c_2> 0$ , let $B\in C(-1, -1+\delta]$ and $H\in C[-1,-1+\delta]$ satisfy
\begin{equation}\label{eqA_5_1}
H(x)=H(-1)+O((1+x)^{c_1}),\quad -1< x\le -1+\delta,
\end{equation}
and
\begin{equation}\label{eqA_5_2}
           (1+x)B(x)+\beta=O((1+x)^{c_2}).
\end{equation}

Assume that $V\in  C^1(-1, -1+\delta]$  and satisfies (\ref{eqA_3_1}).  Then  there exists some constant $a_1$, such that
for  every $0<\alpha<\min\{c_2+\beta, c_2+1, c_1+1\}$, 
\[  
V(x)=a_1(1+x)^{\beta}+\left\{
    \begin{split}
        & \frac{H(-1)}{1-\beta}(1+x) \textrm{ if }\beta\ne 1\\
        & H(-1)(1+x)\ln(1+x)\textrm{ if }\beta= 1\
     \end{split}
    \right.
    +O((1+x)^{\alpha}), \quad -1<x\le -1+\delta.
\]
\end{lem}

\begin{proof}
 Since $V$ is a solution of (\ref{eqA_3_1}), (\ref{eqA_3_2}) holds. By (\ref{eqA_5_2}), we have, for some $a_3\in \mathbb{R}$, 
 \begin{equation}\label{eqA_5_2a}
    w(x)=-\beta \ln(1+x)+a_3+O((1+x)^{c_2}).
 \end{equation}
 We derive from (\ref{eqA_3_2}), using (\ref{eqA_5_1}) and the above that for some constant $a_1\in\mathbb{R}$,
 \[
  \begin{split}
    V(x) & =V(x_0)e^{w(x_0)-w(x)}+e^{-w(x)}\int_{x_0}^{x}e^{w(s)}H(s)ds\\
           & =V(x_0)e^{w(x_0)}e^{-a_3}(1+x)^{\beta}(1+O((1+x)^{c_2}))\\
              &+(1+x)^{\beta}(1+O((1+x)^{c_2}))\int_{x_0}^{x}(1+s)^{-\beta}(H(-1)+O((1+s)^{\min\{c_1,c_2\}}))ds,
   \end{split}
 \]
 from which we conclude the proof.
\end{proof}

\begin{lem}\label{lemA_6}
For $\delta, c_1,c_2>0$,  $\beta<0$ , $0<b<1$ and $\gamma_1,\gamma_2\in \mathbb{R}$, let $B, H\in C(-1, -1+\delta]$ satisfy (\ref{eqA_5_2}) and 
\begin{equation}\label{eqA_6_1}
H(x)=\gamma_1 (1+x)^{b-1}+\gamma_2+O((1+x)^{b-1+c_1}),\quad -1< x\le -1+\delta.
\end{equation}

Assume that $V\in  C^1(-1, -1+\delta]$  and satisfies (\ref{eqA_3_1}) and $V(x)=o((1+x)^{\beta})$. 
Then  
\[  
V(x)=\frac{\gamma_1}{b-\beta}(1+x)^b+\frac{\gamma_2}{1-\beta}(1+x)+O((1+x)^{b+\min\{c_1,c_2\}}).
\]
\end{lem}
\begin{proof}
 Expression (\ref{eqA_3_2}) still holds.  By (\ref{eqA_5_2}), we have (\ref{eqA_5_2a}) for some $a_3\in \mathbb{R}$.  Since $V(x)=o((1+x)^{\beta})$, we obtain, by sending $x_0$ to $-1$ in (\ref{eqA_3_2}) similar to the arguments in the proof of  Lemma \ref{lemA_4},  that
 \[
   V(x)= e^{-w(x)}\int_{-1}^{x}e^{w(s)}H(s)ds.
 \]
We derive from the above using (\ref{eqA_5_2a}) and (\ref{eqA_6_1}) that 
 \[
   \begin{split}
     V(x)  & =(1+x)^{\beta}(1+O((1+x)^{c_2}))\int_{-1}^{x}(1+s)^{-\beta}(\gamma_1 (1+x)^{b-1}+\gamma_2+O((1+s)^{b-1+\min\{c_1,c_2\}}))ds\\
            & =\frac{\gamma_1}{b-\beta}(1+x)^b+\frac{\gamma_2}{1-\beta}(1+x)+O((1+x)^{b+\min\{c_1,c_2\}}).
     \end{split}
 \]
\end{proof}


\begin{lem}\label{lemA_7}
For $\delta>0$, $0<b\le 1$ and  $\beta\le 0$, let $B, H\in C(-1, -1+\delta]$ satisfy (\ref{eqA_3_4}) and
\begin{equation}\label{eqA_7_2}
\lim_{ x\to -1^+} [(1+x)\ln(1+x)]B(x)=-\beta.
\end{equation}
Assume that $V\in  C^1(-1, -1+\delta]$  satisfies (\ref{eqA_3_1}). When $\beta=0$, we also assume (\ref{eqA_4_1}). Then for every  $\epsilon>0$,  there exists some constant $C$, such that
for all $-1<x\le -1+\delta$,
\begin{equation}\label{eqA_7_4}
\left\{
 \begin{array}{ll}
 V(x)\le C(\ln(1+x))^{\beta+\epsilon} & \textrm{ if }\beta<0,\\
 V(x) \ge -C(1+x)^{b}|\ln(1+x)|^{\epsilon} & \textrm{ if }\beta=0.
\end{array}
\right.
\end{equation}
\end{lem}
\begin{proof}
  By (\ref{eqA_7_2}) , 
\begin{equation*}
  w(x)=(-\beta+o(1))\ln(-\ln (1+x)). 
\end{equation*}
Expression (\ref{eqA_3_2}) still holds for all  $x_0\in(-1,-1+\delta]$. If $\beta<0$,  take $x_0=-1+\delta$, 
\begin{equation*}
  \begin{split}
		V(x)& =V(x_0)e^{w(x_0)-w(x)}+e^{-w(x)}\int_{x_0}^{x}e^{w(s)}H(s)ds\\
		       & \le |\ln(1+x)|^{\beta+o(1)}+|\ln(1+x)|^{\beta+o(1)}\int_{x_0}^{x}(\ln(1+s))^{-\beta+o(1)}(1+s)^{b-1+o(1)}ds\\
		       &\le |\ln(1+x)|^{\beta+o(1)}.
\end{split}
\end{equation*}
  
If $\beta=0$, $w=o(1)\ln(-\ln (1+x))$.  By (\ref{eqA_4_1}),  similar as in the proof of Lemma \ref{lemA_4}, 
sending $x_0$ to $-1$ along a subsequence in (\ref{eqA_3_2}) gives
\[
\begin{split}
   V(x) & \ge  -Ce^{-w(x)}\int_{-1}^{x}e^{w(s)}(1+x)^{b-1}ds \ge -C(|\ln(1+x)|)^{o(1)}\int_{x_0}^{x}(|\ln(1+s)|)^{o(1)}(1+s)^{b-1}ds\\
          &\ge -C(1+x)^{b}|\ln(1+x)|^{\epsilon}.
\end{split}
\]
\end{proof}

\begin{rmk}
If in Lemma \ref{lemA_7}, we replace (\ref{eqA_3_4}) and (\ref{eqA_4_1}) by (\ref{eqA_3_7})  and (\ref{eqA_4_2})  respectively,  then we have, instead of (\ref{eqA_7_4}),  that for any $\epsilon>0$,
  \[
     |V(x)|\le  \left\{
 \begin{array}{ll}
  C(\ln(1+x))^{\beta+\epsilon} & \textrm{ if }\beta<0,\\
  C(1+x)^{b}|\ln(1+x)|^{\epsilon} & \textrm{ if }\beta=0.
\end{array}
\right.
  \]
\end{rmk}

\section{Figures}\label{sec_fig}

For a given axisymmetric vector fields ($u_r$, $u_\theta$), the stream lines can be represented in the cross section plane $x_1=0$. The shape of stream lines, along with the graph of ($u_r$, $u_\theta$), depends on parameters $(\mu, \gamma)$. In this section, we choose some typical points on the $(\mu, \gamma)$ plane, whose positions are shown in the left part of Figure \ref{figure3_1}. At each parameter point, we present the graph of $u_r$, $u_\theta$, and the corresponding stream lines. In stead of presenting a full classification of all possible shapes of the stream lines, we prefer to emphasize that four border lines play important roles to determine the shape of stream lines. 

1) The line $l_1: \gamma = 0$ separates the stream lines which are upward and downward along positive $x_3$ axis near the north pole.

2) The line $l_2: \mu = 0$, ($\gamma>-2$) separates the stream lines which are inward and outward to negative $x_3$ axis near the south pole. 

3) The line $l_3: \gamma = - 1+ \sqrt{1+2\mu}$, ($- \frac{1}{2} < \mu < 0$) separates the stream lines which are upward and downward along negative $x_3$ axis near the south pole. 

4) The line $l_4: \mu = -\frac{3}{8}$ separates the stream lines by the amplitude of $u_r$ and $u_\theta$. Namely, on the left of $l_4$, $u_r$ dominates, thus the stream line near south pole is vertical. While on the right of $l_4$, $u_\theta$ dominates, thus the stream line near south pole is horizontal. 

\vspace{0.5cm}

\FloatBarrier

\begin{figure}[!htb]
	\centering
	\includegraphics[height=6cm]{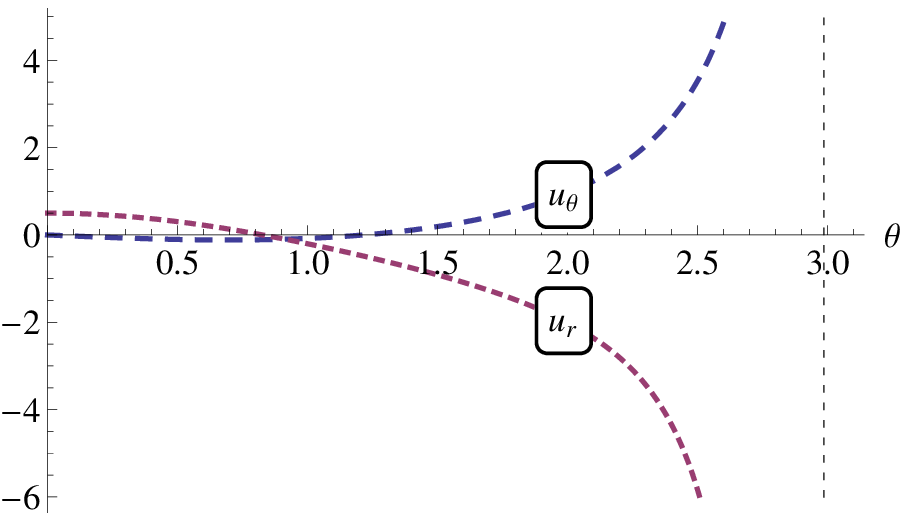}
	\includegraphics[height=6cm]{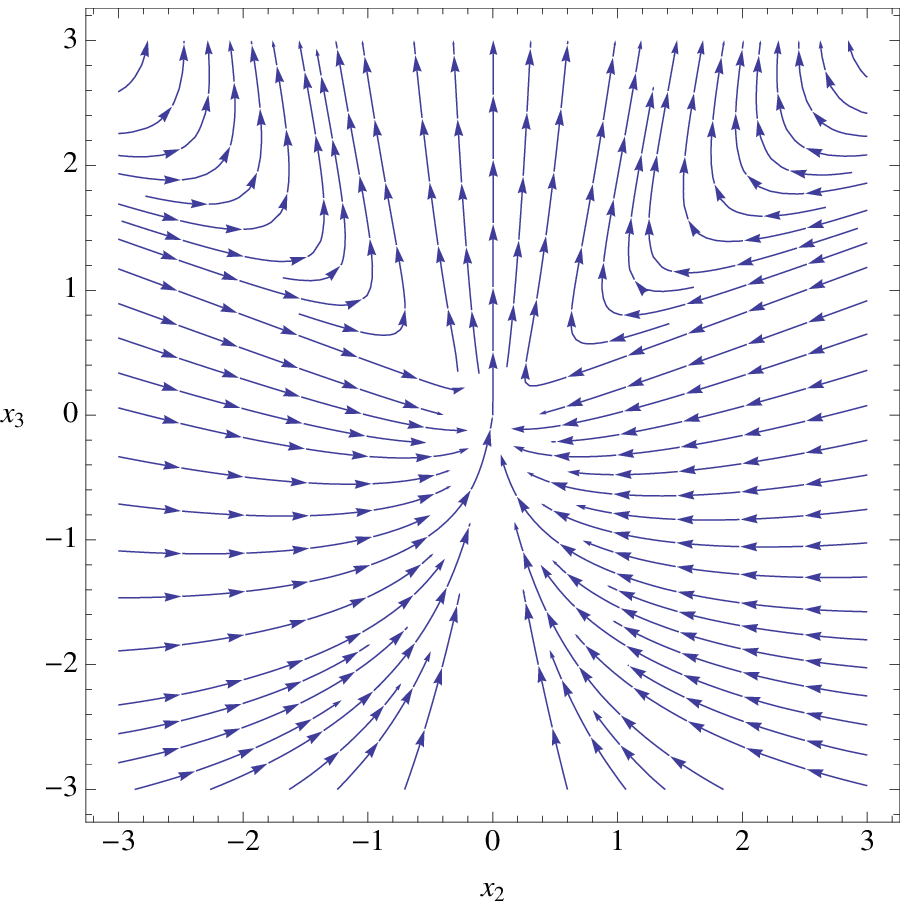}
	\label{fig_pt1}
	\caption{The graphs of $u_\theta$, $u_r$ and stream lines for $P_1$: $\mu = - 1$, $\gamma=\frac{1}{2}$.}
\end{figure}

\begin{figure}[!htb]
	\centering
	\includegraphics[height=6cm]{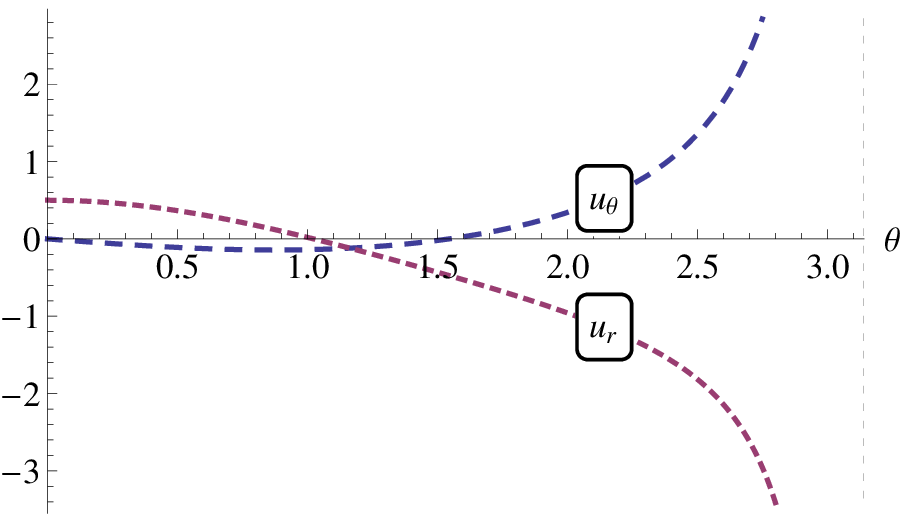}
	\includegraphics[height=6cm]{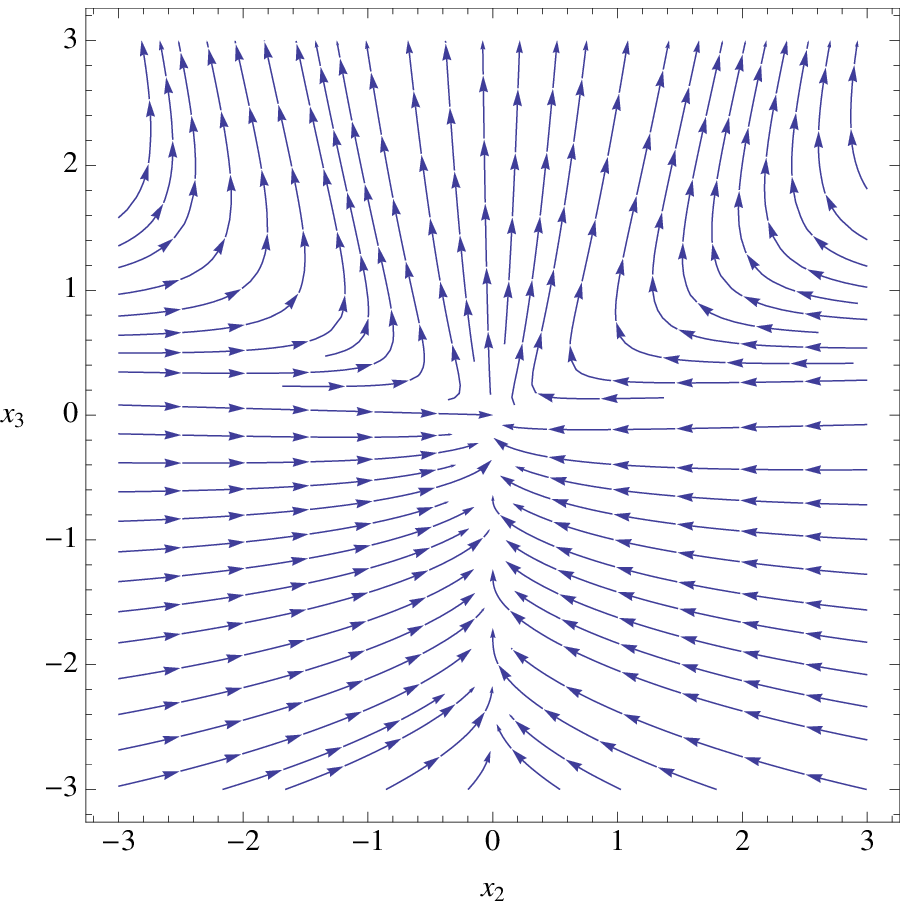}
	\label{fig_pt2}
	\caption{The graphs of $u_\theta$, $u_r$ and stream lines for $P_2$: $\mu = - \frac{1}{2}$, $\gamma=\frac{1}{2}$.}
\end{figure}

\begin{figure}[!htb]
	\centering
	\includegraphics[height=6cm]{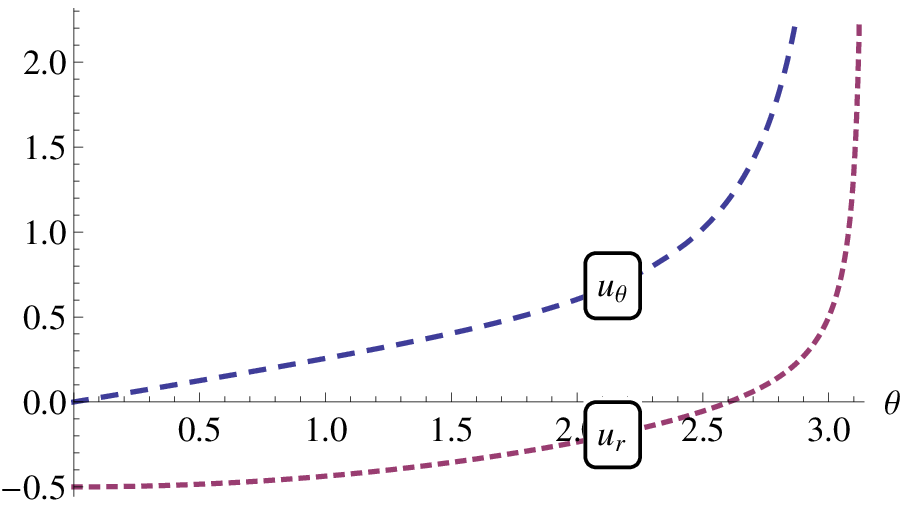}
	\includegraphics[height=6cm]{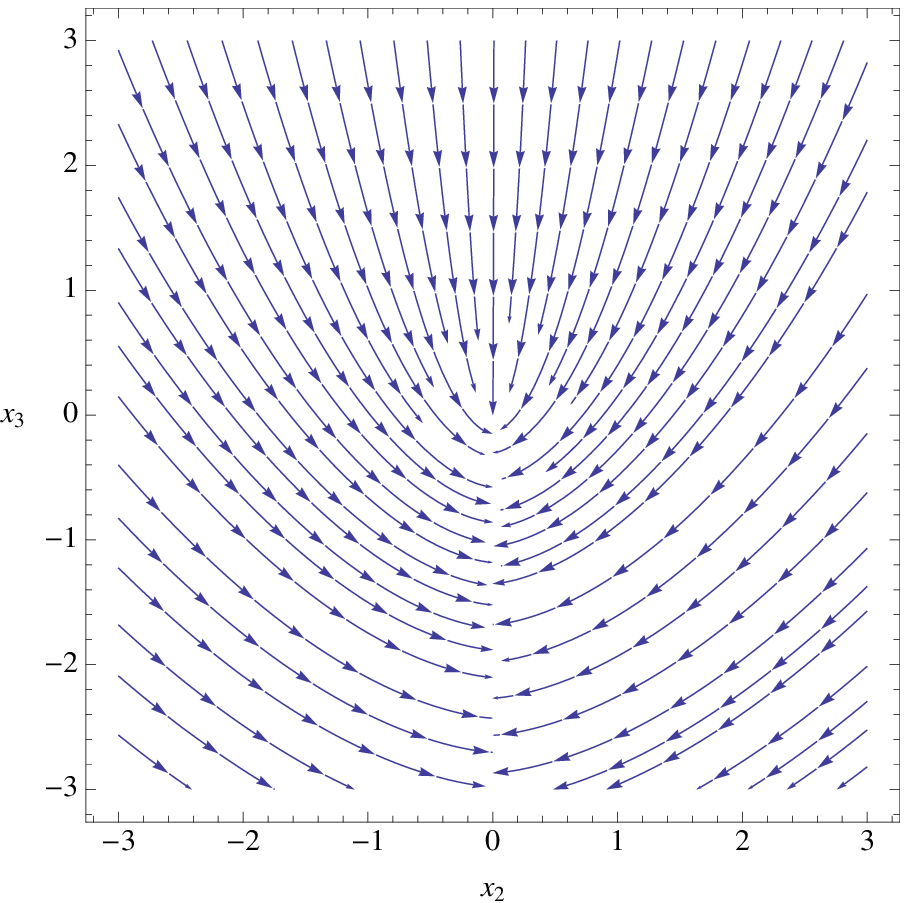}
	\label{fig_pt3}
	\caption{The graphs of $u_\theta$, $u_r$ and stream lines for $P_3$: $\mu = -\frac{1}{4}$, $\gamma=-\frac{1}{2}$.}
\end{figure}

\begin{figure}[!htb]
	\centering
	\includegraphics[height=6cm]{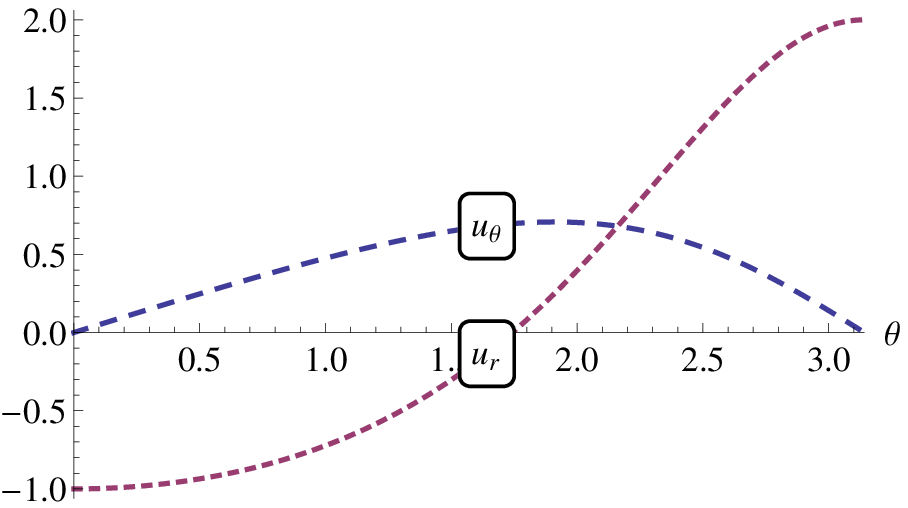}
	\includegraphics[height=6cm]{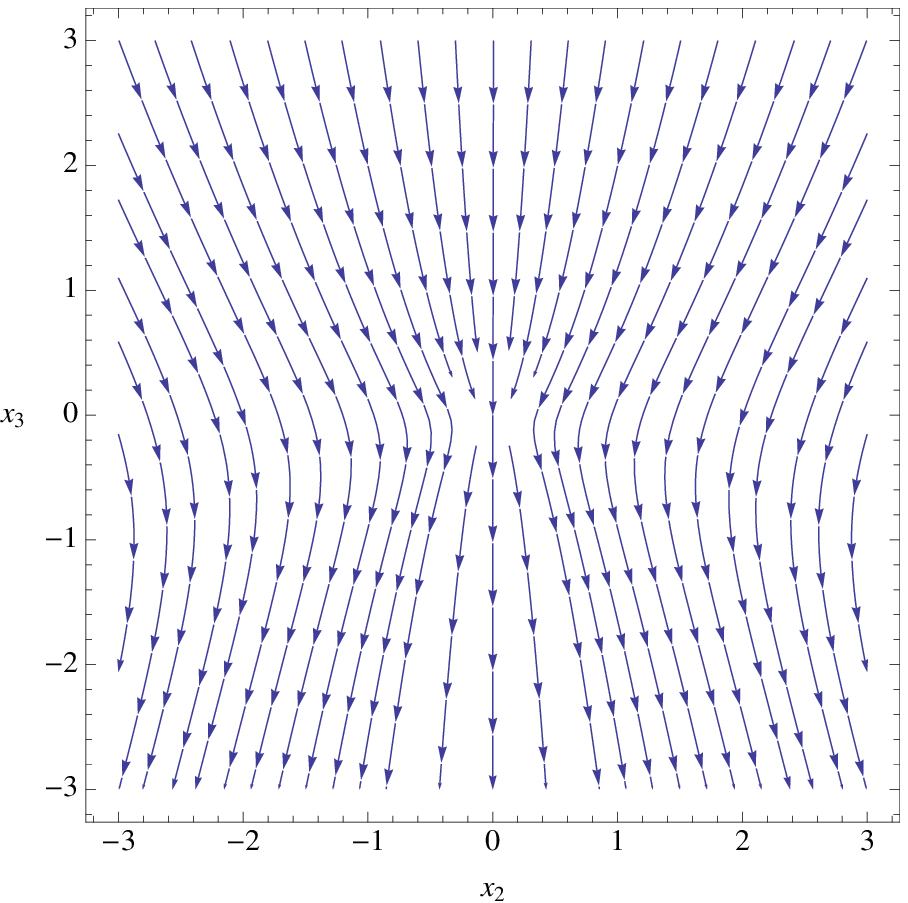}
	\label{fig_pt4}
	\caption{The graphs of $u_\theta$, $u_r$ and stream lines for $P_4$: $\mu = 0$, $\gamma=-1$.}
\end{figure}

\begin{figure}[!htb]
	\centering
	\includegraphics[height=6cm]{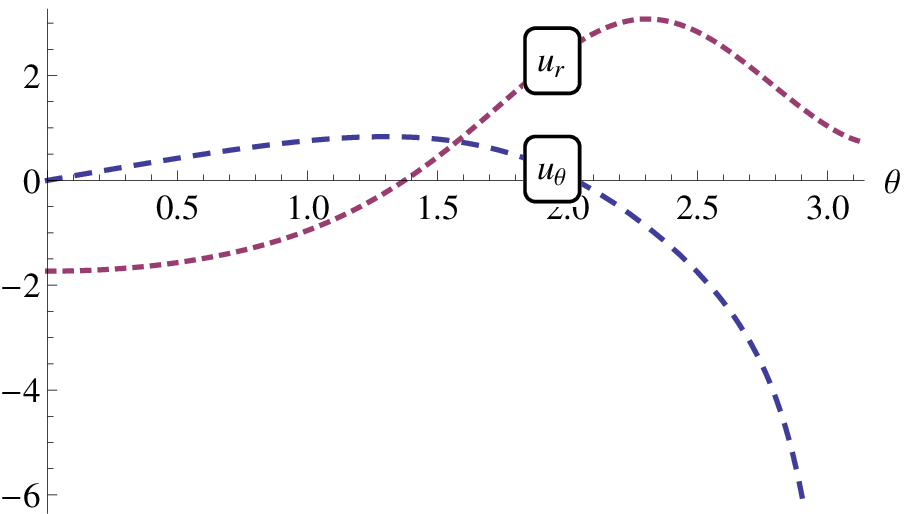}
	\includegraphics[height=6cm]{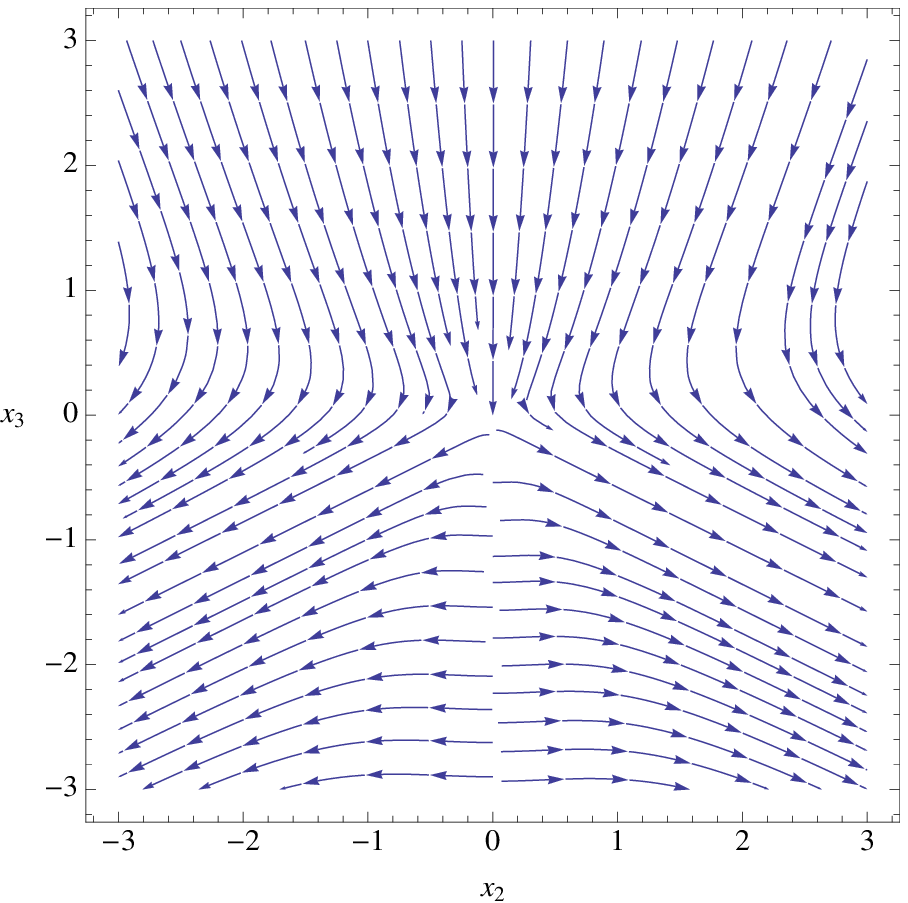}
	\label{fig_pt5}
	\caption{The graphs of $u_\theta$, $u_r$ and stream lines for $P_5$: $\mu = 1$, $\gamma=-\sqrt{3}$.}
\end{figure}

\begin{figure}[t]
	\centering
	\includegraphics[height=6cm]{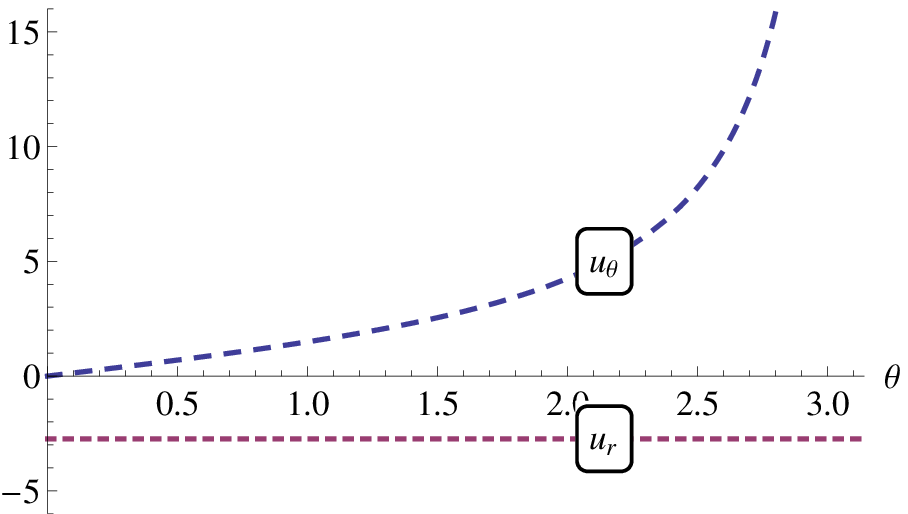}
	\includegraphics[height=6cm]{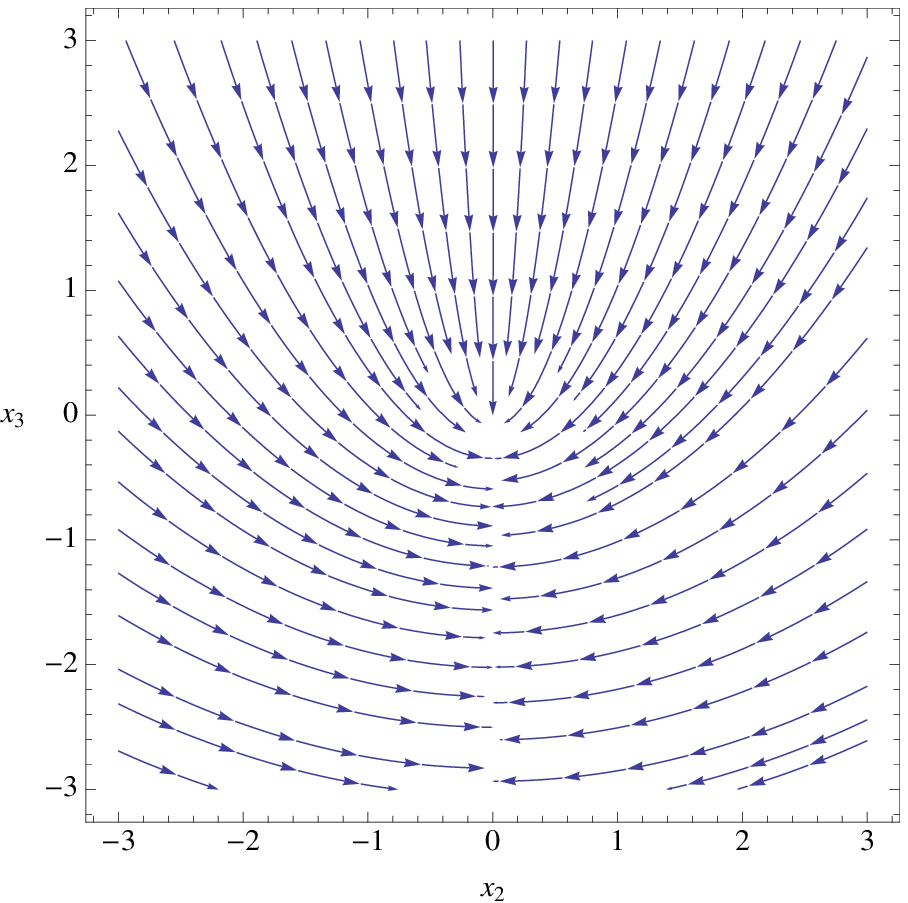}
	\label{fig_pt6}
	\caption{The graphs of $u_\theta$, $u_r$ and stream lines for $P_6$: $\mu = 1$, $\gamma=-1-\sqrt{3}$.}
\end{figure}

\begin{figure}[!htbp]
	\centering
	\includegraphics[height=6cm]{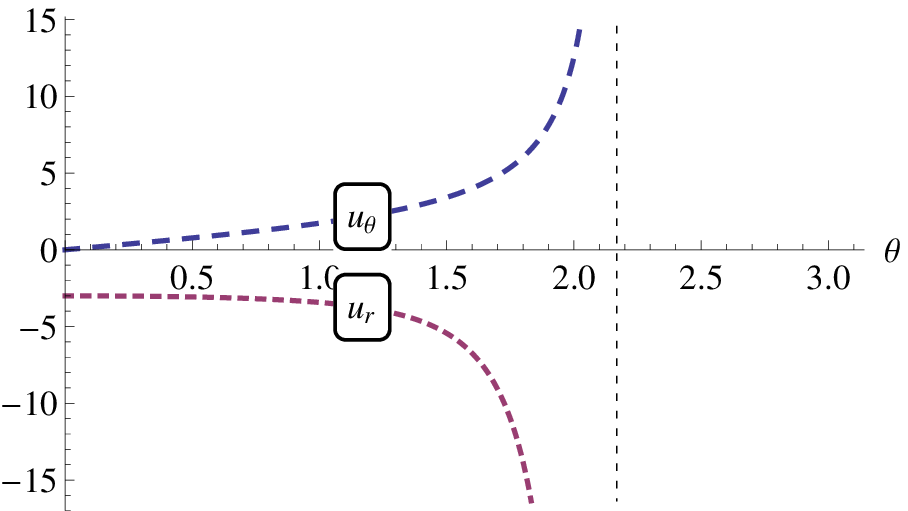}     
	\includegraphics[height=6cm]{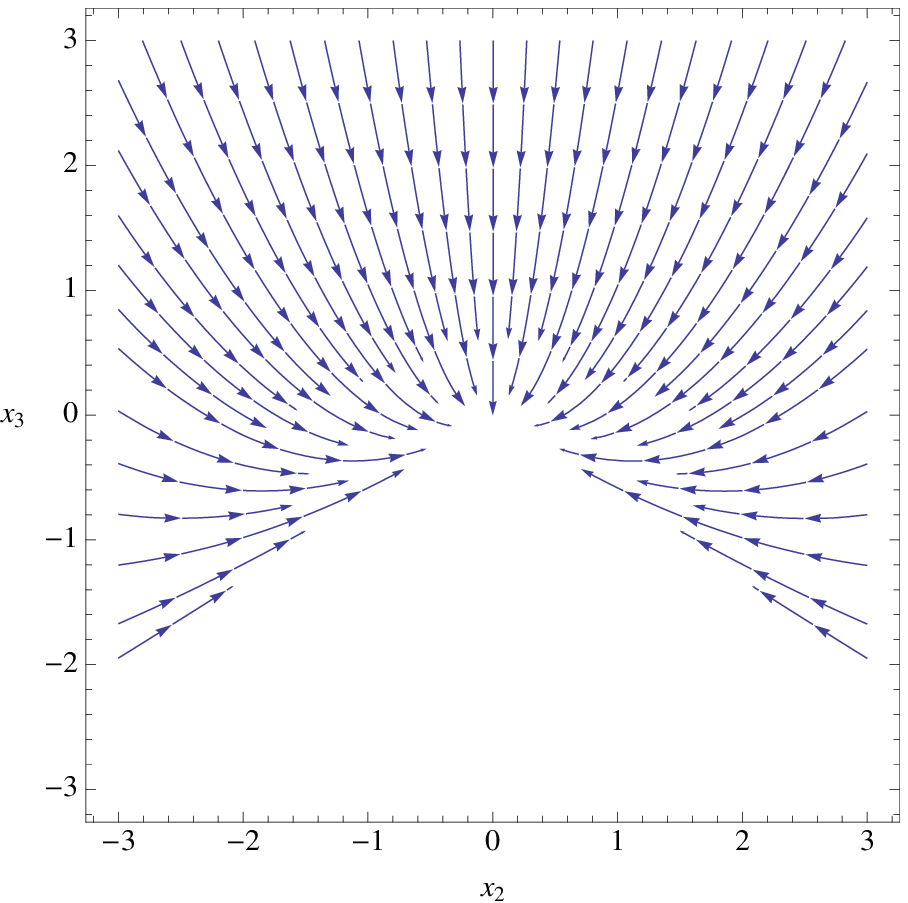}
	\label{fig_pt7}
	\caption{The graphs of $u_\theta$, $u_r$ and stream lines for $P_7$: $\mu = 1$, $\gamma=-3$.}
\end{figure}

\FloatBarrier

\end{document}